\documentclass[pdflatex,bst/sn-mathphys-num]{sn-jnl}

\usepackage{lmodern}
\usepackage{graphicx}%
\usepackage{multirow}%
\usepackage{amsmath,amssymb,amsfonts}%
\usepackage{amsthm}%
\usepackage{mathrsfs}%
\usepackage[title]{appendix}%
\usepackage{xcolor}%
\usepackage{textcomp}%
\usepackage{manyfoot}%
\usepackage{booktabs}%
\usepackage{algorithm}%
\usepackage{algorithmicx}%
\usepackage{algpseudocode}%
\usepackage{listings}%
\usepackage{soul} 

\usepackage{subcaption}

\usepackage{natbib} 
\newcommand{\posR}{\overline{\mathbb{R}}_+}
\newcommand{\RR}{\mathbb{R}}

\theoremstyle{thmstyleone}%
\newtheorem{theorem}{Theorem}
\newtheorem{prop}[theorem]{Proposition}%
\newtheorem{lem}[theorem]{Lemma}%
\theoremstyle{definition}

\newtheorem{cor}[theorem]{Corollary}

\theoremstyle{thmstyletwo}%
\newcommand{\FD}{\mathfrak{D}}

\theoremstyle{thmstylethree}%
\newtheorem{definition}{Definition}%

\begin{document}

\title[Article Title]{Asymmetrically Weighted Dowker Persistence and Applications in Dynamical Systems}


\author*[1]{\fnm{Tobias} \sur{Timofeyev}}\email{Tobias.Timofeyev@uvm.edu}

\author[2]{\fnm{Christopher} \sur{Potvin}}\email{cpotvin@warren-wilson.edu}

\author[3]{\fnm{Benjamin} \sur{Jones}}\email{jones657@msu.edu}
\equalcont{These authors contributed equally to this work.}

\author[4]{\fnm{Kristin M.} \sur{Kurianski}}\email{kkurianski@fullerton.edu}
\equalcont{These authors contributed equally to this work.}

\author[5]{\fnm{Miguel} \sur{Lopez}}\email{mlopez3@sas.upenn.edu}
\equalcont{These authors contributed equally to this work.}

\author[6,7]{\fnm{Sunia} \sur{Tanweer}}\email{tanweer1@msu.edu}
\equalcont{These authors contributed equally to this work.}

\affil[1]{\orgdiv{Department of Mathematics and Statistics}, \orgname{University of Vermont}, \orgaddress{\street{85 South Prospect Street}, \city{Burlington}, \postcode{05405}, \state{Vermont}, \country{USA}}}

\affil[2]{\orgdiv{Department of Environmental Studies (Mathematics and Data Science)}, \orgname{Warren Wilson College}, \orgaddress{\street{701 Warren Wilson Road}, \city{Swannanoa}, \postcode{28778}, \state{North Carolina}, \country{USA}}}

\affil[3]{\orgdiv{Department of Mathematics}, \orgname{Michigan State University}, \orgaddress{\street{619 Red Cedar Road}, \city{East Lansing}, \postcode{48823}, \state{MI}, \country{USA}}}

\affil[4]{\orgdiv{Department of Mathematics}, \orgname{California State University, Fullerton}, \orgaddress{\street{800 N. State College Blvd.}, \city{Fullerton}, \postcode{92831}, \state{CA}, \country{USA}}}

\affil[5]{\orgdiv{Department of Mathematics}, \orgname{University of Pennsylvania}, \orgaddress{\street{209 S. 33rd Street}, \city{Philadelphia}, \postcode{19104}, \state{PA}, \country{USA}}}

\affil[6]{\orgdiv{Department of Mechanical Engineering}, \orgname{Michigan State University}, \city{East Lansing}, \postcode{48823}, \state{MI}, \country{USA}}

\affil[7]{\orgdiv{Department of Computational Mathematics, Science and Engineering}, \orgname{Michigan State University}, \orgaddress{\city{East Lansing}, \postcode{48823}, \state{MI}, \country{USA}}}


\abstract{By their nature it is difficult to differentiate chaotic dynamical systems through measurement. In recent years, work has begun on using methods of Topological Data Analysis (TDA) to qualitatively type dynamical data by approximating the topology of the underlying attracting set. This comes with the additional challenges of high dimensionality incurring computational complexity along with the lack of directional information encoded in the approximated topology. Due to the latter fact, standard methods of TDA for this dynamical data do not differentiate between periodic cycles and non-periodic cycles in the attractor. We present a framework to address both of these challenges. We begin by binning the dynamical data, and capturing the sequential information in the form of a coarse-grained weighted and directed network. We then calculate the persistent Dowker homology of the asymmetric network, encoding spatial and temporal information. Analytically, we highlight the differences in periodic and non-periodic cycles by providing a full characterization of their one-dimensional Dowker persistences. We prove how the homologies of graph wedge sums can be described in terms of the wedge component homologies. Finally, we generalize our characterization to cactus graphs with arbitrary edge weights and orientations. 
Our analytical results give insight into how our method captures temporal information in its asymmetry, producing a persistence framework robust to noise and sensitive to dynamical structure. 
}




\maketitle

\section{Introduction}\label{sec1}


A key objective in the study of large measured dynamical systems is classification of steady states. One challenge is to distinguish between qualitative states of a chaotic system when measurements are nonrepeatable. In climate science and meteorology, the study of recurrent, persistent patterns in atmospheric dynamics, such as blocking events or phases of a large scale oscillation, have served as qualitative descriptors toward these goals. These patterns are known as ``regimes," and the existence and significance of such features have been studied extensively, despite varied definitions and detection methodologies (\cite{Hannachi2017,Gilmore_1998,Palmer99}). Recent work, such as \cite{Strommen_2022} and \cite{Kappe_Bottinger_Leitte_2022}, has proposed to frame regimes via a unifying mathematical framework in topological data analysis. This is part of a growing body of work that seeks to understand dynamical systems via an underlying topology (\cite{Yalniz20, Maletic_Zhao_Rajkovic_2016,Tanweer_Khasawneh_Munch_Tempelman_2024}).

Persistent homology is of particular utility in these applied settings, promising to bridge the gap between discrete data and smooth manifolds. This utility is in service of finding topological invariants and the low dimensional summary of qualitative features in the form of a persistence barcode or landscape. In addition to direct interpretation, this description has also found use as input to  machine learning algorithms for automated system classification (\cite{Pun_Lee_Xia_2022, Hussain_Shah_Rafia_Fatima_Huerta-Cuellar_Garcia-Lopez_Mata_amirez_Jaimes-Reategui_2025}).

One limitation of current approaches to finding topological invariants in dynamical data (such as regimes) is that we would like those topological invariants to be features of dynamical behavior, and their topology alone does not explicitly describe a qualitative dynamic. While topological descriptions of state space point cloud data capture geometric structure, they do not encode temporal ordering or the direction of flows in that space. This aspect represents a limitation of the topological information encoded in the point cloud samples. In this work we study how these features can be made to reflect both positional patterns in state space and the way in which the system directs flow through it. Our goal is to incorporate temporal information into topological approaches for the classification of dynamical systems.

We present a framework for a topological description of dynamics, including temporal information through asymmetric structure. We draw from the field of symbolic dynamics, in which a dynamical system is partitioned into a discrete set of states, and trajectories are studied through the sequences they generate. In this work, we construct such a representation by partitioning the state space into coarse-grained bins, following approaches such as \cite{molkenthin_networks_2014} and \cite{Myers_2023}. These bins consolidate proximate states, and a directed graph is formed by recording transitions between bins along the trajectory.

This directed network encodes temporal information through edge directionality, providing a natural bridge between time series data and combinatorial representations of dynamics. Other discretization strategies exist, including Markov partitions (\cite{Adler_1996}), clustering-based approaches (\cite{Franch2020}), and entropy-based symbolizations (\cite{Galatolo2010}). The binning method was chosen because it provides a straightforward approach to recovering the physical attributes of each state of the dynamical system, and a regular granularity allows paths and persistence over the network to correspond with distances in phase space.

The resulting graph representation allows us to study dynamical behavior through its structure. In particular, we focus on periodic behavior and recurrent dynamics (\cite{Bauer_Hien_Junge_Mischaikow_2025,Strommen_2022}), which manifest as cycles in the network and corresponding topological features. This representation ultimately enables the application of persistence for directed topologies to extract a qualitative description of the system.

For our persistent homology computations, we employ an asymmetric homology theory known as Dowker homology, originally introduced in \cite{Dowker_1952} for relations between two sets and later extended to directed networks in \cite{Chowdhury_Memoli_2016}. This construction is part of a larger body of work developing persistent homology for directed combinatorial structures, including approaches based on flag complexes (\cite{chaplin_notion_2024, luet2019Flagser}), path homology (\cite{Grigor'yan20}), and walk-length filtrations (\cite{Munoz_Munch_Khasawneh_2025}). 
We focus on the Dowker complex as it has strong underlying symmetry in directionality via Dowker Duality, which guarantees equivalent homology under time reversal  (see Section \ref{sec:network_persistence} and \cite{Dowker_1952}) as well as existing theory in computability.
Similar methods have also been applied in settings such as biological systems (\cite{Peek_Pritam_Skerritt_Chalup_2025}) and time-dependent networks (\cite{Ye_Jiang_Jiang_Li_2023}). In this framework, we further develop theoretical connections between Dowker persistence and dynamical behavior, particularly in relation to cycles.

Section \ref{sec:bg} provides background on persistent homology, directed network persistence, and symbolic dynamics. We then establish a connection between Dowker homology and the graph-theoretic notion of dominating sets (\cite{CARO_Dominating_Number}) in Section \ref{sec:dom_set_persistence}, showing how persistence depends on the size and structure of such sets. In Section \ref{sec:cycle_persistence}, we give a complete classification of one-dimensional homology for weighted, directed cycle graphs with arbitrary edge orientations. This relates persistent features to underlying cycle structure. We extend these results in Section \ref{sec:wedge_persistence} by showing that persistence behaves additively under wedge sums, allowing for the characterization of a broader directed graph class, known as cactus graphs. Returning to the dynamical setting, Section \ref{sec:dyn_traj} shows that trajectories induce contractible maximal Dowker complexes, ruling out infinite persistence in this context. Finally, in Section \ref{sec:experiments} we present computational experiments on systems including Lorenz `63 and Charney–DeVore, demonstrating how these topological features reflect underlying dynamical behavior. The classification of cactus graphs inspired the acronymization for our publicly available code (CACTIS)\footnote{Github repository available at \url{https://github.com/cactismath/CACTIS}}.

\section{Background}\label{sec:bg}

\subsection{Persistent Homology}
The homology $H_n(X)$ of a topological space $X$ is an algebraic description of its topological structure in terms of non-bounding cycles of dimension $n$. This classical theory was extended by \cite{Edelsbrunner2000}, permitting the study of point cloud data sets endowed with topologies of varying granularity. At each level, the birth and death of homological features is recorded in a diagram as a summary of its topological structure. Numerous applications of topological persistence arise in biology (\cite{Nicolau2011}), neuroscience (\cite{Giusti2016}), image processing (\cite{Carlsson2007}) and more (\cite{DONUT}). The study of persistence itself has become an independent subfield of topology and matured the practice into a general purpose data analysis tool supported with implementations in various programming languages (\cite{giotto}, \cite{Fasy2014}, \cite{javaplex}).

We assume the reader is familiar with simplicial homology and other basic definitions from algebraic topology (c.f. \cite{Hatcher_AT}) but give a brief summary here to establish our notation, which we borrow from \cite{chowdhury2018AsymNets}.
A \textbf{simplicial complex} $X$ over a set $V$ is a collection of subsets of $V$ called simplices. The collection $X$ is required to be downward closed, that is, if a simplex $\sigma$ is in $X$, then each subset $\sigma'\subseteq \sigma$ must also be in $X$. We will assume that each subset is linearly ordered $\sigma = \{v_0 < v_1 < \dots < v_k\}$ where $v_i \in V$ and $k \in \mathbb{Z}_+$. We call such a $\sigma$ a $k$\textbf{-simplex} and denote them by $\sigma = [v_0, v_1, \dots, v_k]$ to emphasize the orientation. Any simplicial complex can be oriented by putting a total order on the underlying vertex set. Note that the chosen orientation is arbitrary and only needed for the homology calculation.

For a simplicial complex $X$ and a dimension $k \in \mathbb{Z}_+$, a $k$-chain is a formal sum of simplices $\sum_{i \in I} a_i \sigma_i$ such that $I$ is a finite index set, $a_i \in \mathbb{R}$ and $\sigma_i$ is a $k$-simplex in $X$. We denote the vector space of $k$-chains by $C_k(X)$. These spaces are related by \textbf{boundary maps} $\partial_k : C_{k}(X) \to C_{k-1}(X)$ which lower the dimension of a $k$-simplex via an alternating sum:
\[ \partial_k ([v_0, v_1, \dots, v_k]) := \sum_{i = 0}^k (-1)^i[v_0, v_1, \dots , \hat{v}_i, \dots, v_k]\]
where $\hat{v}_i$ denotes that $v_i$ is omitted from the corresponding simplex. This collection of chain spaces and boundary maps form a \textbf{chain complex} $C_\bullet(X) = (C_k(X), \partial_k)_{k \in \mathbb{Z}_+}$, which satisfy a boundary condition $\partial_{k} \circ \partial_{k+1} = 0$ for all $k \in \mathbb{Z}_+$. The $k$-\textbf{cycles} of a chain complex $C_\bullet(X)$ are those $k$-chains in $\ker\partial_k$ and $k$-\textbf{boundaries} are those $k$-chains in $\text{im} \, \partial_{k+1}$. The $k$-th homology of the simplicial complex $X$ is defined as the quotient $H_k(X) := \ker\partial_k / \text{im} \, \partial_{k+1}$. We refer to non-zero equivalence classes of $H_k(X)$ as topological features.

The basis of persistent homology is to track the changes in homology in a sequence of simplicial complexes. This is enabled by the fact that the homology construction $H_k(-)$ is functorial, that is, given a morphism of simplicial complexes $f : X \to Y$, there is an induced map on homology $H_k(f) : H_k(X) \to H_k(Y)$ for each dimension $k$. A \textbf{filtration} is a sequence of simplicial complexes $\{X_i\}_{i \in \mathbb{N}}$ totally ordered by containment, that is, $X_i\subseteq X_{i+1}$ for each $i$. The injections $X_i \hookrightarrow X_{i+1}$ induce maps on homology $\mu_{i,j} : H_k(X_i) \to H_k(X_j)$ for $i \leq j$. For each dimension $k$, the sequence of vector spaces and induced maps $\{H_k(X_i), \mu_{i,j}\}_{i \leq j}$ is the \textbf{persistence module}. A fundamental theorem in TDA asserts that persistence modules are completely characterized by \textbf{persistence barcodes}, which record the birth and death times of the homology classes (see Figure \ref{fig:unweighted_consistent_hex} as an example) (\cite{edelsbrunner2010computational}). Barcodes provide a visual summary of the number and lifespan of topological features at each stage of the filtration. 

We also introduce a tool used in the analysis of a complexes topology. Given a finite family of sets $C=\{U_i\}_{i\in I}$ with indexing set $I$, the \textbf{nerve} of $C$, denoted  Nerv$(C)$ is defined as the set family containing all $J\subseteq I$ such that $\cap_{i\in J}U_i\neq \emptyset$. The nerve is a complex that encodes that intersection structure of $C$. The main use of the nerve for this work is the setting where $K=\cup_{i\in I}K_i$ is an abstract simplicial complex and each $U_i=|K_i|,$ where $|\cdot|$ is the geometric realization.  The \textbf{nerve theorem} then states that if $\cap_{i\in J} U_i$ is empty or contractible for all $J\subseteq I$, then Nerv$(C)$ is homotopy equivalent to $K=\cup_{i\in I}K_i$ (\cite{Borsuk1948}).

\subsection{Network Persistence}\label{sec:network_persistence}

We now overview how persistent homology can be used to study the shape of dynamical system trajectories. One could take the trajectory as a point cloud and apply the standard pipeline, but this would be neglecting the sequencing of points in the trajectory. We model discrete samples of a trajectory as directed networks and apply the work of \cite{chowdhury2018AsymNets} to study their homology. A key insight from that work was to use the Dowker complex of the adjacency matrix to encode directionality within the simplicial representation of the digraph. This is opposed to other simplicial models of directed networks such as the flag complex \cite{luet2019Flagser}.

The Dowker complex defined by \cite{Dowker_1952} is a way of extracting a topological space from a relation. Formally, we define a relation as a Boolean function $R : X \times Y \to \{0,1\}$ between two sets $X$ and $Y$. Note that every relation admits a transpose $R^T : Y \times X \to \{0,1\}$ where $R^T(y,x) := R(x,y)$. 

\begin{definition}
    The \textbf{Dowker source complex of $R$} is a simplicial complex $\FD^\text{so}(R)$ with vertices $X$ and simplicies $\sigma \subseteq X$ whenever there exists a $y \in Y$ such that $R(x,y) = 1$ for every $x \in \sigma$. 
\end{definition}

Notice this definition is inherently asymmetrical; one can give a reciprocal definition of the \textbf{Dowker sink complex} $\FD^{si}(R)$, with vertices $Y$ and simplices governed by $X$. Equivalently, this is the Dowker source complex of the transposed relation $\FD^\text{si}(R) = \FD^\text{so}(R^T)$. A celebrated theorem of Dowker is that, homotopically, these two constructions are equivalent. 

\begin{theorem}[Dowker]\label{thm:dowker}
    For a binary relation $R$, the simplicial complexes $\FD(R)$ and $\FD^{si}(R)$ are homotopy equivalent. 
\end{theorem}

\begin{figure}[h!]
    \centering
    \includegraphics[width=10cm]{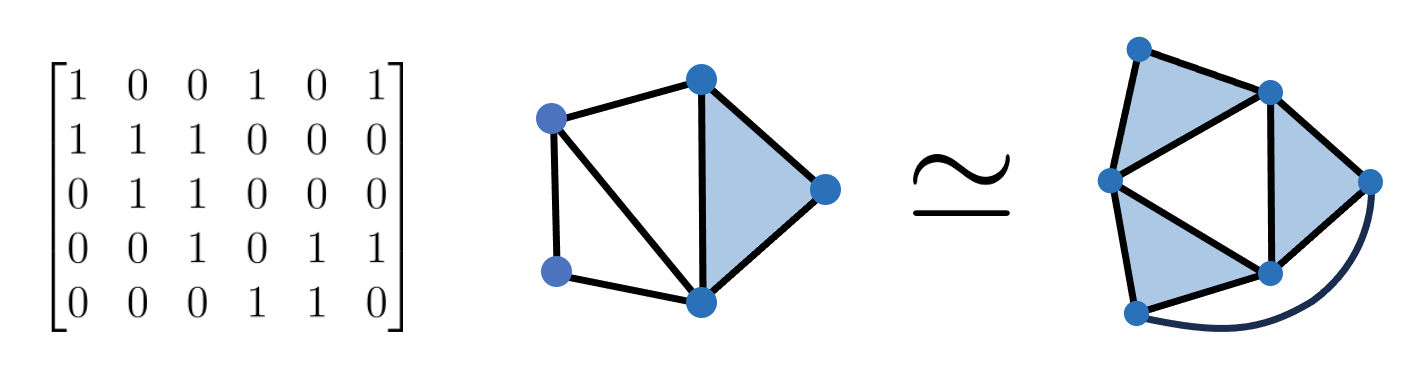}
    \caption{A relation $R$, represented as a binary matrix, followed by the Dowker complex $\FD(R)$ and $\FD^{si}(R)$. Both of these complexes are homotopy equivalent to a wedge of two circles.}
    \label{fig:enter-label}
\end{figure}

\begin{definition}
    A \textbf{network} or \textbf{weighted digraph} $G = (X, \omega)$ is a set of vertices $X$ together with a function $\omega : X \times X \to \posR := \RR_{\geq 0} \cup \{\infty\}$. 
\end{definition}

A \textbf{path} $p : x \rightsquigarrow y$ in $G$ is a finite sequence of vertices $p = ( x_0, x_1, \dots, x_n)$ where $x_0 = x$ and $x_n = y$. The weight of a path $p$ is the sum $\omega(p) = \sum_i \omega(x_i,x_{i+1})$ or $\omega(p) = 0$ if $p=(x_0)$ is the constant path. We associate to $G$ its shortest path function $d : X \times X \to \posR$ where 
    \begin{equation}\label{def: shortest path function}
    d(x,y) := \inf_{p: x \rightsquigarrow y} \omega(p). 
    \end{equation}
\begin{definition}
    Let $G=(X,\omega)$ be a weighted digraph. The \textbf{path completion} of $G$ is the weighted digraph $P(G) = (X,d)$ where $d$ is the shortest path function of $G$ in Equation \ref{def: shortest path function}. 
\end{definition}
    
The function $d$ is a so called \textbf{Lawvere metric}, a relaxation of a standard metric, not requiring the symmetry or separation axioms of a standard metric, as well as possibly taking infinite values. This distance function allows us to treat a graph as a Lawvere metric space on which to consider Dowker persistence. It is easy to show that the path completion is idempotent, that is, $P(P(G))=P(G)$.

We now turn our attention to the Dowker filtration over a network. Given a network $G=(X,\omega)$, define $R_\delta(G) : X \times X \to \{0, 1\}$ to be the relation
\begin{equation}
        R_\delta(G)(x,y) := 
        \begin{cases}
            1, & \omega(x,y) \leq \delta \\
            0, & \text{else}
        \end{cases}
\end{equation}


\begin{definition}[Dowker Filtration]
    Let $G = (X, \omega)$ be a weighted digraph. The \textbf{Dowker source filtration} $\{\FD^\text{so}_\delta(G):=\FD^\text{so}(R_\delta(P(G))) \}_{\delta \in \posR}$ is a filtration of Dowker complexes associated to the relation functions $R_\delta(P(G))$. The \textbf{Dowker sink filtration} counterpart is defined analogously with $\FD_\delta^\text{si}(R_\delta(P(G)))$.
\end{definition}

The terminology of `sink' versus `source' corresponds to how $n$-simplices are added to the respective Dowker complex (see Figure \ref{fig:source-filt}). In the former incoming edges determine added simplices, whereas in the latter, the outgoing edges add simplices.
\begin{figure}[h]
    \centering
    \includegraphics[width=10cm]{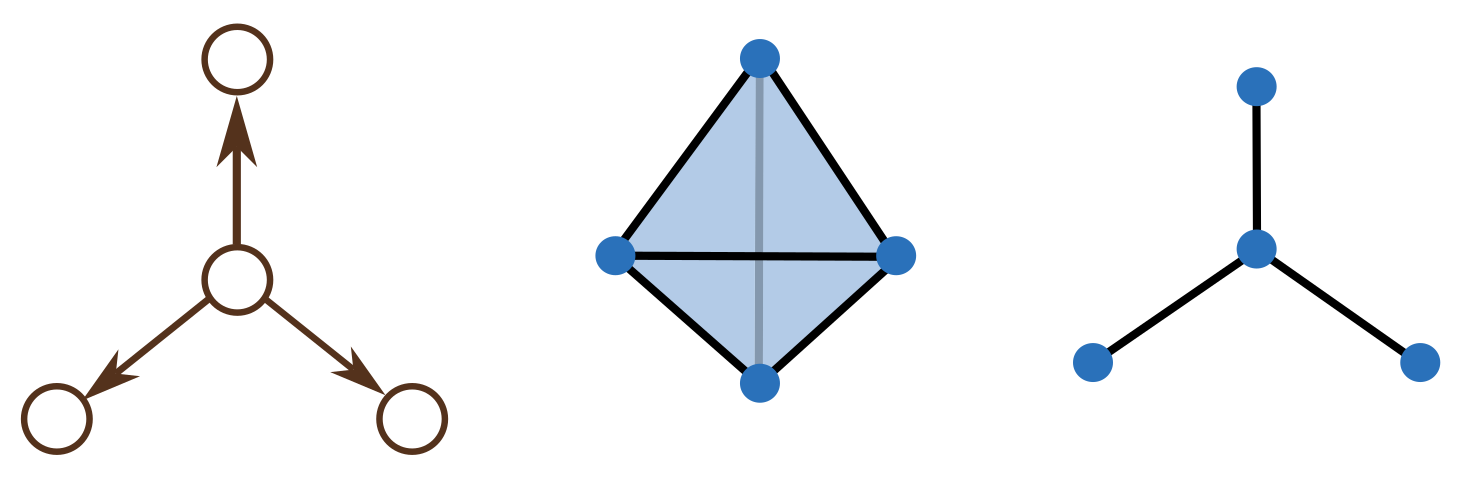}
    \caption{In the source filtration, a source node of degree three contributes a $2$-simplex to the Dowker complex. In the sink filtration, three $1$-simplices are added instead.}
    \label{fig:source-filt}
\end{figure}

We note that in our framework, as applied to a dynamical system, Theorem \ref{thm:dowker} guarantees that homologies in forward and backward time are equivalent. This matches our expectation in the characterization of steady states.

In the remainder of this paper we only consider the Dowker source complex and write $\FD_\delta(G) = \FD^\text{so}_\delta(G)$. Additionally, in Sections \ref{sec:dom_set_persistence} and \ref{sec:wedge_persistence} we will need to address the Dowker complex of the graph without the path completion as an intermediate step. For this we use the notation $\FD_\delta^\star(G) :=\FD(R_\delta(G))$ given a network $G$.

\section{Characterizing Dowker Persistence}

\subsection{Dominating Sets and Dowker Persistence}\label{sec:dom_set_persistence}

One of the fascinating properties of Dowker filtrations is that homological features can persist indefinitely. As we construct our Dowker complex on an underlying directed graph, we are interested in understanding the structural relationship between the two, and the representation of graph features in persistence. 

In this section we frame these invariant persistent features in terms of the dominating sets of this underlying graph, and explore the relationship between graph structure and the largest possible complex. We first show that any dominating set informs a bound on the persistence parameter at which a maximal complex is finalized. We then show that the intersections of the neighborhoods of dominating set elements fully describe the maximal complex, and in particular their nerve provides a reduced description.

We note that the subject of this section benefits from being considered without the path completion as it pertains more generally to the structure of the Dowker complex. For this reason we write our results in terms of the non path completed Dowker complex, $\FD^\star(G)$. The results remain applicable to the path completed Dowker complex for the same reason, and since $\FD^\star(G)=\FD^\star(P(G))$.

Given a weighted digraph $G$, we define the \textbf{maximal Dowker complex} of a Dowker filtration to be the Dowker complex $\FD_{\delta_\text{max}}(G)$ with a choice of $\delta_\text{max} >0$ such that $\FD_{\delta_\text{max}}(G)=\FD_{\delta}(G)$ for all $\delta >\delta_\text{max}$. We define the maximal non path completed complex, $\FD^\star_{\delta_\text{max}}(G)$, in the same  way. Note that the existence of this maximal Dowker complex follows from the assumption that the graph $G$ is finite. To give some intuition for this value, we find a straightforward upper bound of $\delta_\text{max}$ with the maximum edge weight.

\begin{prop}
    Let $G=(X,\omega)$ be a weighted directed graph. Let $D\subseteq X\times X$ be the set of ordered vertex pairs $(x,y)$ such that $\omega(x,y)<\infty$.
    Let $\delta_w = \max_{(v,w)\in D}\omega(v,w)$ and $\delta_\text{max}$ be the value at which the non path completed maximal Dowker source complex, $\FD^\star_{\delta_\text{max}}(G)$, occurs. Then $\delta_w \geq \delta_\text{max}.$
\end{prop}
\begin{proof}
    Let $\delta_w$ be defined as above, and $\delta'>\delta_w$. Naturally, $\FD^\star_{\delta_w}(G)\subseteq \FD^\star_{\delta'}(G)$. We show the other inclusion. Let $\sigma\subseteq X$ be a simplex of $\FD^\star_{\delta'}(G)$. It follows that there exists $y\in X$ such that $(y,x)\in R_{\delta'}(G)$ for all $x\in \sigma$. Thus $\omega(y,x)\leq \delta'$ for all $x\in \sigma$. It follows that $\omega(y,x)\leq \delta'$ for all $x\in \sigma$. Thus $\omega(y,x)< \infty$ for all $x\in \sigma$, implying that $\omega(y,x)\leq \delta_w$ for all $x\in \sigma$. Hence $\sigma$ is a simplex of $\FD^\star_{\delta_w}(G)$.
    This yields the desired inclusion.
\end{proof}


In order to proceed with our analysis we introduce the graph theoretic terminology of a directed dominating set. This is a set of vertices which `see' all other vertices of the graph with finite weight. 
\begin{definition}
    Let $G=(X,\omega)$ be a weighted directed graph. 
    Define a \textbf{source dominating set of $G$} as a subset $K\subseteq X$ such that for all $x\in X$ there exists $k\in K$ such that $\omega(k,x)<\infty$.
    We call a dominating set minimal if no strict subset is also a dominating set.
\end{definition}

We show the correspondence between source dominating sets and the source complexes which are the focus of our analysis. As source dominating sets see all the vertices in the graph, and elements of the source complex are induced by common in-neighbors, it stands to reason that the simplices in the maximal complex can be induced by the dominating set. We can formalize this intuition in the following proposition.


\begin{prop}\label{so_dom_set_simplex}
    Let $G=(X,\omega)$ be a weighted directed graph, and $K$ a source dominating set. Let $\sigma$ be a simplex of $\FD^\star_{\delta_\text{max}}(G)$. Then there exists $k\in K$ and $\delta$ such that $(k,x)\in R_\delta(G)$ for all $x\in \sigma$. In other words there exists a vertex $k$ in the source dominating set which can generate $\sigma$ in $\FD^\star_{\delta}(G)$.
\end{prop}
\begin{proof}
Let $\sigma$ be a simplex of $\FD^\star_{\delta_\text{max}}(G)$ and $K$ be a source dominating set of $G$. Since $\sigma$ is a simplex of $\FD^\star_{\delta_\text{max}}(G)$, it follows that there exists vertex $h$ such that $\omega(h,x)<\delta_{\text{max}}$ for all $x\in \sigma$. However, since $h\in X$ it follows that either an element of the dominating set, $k\in K$, sees it with an edge of finite weight, or $h=k\in K$. In either case it follows that all edges from $k$ to $x\in \sigma$ also have finite weight. Thus $\delta^* = \max_{x\in \sigma} \omega(k,x)$ is finite and $(k,x) \in R_{\delta^*}(G)$ for all $x\in \sigma$.
\end{proof}

Now let us consider the local scale of the complex. Let $N(v)=\{x\in G | \omega(v,x)<\infty\}$ be the out neighborhood of $v$ in $G$. We assume the graph has trivial self loops so that $v\in N(v)$. We also let $[N(v)]$ denote the simplex with vertices $N(v)$. With this new notation we can reframe Proposition \ref{so_dom_set_simplex} as the union of fully connected neighborhoods of the elements of a dominating set. Let Cl$(\Delta)$ denote the simplicial closure of a set $\Delta$.

\begin{figure}
    \centering
    \includegraphics[
    trim= 0 7cm 0 0,
    width=0.85\linewidth]{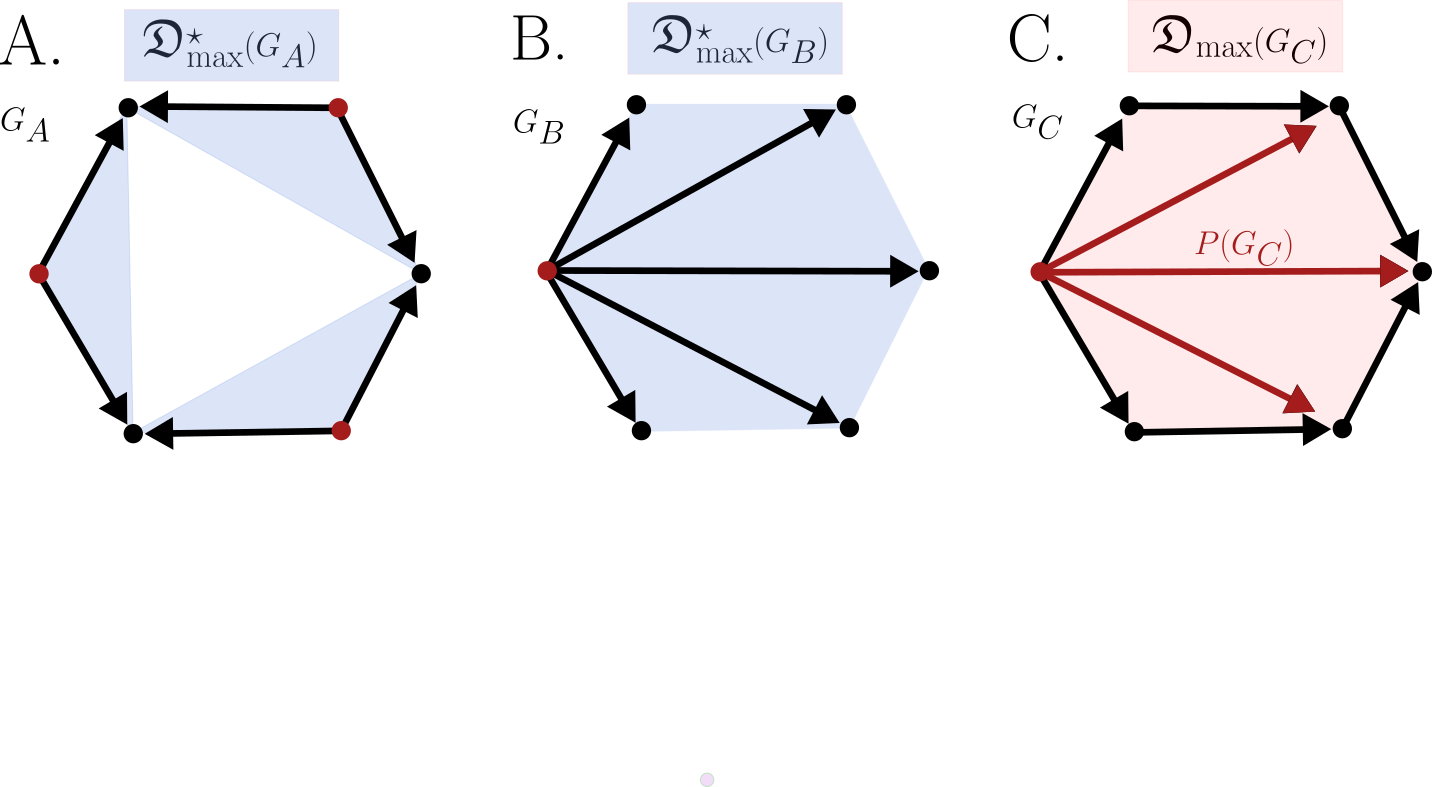}
    \caption{Above we have $3$ example cases. Each base graph has $6$ vertices. The choice of edge orientation determines the vertices which can form a minimal source dominating set in the path completion. The red vertices form minimal dominating sets and are in these cases unique. In \textbf{A} and \textbf{B} the blue shaded regions denote the simplices of the maximal Dowker complex without path completion. This allows us to compare and see that \textbf{A} has nontrivial first homology, allowed by its minimal dominating set of size $3$, whereas \textbf{B} has trivial homology with its minimal dominating set of size $1$. In \textbf{C} the red shaded region indicates the single fully connected maximal simplex of the maximal Dowker complex with path completion. The edges imparted by the path completion are in red as well.}
    \label{fig:dom_set_examples}
\end{figure}

\begin{cor}\label{cor:cover}
    Let $G$ be a network. For any dominating set $K = \{v_0, \dots, v_k\}$ we have that $\FD^\star_{\delta_\text{max}}(G) = \text{Cl}\left([N(v_0)] \cup \dotsc \cup [N(v_k)]\right)$. 
\end{cor}

In this way, the dominating sets capture the `reach' of the source complex placed upon the graph. In particular, it carves out a minimal subgraph of the original with the same end behavior.
Given a graph $G=(V,\omega)$ and a source dominating set $K$, define the subgraph of $G$ induced by $K$ to be $G^K(V,\omega')$ where
\begin{equation*}
    \omega'(x,y) = 
    \begin{cases}
            \omega(x,y) & \hspace{10pt}\text{if $x\in K$}\\
            \infty &    \hspace{10pt}\text{otherwise.}
    \end{cases}
\end{equation*}

We can formalize that this subgraph captures the maximal behavior of the source complex in the following claim.

\begin{prop}\label{dom_graph_equiv}
    Let $G=(X,\omega)$ be a weighted directed graph and $K$ a dominating set. Let $H = G^K$ be the subgraph induced by $K$. Then $\FD^\star_{\delta_\text{max}}(G)=\FD^\star_{\delta_\text{max}}(H)$
\end{prop}
\begin{proof}
    Since the edge set of $H$ is a subset of the edge set of $G$, it follows that for all $\delta$, any simplex of $\FD^\star_\delta(H)$ must also be contained in $\FD^\star_\delta(G)$. Thus $\FD^\star_{\delta_\text{max}}(H)\subseteq \FD^\star_{\delta_\text{max}}(G)$

    Let $\sigma\in \FD^\star_{\delta_\text{max}}(G)$. Then by Proposition \ref{so_dom_set_simplex} it follows that $\sigma\in\FD^\star_{\delta}(H)$ for some $\delta$ value. Thus $\sigma\in \FD^\star_{\delta_\text{max}}(H)$. Hence $\FD^\star_{\delta_\text{max}}(H)\supseteq \FD^\star_{\delta_\text{max}}(G)$.
\end{proof}

Further, because the finite weight edges of the dominating set induced subgraph are a subset of the edges of the full graph, it follows the that edge weight bound of the former  is less than or equal to the edge weight bound of the latter.

\begin{prop}
    Let $G=(X,\omega)$ be a weighted directed graph, and $K$ a dominating set. Let $R\subseteq X\times X$ be the set of ordered vertex pairs $(x,y)$ such that $x\in K$ and $\omega(x,y)<\infty$.
    Let $\delta_r = \max_{(v,w)\in R}\omega(v,w)$ and $\delta_\text{max}$ be the value at which the maximal Dowker complex occurs. Then $\delta_r$ is an upper bound on $\delta_{\text{max}}$ and a lower bound on the edge weight bound of $G$.
\end{prop}

\begin{proof}
    We note that $\delta_r$ as defined is also the edge weight bound of the the subgraph induced by $K$, which we denote by $H$.
    
    Let $\delta_e$ be the edge weight bound of $G$. The inequality $\delta_r\leq \delta_e$ follows from the fact that the finite weight edges of $H$ are a subset of those in $G$.

    Let $\sigma\in \FD^\star_{\delta_\text{max}}(G)$. By Proposition \ref{so_dom_set_simplex} there exists $v\in K$ and $\delta$ such that $(v,x)\in R_\delta$ for all $x\in \sigma$. Since $\delta_r$ is at least a big as $\max_{x\in \sigma} \omega'(v,x)$ by construction, it follows $\delta\leq \delta_r$. Thus $\sigma\in \FD^\star_{\delta_r}(G)$. It follows that $\FD^\star_{\delta_\text{max}}(G) \subseteq \FD^\star_{\delta_r}(G)$. Thus $\delta_r\geq \delta_{\text{max}}$.
\end{proof}

This result tells us that any dominating set induces its own edge weight bound on the maximal Dowker complex. This further illustrates the strong relationship between the dominating set and the maximal Dowker complex. This also points to another key feature; the vertices seen by an element of the dominating set become a fully connected simplex in the maximal complex. Further, the union of these large neighborhood simplices from the dominating set is the maximal complex. This suggests that any topological features in the maximal complex are the result of the interactions between these large simplices induced by the dominating set elements. In other words, their nerve is homotopic to the maximal complex. Let Nerv$(A)$ denote the nerve of the set $A$.

\begin{prop}\label{prop:dominantnerve}
    Let $G$ be a weighted graph, and $K$ one of its minimal dominating sets. Then $\FD^\star_{\delta_\text{max}}(G)$ is homotopy equivalent to the \text{Nerv}$\left(\{[N(v_1)], \dotsc , [N(v_k)]\}\right)$.
\end{prop}

\begin{proof}
    Let $K=\{v_1,\dotsc, v_k\}$ be a minimal dominating set of $G$, with indexing set $I=\{1,\dotsc, k\}$. By Corollary \ref{cor:cover} it follows $\FD^\star_{\delta_\text{max}}(G) = \text{Cl}\left([N(v_1)] \cup \dotsc \cup [N(v_k)]\right)$. Let $N$ be the nerve of $\{|N(v_1)|, \dotsc , |N(v_k)|\}$, that is the set of all $J\subseteq I$ such that $\cap_{i\in J}|N(v_i)| \neq \emptyset$, where $|\cdot|$ denotes the geometric realization of the simplex. Notice that $N$ is a simplicial complex because it is closed under containment. By the nerve theorem, if all $J\subseteq I$ have $\cap_{i\in J}|N(v_i)|$ as empty or contractible, then $N$ is homotopy equivalent to the union $\FD^\star_{\delta_\text{max}}(G) = [N(v_1)] \cup \dotsc \cup [N(v_k)]$. By the construction of the maximal Dowker complex, we know that any intersection of neighborhoods is entirely seen by some vertex (in its neighborhood) and thus gets fully connected. Thus the corresponding intersection of simplicial complexes is contractible and we are done.  
\end{proof}

We might be interested in properties of the underlying graph which guarantee that the maximal Dowker complex is contractible. Proposition \ref{prop:dominantnerve} allows us to frame this in terms of the interactions of dominating set neighborhoods. As an example, if the intersection of all the neighborhoods is nonempty, then the aforementioned nerve is fully connected, implying the complex is contractible.

\begin{cor}
    Let $G$ be a network, and $K=\{v_0,\dotsc,v_k\}$ a minimal dominating set. If $N(v_1) \cap \dotsc \cap N(v_k)\not=\emptyset$, then the maximal Dowker complex $\FD^\star_{\delta_\text{max}}(G)$ has trivial reduced homology.
\end{cor}

When described in terms of a nerve, the size of the dominating set bounds the homological complexity of the maximal complex.

\begin{prop}\label{prop:Dominating_set_reduction}
    Let $G$ be a network, and $K$ one of its minimal dominating sets. If $|K|\leq k$ then $\FD(R_\text{max})$ has trivial $k-1$ reduced homology (and greater) for $k \geq 1$.
\end{prop}

\begin{proof}
    By Proposition \ref{prop:dominantnerve}, we have that $\FD^\star_{\delta_\text{max}}(G)$ is homotopic to the nerve of $\{|N(v_1)|, \dotsc , |N(v_k)|\}$ which we will call $N$. This is a simplicial complex on a base set of $k$ elements. Since an $n$ simplex has $n+1$ vertices, it follows that $N$ with $k$ vertices cannot contain the boundary of a simplex of dimension $k$ or higher. Thus $N$ must have trivial homology for dimension $k-1$ and greater. Since $N$ is homotopically equivalent to $\FD^\star_{\delta_\text{max}}(G)$ it follows that $\FD^\star_{\delta_\text{max}}(G)$ shares these trivial homology classes.
\end{proof}

This result relates Dowker homology to minimum dominating set size, known as the directed dominating number in the literature, \cite{CARO_Dominating_Number}. Dominating sets have a rich relationship with the characterization of graphs and their structural connectivity, so we believe it important that the Dowker complex picks up on this and how. It ties a binary graph theoretical understanding of an `important subset' to our model of weighted persistence. 

Let us again consider the Dowker persistence with the path completion. On $P(G)$ a dominating set describes a set of vertices such that in $G$ the entire vertex set is a directed finite length path away. For this reason, given a strongly connected graph $G$, the completion $P(G)$ has a dominating set of size $1$. From Proposition \ref{prop:Dominating_set_reduction} we find that in this case the maximal homology is contractible. This shows how our dominating set analysis of the maximal complexes reveals fewer connected regions coming from the presence of infinite distances in our Lawvere metric $d$.

We believe the additional information captured in the persistence adds nuance to the project of typing high dimensional time series data. In parallel, perhaps this relationship suggests that Dowker persistence could aid in computing or approximating dominating sets, as the problem is in general NP-complete.

\subsection{Persistence of Cycle Graphs}\label{sec:cycle_persistence}

Our framework approaches typing dynamical point cloud data by encoding temporal information in a binned digraph, and evaluating topological features in the accompanying Dowker filtration. We are particularly interested in the ability of this method to distinguish periodic behavior in the dynamical trajectory from cycles formed simply due to point cloud proximity. In the following section we address this question analytically. We start by showing that every nontrivial element of the first homology corresponds with a cycle in the underlying graph, directly relating topological features in the filtration to the measured data. We then characterize the persistence of cycle graphs with any edge orientation. This gives insight into how the included directed information informs the persistence of topological features.

To begin our analysis, we need to get a better sense of what the first homology class tells us about the underlying graph. In particular, the Dowker complex on the path completion of a graph has $1-$simplices which correspond to paths ignoring the orientation in the underlying graph. Therefore all nontrivial elements of the first homology correspond to graph cycles in the underlying graph. We formalize this in the following result. 

\begin{prop}\label{prop:hom_to_bdry}
    Consider $\FD_\delta(G)$ for some network $G=(V,\omega)$ and $\delta$. Let $\alpha$ be a non-bounding $1-$cycle of the complex. There exists $\eta=\sum_{i}e_i\eta_i$ for which $\eta_i=[x,y]$ implies $\omega_{G}(x,y)\leq \delta$ or $\omega_G(y,x)\leq \delta$ and $[\eta]=[\alpha]$ upon passing to homology.
\end{prop}

\begin{proof}
    Let $\alpha=\sum_i a_i\alpha_i$. For some $j$, suppose that $\alpha_j=[v,w]$ is such that $\omega_G(v,w) > \delta$ and $\omega_G(w,v) > \delta$ hold. By definition, both bounds hold in $P(G)$ as well. It follows by construction that there exists a vertex $x$ such that $\omega_{P(G)}(x,v)\leq \delta$ and $\omega_{P(G)}(x,w)\leq \delta$. This implies the existence of the 2-chain $[x,v,w]$ in the chain complex. Thus the first chain group contains
    \[
    -a_j \partial_2([x,v,w]) + \alpha
    = -a_j (\alpha_j-[x,w]+[x,v]) + \sum_i a_i\alpha_i.
    \]
    The addition of the boundary removes the term $a_j\alpha_j$ from $\alpha$, instead leaving $-a_j[x,w]$ and $a_j[x,v]$ which correspond with edges of weight less than $\delta$ in $P(G)$.
    Because it differs by a boundary, this new chain is homologous to $\alpha$. If we repeat this process for all elements of $\alpha$, we find that $\alpha$ is homologous to $\zeta=\sum_i z_i \zeta_i$ where $\zeta_i=[x_i,v_i]$ implies $\omega_{P(G)}(x_i,v_i)\leq \delta$ or $\omega_{P(G)}(v_i,x_i)\leq \delta$.

    Now we find $\eta$, with edges corresponding to edges in $G$, that $\zeta$ is homologous to. For some $p$, consider $\zeta_p=[x_p,v_p]$. As we saw previously this implies $\omega_{P(G)}(x_p,v_p)\leq \delta$ or $\omega_{P(G)}(v_p,x_p)\leq \delta$. Without loss of generality, we assume the former. Since the edge weights of $P(G)$ come from directed paths in $G$ it follows that there exist vertices $y_2,\dotsc,y_{m-1}$ such that with $y_1=x_p$ and $y_m=v_p$ we have
    \[
    \omega_{P(G)}(x_i,v_i)=\sum_{j=1}^{m-1}\omega_G(y_j,y_{j+1}).
    \]
    Since we assume edge weights are positive, it follows that  $\omega(y_j,y_{j+1})\leq \delta$ for all $1\leq j<m$. It also follows that $\omega_{P(G)}(y_1,y_j)\leq \delta$ for all $1\leq j<m$ and therefore the complex contains $[x_i,y_1,\dotsc,y_m,v_i]$.

    It follows that the first chain group contains
    \begin{align*}
        \sum_{1\leq j\leq m-2} 
        z_p \partial_2([y_j,y_{j+1},y_m])
        &=
        \sum_{1\leq j\leq m-2} 
        z_p ([y_{j+1},y_m]-[y_j,y_m]+[y_j,y_{j+1}])
        \\
        &=
        -z_p[y_1,y_m]+z_p \sum_{1\leq j\leq m-2} [y_j,y_{j+1}].
    \end{align*}

    Thus 
    \begin{align*}
        \zeta + \sum_{1\leq j\leq m-2} 
        z_p \partial_2([y_j,y_{j+1},y_m])
        &=
        \zeta -z_p \zeta_p+z_p \sum_{1\leq j\leq m-2} [y_j,y_{j+1}],
    \end{align*}
    this replaces the term $z_p\zeta_p$ in $\zeta$ with a sum of $1$ simplices corresponding to edges with weights less than $\delta$ in $G$. Since the difference is again a boundary, this new chain is homologous to $\zeta$. 
    If we again repeat this process for all the elements of $\zeta$, we find a new $1-$chain $\eta$ with elements $\eta_j=[x_p,x_q]$ such that $\omega_G(x_p,x_q)\leq \delta$ or $\omega_G(x_q,x_p)\leq \delta$. Additionally, the homology class $[\alpha]=[\zeta]=[\eta]$. 
    
\end{proof}

The above result ties the existence of non bounding cycles in homology to graph cycles with arbitrary edge orientation in the underlying graph, stating through the contrapositive that you can only have first homology cycles if they correspond with graph cycles. 

This motivates the pursuit of a better understanding of the persistence of Dowker complexes on cycle graphs with arbitrary orientations. In what remains of this section, we characterize this persistence. For us cycle graphs are weighted directed graphs whose underlying undirected graph is a cycle. 

\begin{definition}
    A \textbf{weighted cycle network} of size $n$ is the weighted digraph $G = ([n],\omega)$ on $n$ vertices $[n] := \{1,\dots,n\}=\mathbb{Z} /n\mathbb{Z}$ such that either $\omega(i,j) < \infty$ or $\omega(j,i) < \infty$ if and only if $j = i+1 \pmod{n}$.
\end{definition}

If such a graph has all edge orientations in the positive direction (mod $n$) then we call this a \textbf{consistently oriented cycle graph}. 


\subsubsection{Consistently Oriented Cycle Graphs}

We begin by considering a subclass of cycle graphs we have particular interest in due to their relationship with periodic behavior in our setting; consistently oriented cycle graphs. In this section we fully characterize their 1 dimensional Dowker persistence, extending the previous results of \cite{chowdhury2018AsymNets}, which do so in the case of uniform weight $1$. Two examples satisfying this condition can be seen in Figures \ref{fig:unweighted_consistent_hex} and \ref{fig:unweighted_consistent_oct}.
Both exemplify the results of the above which find a nontrivial element in the first homology when $1\leq \delta\leq \lceil \frac{n}{2} \rceil$. We will show that a graph theoretic perspective aids in the generalization of this result, describing first homological features in terms of paths in the underlying graph. For the consistently oriented cycle, we will show that the homology class closes when two vertices see one another across the graph with shortest paths. 
\begin{figure}[t]
    \begin{minipage}{.5\textwidth}
        \centering
        \includegraphics[width=\linewidth]{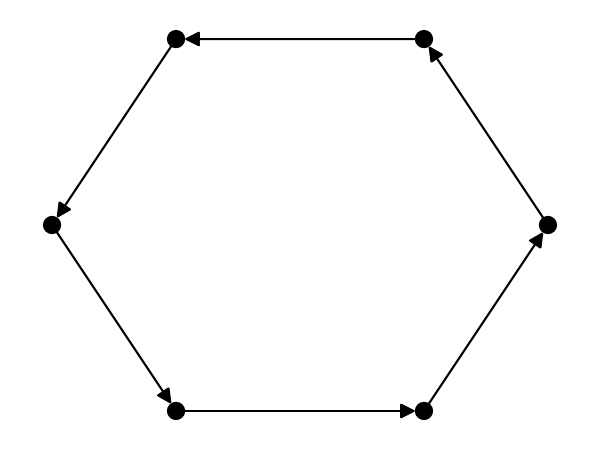}
    \end{minipage}%
    \begin{minipage}{0.5\textwidth}
        \centering
        \includegraphics[width=\linewidth]{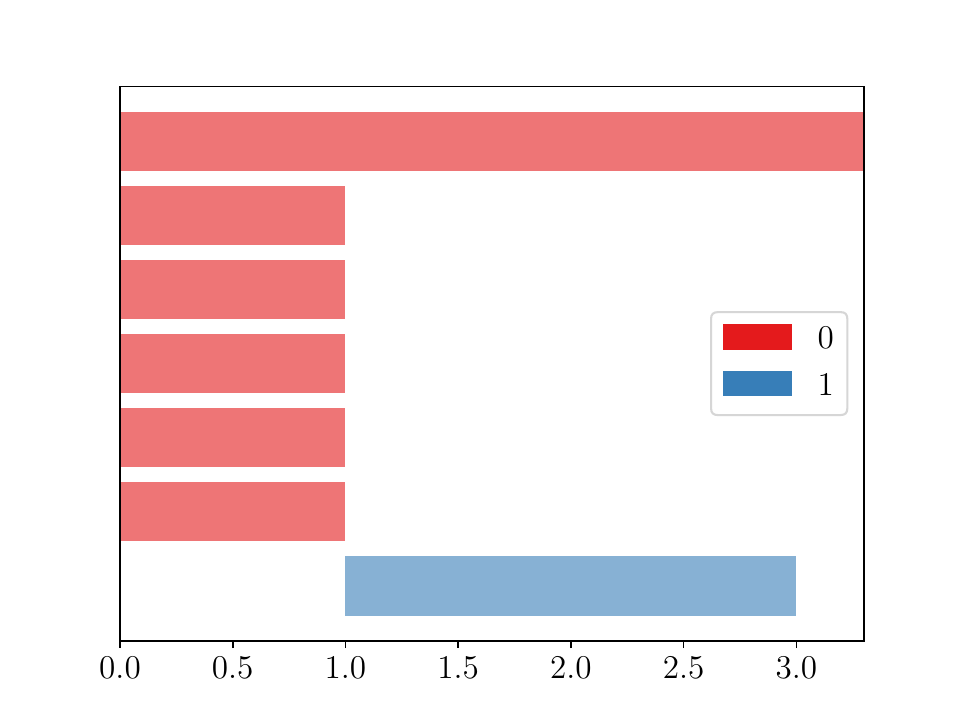}
    \end{minipage}
    \caption{Example of a uniform weight $1$, consistently-oriented cycle with six vertices (left) and its corresponding barcode (right). See the first homology class born at $\delta=1$ and die at $\delta=\lceil 6/2 \rceil =3$.}
    \label{fig:unweighted_consistent_hex}
\end{figure}

\begin{figure}[t]
    \begin{minipage}{.5\textwidth}
        \centering
        \includegraphics[width=0.8\linewidth]{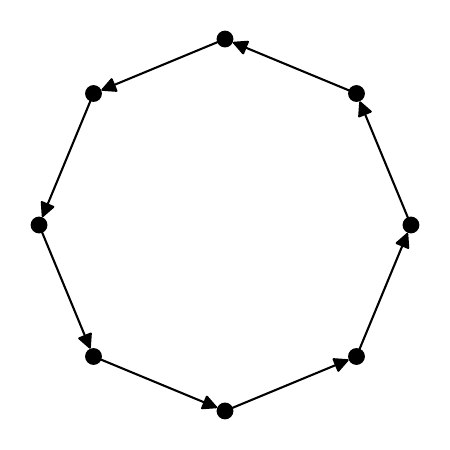}
    \end{minipage}%
    \begin{minipage}{0.5\textwidth}
        \centering
        \includegraphics[width=\linewidth]{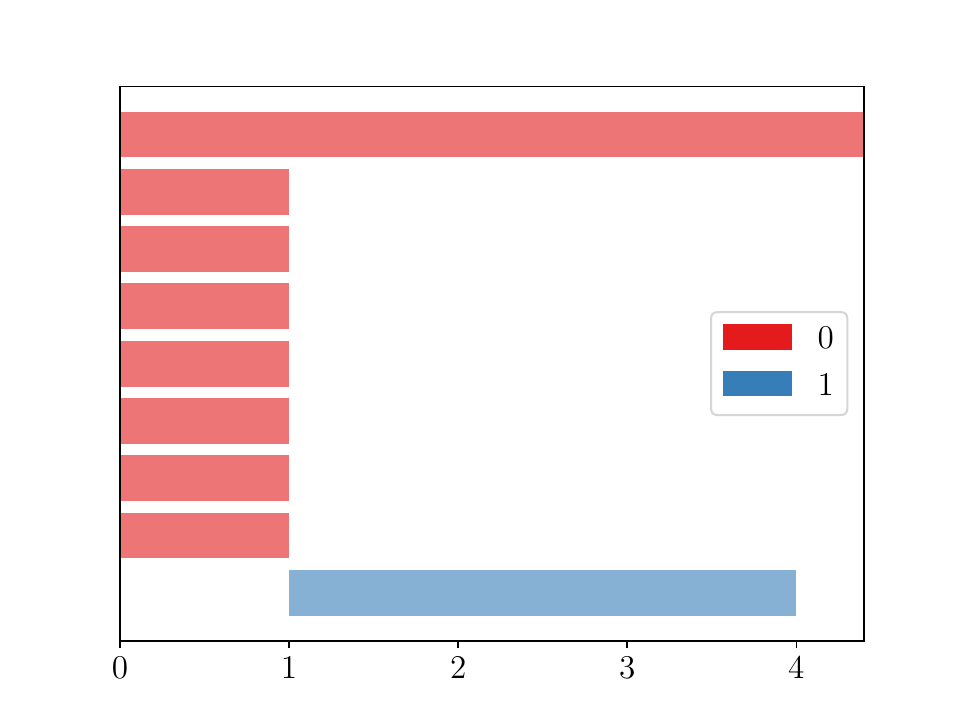}
    \end{minipage}
    \caption{Example of a uniform weight $1$, consistently-oriented cycle with eight vertices (left) and its corresponding barcode (right).  See the first homology class born at $\delta=1$ and die at $\delta=\lceil 8/2 \rceil =4$.}
    \label{fig:unweighted_consistent_oct}
\end{figure}

We start by distinguishing one property of consistently oriented cycle graph complexes which will be useful in the coming analysis. That is, every $1-$simplex of the Dowker complex corresponds to an edge between those vertices which has weight less than $\delta$. This comes from the fact that all vertices in a consistently oriented cycle graph have out degree $1$.

\begin{lem}\label{lem:simplex_order}
    Let $G$ be a consistently oriented cycle graph, and $\FD_\delta(G)$ the Dowker complex on $P(G)$. Every simplex $\Delta = [x_1,\dotsc,x_k]\in \FD(R_\delta)$ has one vertex $x_j$ such that $\omega(x_j,x_i)<\delta$ for all $i$. Moreover, every $1-$simplex in $\Delta$ corresponds to at least one directed edge in $P(G)$ with weight less than $\delta$.
\end{lem}

\begin{proof}
    Let the $\Delta = [x_1,\dotsc,x_k]$ be the simplex in question, with labels reflecting the order of $V(G)$. This simplex is included in the Dowker complex if and only if there exists a $v\in V(G)$ such that $\omega(v,x_i)<\delta$ for all $i$. 

    Suppose $v$ is not a vertex of $\Delta$ and by symmetry let $v=0$ without loss of generality. Notice that in a consistently oriented cycle graph, there exists a unique directed path between any two vertices. Further, given vertices with $x_1<x_2<x_3$ in the vertex order, $x_2$ lies on the unique directed path between $x_1$ and $x_3$, respecting that vertex order. Thus $\omega(x_1,x_3)=\omega(x_1,x_2) + \omega(x_2,x_3)$. Therefore, since $\omega(v=0,x_1)<\delta$ and $\delta>\omega(v,x_i)=\omega(v,x_1)+\omega(x_1,x_i)$ for all $i>1$, it follows $\omega(x_1,x_i)<\delta$ for $i\geq 1$.

    Note that this same argument implies $\omega(x_2,x_i)<\delta$ for $i\geq2$, and so on for all the vertices, implying the existence of sufficiently weighted edges for all pairs of vertices in $\Delta$.
\end{proof}

One can verify that the path completion of a consistently oriented cycle graph admits any vertex as a dominating set, which reflects that this network is strongly connected. Additionally, following Proposition \ref{prop:dominantnerve} we know that the maximal complex will always be contractible and therefore have a trivial first homology. Thus the aim is to show under what conditions on the underlying graph and $\delta$ there exists a nontrivial first homology. We make an inductive argument through a reduction to a smaller cycle graph.

We define an edge contraction operation in order to provide a reduced description of the Dowker complex. Let $G$ be a cycle graph on $n$ vertices. Define the contraction of $G$ with respect to vertex $n$ to be the graph $G'=G-\{n\}$ with the additional nonzero edge $\omega_{G'} (n-1,1) = \omega_G(n-1,n)+\omega_G(n,1)$. This creates a new cycle graph and preserves edge weight sums that don't have $n$ as an endpoint. We start by characterizing when the Dowker complexes over these graphs are homologically equivalent.

Given a graph $G$, and complex $\FD_\delta(G)$,  we denote the subcomplex associated to the $i$th vertex's out $\delta$ neighborhood as  $U_i=\text{Cl}([\{j| \omega_{P(G)}(i,j)\leq \delta\}])$. This set is in the full complex by construction, and represents the simplices that vertex $i$ contributes. Naturally the union of all out neighborhoods will yield the full complex, $\cup_{i=1}^n U_i = \FD_\delta (G)$, meaning the set $\{U_i\}_{i\in [n]}$ forms a proper cover of $\FD_\delta(G)$, which we call the neighborhood cover.

Additionally, for all $i$, $U_i$ is a simplex of some dimension and therefore any intersections between elements of the neighborhood cover are also simplices and therefore contractible. Thus by the nerve theorem $\FD_\delta(G)$ is homotopy equivalent to the nerve of $\{U_i\}_{i\in [n]}$.

In the following propositions we establish the relationship between the contracted cycle graph complex $A'$ and the uncontracted cycle graph complex $A$. We start by considering the neighborhood cover of $A'$, which has $n-1$ elements, $\{U_i\}_{i\in [n-1]}$, and adding an additional element ${U'}_{n-1}=\text{Cl}([\{j\neq n | \omega_{P(G)}(n,j)\leq \delta\}])$. In words, this is the subset of $U_1$ which vertex $n$ sees in $G$. Thus its inclusion in the cover won't change the homological equivalence and captures part of the connectivity of $A$ in the nerve, that is the vertices which $n$ sees.  

Notably, the addition of $U'_n$ to the cover will not make the nerve equal to $A$. It misses the subsimplices of $[\{j|\omega_{P(G)}(j,n)\leq \delta\}]$. This simplex is induced by the vertex furthest away from $n$, say $b$, as on a consistently oriented cycle graph it necessarily sees all the vertices between. We denote this simplex, induced by $b$ as $\Delta_b$. We formalize how we can reconstruct the original complex exactly, using the nerve and $\Delta_b$, in Proposition \ref{prop:reduced_equiv}.


\begin{prop}\label{prop:reduced_equiv}
    Let $G$ be a consistently oriented cycle graph and $A=\FD_\delta(G)$.
    Let $b$ be the smallest index where $\omega_{P(G)}(b,n)\leq \delta$, and $\Delta_b=[b,b+1,\dotsc,n]$ the induced simplex. Consider the neighborhood cover of the Dowker complex on the reduced cycle graph of $G$, $A'=\FD_\delta(G')$, augmented with an $n$th element $\{U_1, \dotsc, U_{n-1},U_n\}$ where $U_n\equiv \text{Cl}([\{j\neq n | \omega_{P(G)}(n,j)\leq \delta\}])$. Then Nerv$\left(\{U_1, \dotsc, U_{n-1},U_n\}\right) \cup \text{Cl}(\Delta_b)$ equals $\FD_\delta(G)$ under the bijective vertex mapping $\phi:[n]\to \{U_i\}_{i\in [n]}$. In other words, The addition of $\Delta_b$ to the nerve of the modified cover of the reduced complex $A'$ makes equivalence with $A$.
\end{prop}
\begin{proof}
    We show the equivalence by proving that $[v_1,\dotsc,v_m]\in A$ if and only if $[\phi(v_1),\dotsc,\phi(v_m)]\in$ Nerv$
    \left(\{U_1, \dotsc, U_{n-1},{U}_n\}\right) \cup \text{Cl}(\Delta_b)$.

    ( $\Longrightarrow$ )
    Let $[v_1,\dotsc,v_m]\in A$. Assume this is ordered so that $i<j\implies\omega_{P(G)}(v_i,v_j)\leq \delta$ (we can do this by Lemma \ref{lem:simplex_order}). We consider two cases. 

    (Case $1$) Assume that $v_m\neq n$. From the claim $\omega_{P(G)}(v_j,v_m)\leq \delta$ for all $1\leq j<m$. Since $v_m\in V(G')$ it follows $\omega_{P(G')}(v_j,v_m)\leq \delta$ for all $1\leq j<m$. Thus $v_m\in\cap_{v_j\neq n} \phi(v_j)$. But if there exists $1\leq j<m$ such that $v_j=n$, the fact that $\omega_{P(G)}(v_j,v_m)\leq \delta$ implies $v_j\in U_n$ too. Thus $v_m\in\cap_{j=1}^m \phi(v_j)$. Therefore $[\phi(v_1),\dotsc,\phi(v_m)]\in\text{Nerv}\left(\{U_1,\dotsc,U_n\}\right)$.
    
    (Case $2$) Assume $v_m=n$. From the claim $\omega_{P(G)}(v_j,v_m=n)\leq \delta$ for all $1\leq j<m$. By construction $\phi(v_j)\in \Delta_b$ for all $1\leq j< m$. Thus $[\phi(v_1),\dotsc,\phi(v_m=n)]\in \text{Cl}(\Delta_b)$.

    ( $\Longleftarrow$ )
    Let $[\phi(v_1),\dotsc,\phi(v_m)]\in$ Nerv$\left(\{U_1, \dotsc, U_{n-1},{U'}_n\}\right) \cup \text{Cl}(\Delta_b)$. By definition $\cap_{i=1}^m \phi(v_i)\neq \emptyset$ or $[\phi(v_1),\dotsc,\phi(v_m)]\in$ Cl$(\Delta_b)$. 
    
    In the first case there exists $v_p$ in this intersection such that $\omega_{P(G)}(v_i,v_p)\leq \delta$ for $1\leq i\leq m$. This works for $v_i= n$ by our definition of $U_n$. Thus there exists $v_q$ with largest edge weight to $v_p$, $\omega_{P(G)}(v_q,v_p)$, which in turn implies that it sees all the other vertices in question and we have $[v_1,\dotsc,v_p]\in A$.

    In the second case, by the definition of $\Delta_b$, for all $1\leq i\leq m$ the edge weight $\omega_{P(G)}(v_i,v_m=n)\leq \delta$. Thus, as before there exists $v_p$ with largest edge weight $\omega_{P(G)}(v_p,n)$ which can see all the other vertices in question and we have $[v_1,\dotsc,v_p]\in A$.
\end{proof}

The above shows the relationship between the reduced complex and the original. We now show how and when the addition of $\Delta_b$ changes the first homology.
In other words, when does the reduced complex have the same homology as the original? We find in Proposition \ref{prop:1sthom_equiv} that a mismatch can only occur when vertices can see each other in the underlying graph.

\begin{figure}
    \centering
    \includegraphics[width=0.9\linewidth]{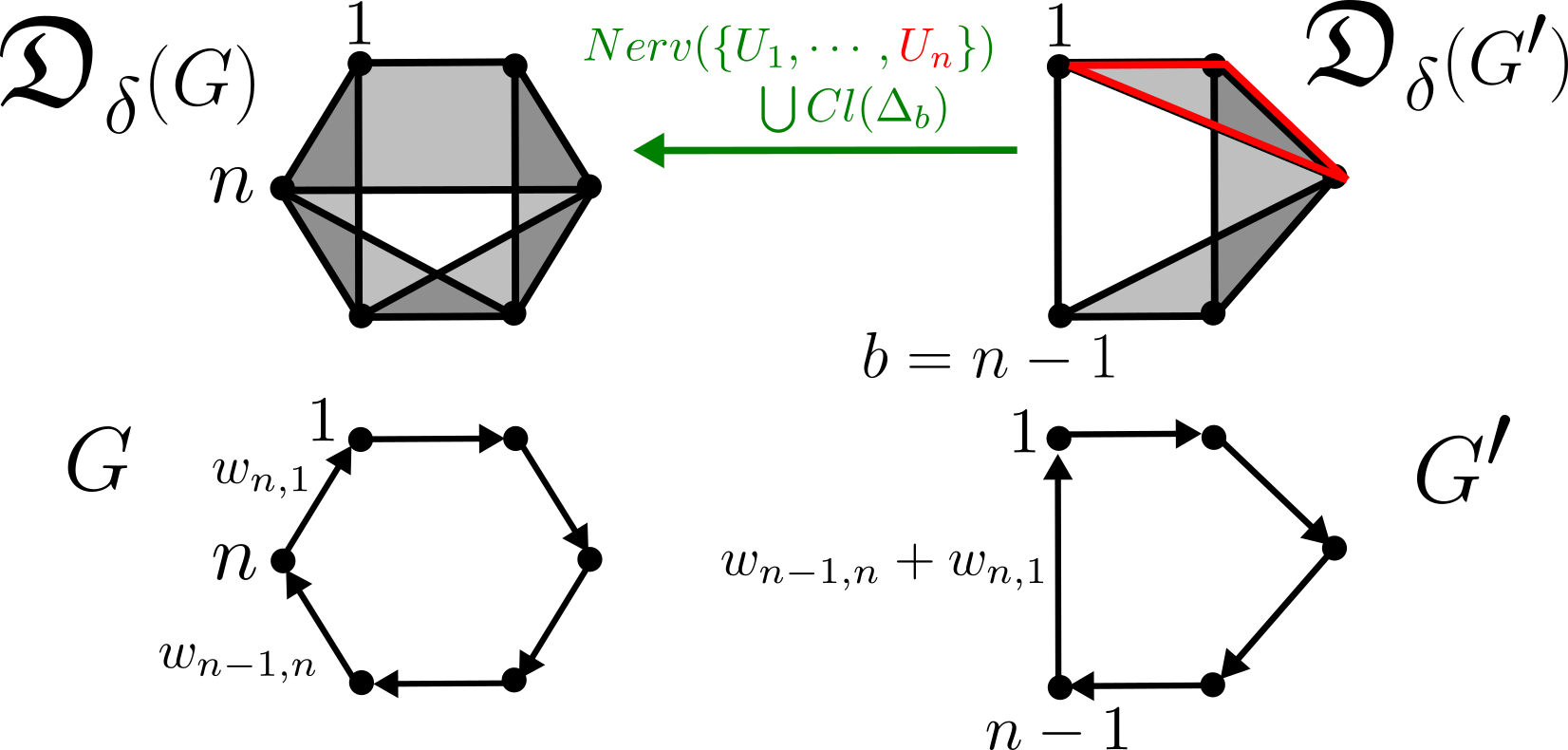}
    \caption{This figure exemplifies the argument of Proposition \ref{prop:reduced_equiv}. On the bottom we depict the original graph $G$, as well as its contraction, $G'$. On the top, we show their associated Dowker complexes and indicate the augmented nerve of the complex of $G'$ which recovers the complex of $G$.}
\end{figure}

\begin{prop}\label{prop:1sthom_equiv}
    Let $G$ be a consistently oriented cycle graph. Assume $\delta\geq \min_{i\in [n]}\{\omega_G(i,i+1)+\omega_G(i+1,i+2)\}$ and choose labeling such that $i=n-1$ satisfies this constraint. Assume that no two vertices see each other with weight less than $\delta$ in $P(G)$. Let $b$ be the smallest index where $\omega_{P(G)}(b,n)\leq \delta$, and $\Delta_b=[b,b+1,\dotsc,n]$ the induced simplex. Let $G'$ be the contraction of $G$ with respect to $n$. Then $A= \FD_\delta(G)$ and $A' = \FD_\delta(G')$ have the same first homology. 
    
\end{prop}

\begin{proof}
     Consider the nerve of the neighborhood cover of $A'$ augmented with an $n$th element, $N=$Nerv$\left(\{U_1,\dotsc,U_n\equiv\text{Cl}([\{j\neq n | \omega_{P(G)}(n,j)\leq \delta\}\}])\right)$, and let $B=\text{Cl}(\Delta_b)$. Since $G$ and $G'$ are cycle graphs, Proposition \ref{prop:hom_to_bdry} has shown that each corresponding complex can have at most one nontrivial first homology class. Further, by Proposition \ref{prop:reduced_equiv} we have $A\approx N\cup B$, and by the nerve theorem we have homotopy equivalence between $N$ and $A'$. Thus it is sufficient to show that $N$ has nontrivial first homology if and only if $N\cup B$ has nontrivial first homology. 

     Suppose $\delta<\max_i \omega_G(i,i+1\pmod{n})$. Then $\delta<\max_i \omega_{G'}(i,i+1\pmod{n})$. By the contrapositive of Proposition \ref{prop:hom_to_bdry} it follows that neither $A$ nor $A'$ may have nontrivial first homology. Thus we assume $\delta\geq \max_i \omega_G(i,i+1\pmod{n})$.

     ($\implies$) For the first implication, suppose there exists an $H^1$ cycle in $N\cup B$. By Proposition \ref{prop:hom_to_bdry}, we may choose a cycle representative from $N\cup B= A$ of $\gamma=\sum_{i=1}^n b_i e_i$ with $\{e_i\}_{i=1}^n$ being the set of $1-$simplices corresponding to the edges of $G$. Note these are all in the complex because $\delta\geq \max_i \omega_G(i,i+1\pmod{n})$. Because our assumption is that $\delta\geq \omega(n-1,n)+\omega(n,1)$, it follows that $U_{n-1}\cap U_1\neq \emptyset$ implying $N$ includes $e_{n-1}=[n-1,n]$. This complex includes $e_n=[1,n]$ since $U_n\subseteq U_1$, as well as all the other $e_i$ by construction. Thus the first chain group of $N$ contains $\gamma$, and since $N$ is a subcomplex of $N\cup B$, it follows that $\gamma$ is nonbounding in $N$ as well.



    ( $\Longleftarrow$ ) For the second implication we assume there exists an $H^1$ cycle in $N$ and one does not exist in $N\cup B$. 
    
     It follows that the addition of $B$ adds $2-$simplices to $N$ which allow for the non-bounding cycle of the nerve to be expressed as a boundary. We showed in Proposition \ref{prop:hom_to_bdry} that any non-bounding cycle of a cycle graph is homologous to the outer cycle from the base graph. Let $\gamma=\sum_{i=1}^n b_i e_i$ be the nontrivial cycle representative from $N$, where $\{e_i\}_{i=1}^n$ is the set of $1-$simplices corresponding to outer cycle edges, that is the edges of $G$. By assumption there exists a $2-$chain of $N\cup B = A$, $\alpha$, such that $\partial_2(\alpha)=\gamma$. 
     
     Recall that $\Delta_b$ is a simplex in $A$ created by the vertices $v_j$ which have $n$ in their out $\delta$ neighborhood. We note that the simplices in the closure of $\Delta_b$ which do not include $n$ are already in $N$ since the weight function of $P(G)$ is preserved in $P(G')$ for edges not containing $n$. By construction, $n$ is a sink in all the remaining simplices of $B$. With this in mind, it follows

    \begin{align*}
        \gamma &= \partial_2\left(\alpha\right)\\
        &= 
        \partial_2\left(\sum_{\Delta=[i,j,k]\in N\cup B}a_{\Delta} [i,j,k]\right)\\
        &= 
        \partial_2\left(
        \sum_{\Delta=[i,j,n]\in B\setminus N}a_{\Delta} [i,j,n]
        +
        \sum_{\Delta=[i,j,k]\in N}a_{\Delta} [i,j,k]
        \right)\\
        &=
        \sum_{\Delta=[i,j,n]\in B\setminus N}a_{\Delta} ([j,n]-[i,n]+[i,j])
        +
        \sum_{\Delta=[i,j,k]\in N}a_{\Delta} ([j,k]-[i,k]+[i,j])\\
        &=
        \left(a'_{[n-1,n]}[n-1,n]
            +
            \sum_{\substack{\Delta'=[i,j]\in N\\
            b\leq  i,j< n   }}a'_{\Delta'} [i,j]
            +
            \sum_{\Delta'=[i,n]\in B}a'_{\Delta'} [i,n]
        \right)
        +
        \left(
            \sum_{\Delta'=[i,j]\in N}a'_{\Delta'} [i,j]
        \right)
    \end{align*}

    If the boundary sum on the left is zero, then the boundaries from $\Delta_b$ must have canceled themselves out and the right sum constitutes a $2-$chain of $N$ whose boundary is $\gamma$, a contradiction. So we assume that the left summand is not zero. 
    
    Suppose $\sum_{\Delta'=[i,n]\in B}a'_{\Delta'} [i,n]=0$.
    Notice that that any boundary containing $[n-1,n]$ with nonzero coefficient must also contain $[i,n]$ for some $i$. Thus this assumption implies $a'_{[n-1,n]}=0$. It follows the boundary of simplices from $\Delta_b$, represented in the left summand, is $\sum_{\substack{\Delta'=[i,j]\in N\\b\leq  i,j< n}}a'_{\Delta'} [i,j]$. 
    However, since $\omega_{P(G)}(b,n-1)\leq \delta$ it follows that $[b,b+1,\dotsc,n-1]\in N$ which implies this boundary is trivial in $N$. This also contradicts the assumption that $\gamma$ is a nonbounding cycle in $N$ since the sum of boundaries trivial in $H^1$ cannot make a nontrivial element.

    Thus there exists some $\Delta'=[i,n]\in B$ where $b\leq i<n-1$ and $a'_{\Delta'}\neq 0$ above. Then it follows that an element in the right summand must cancel this out. Therefore $[i,n]\in B\cap N$. Since $[i,n]\in B$ we know $\omega_{P(G)}(i,n)\leq \delta$ by construction. However, from $[i,n]\in N$ it follows that $\omega_{P(G)}(n,i)\leq \delta$ (by the construction of $U_n$). But now $i$ and $n$ can see each other and we have a contradiction.
\end{proof}

This tells us that the condition of two vertices seeing one another with weight less than $\delta$ is enough to guarantee that the first homology is contractible. This observation in combination with an inductive argument allows us to show the triviality of the first homology class if and only if two vertices can see one another in the underlying graph.

\begin{prop}\label{prop:bidirectionH1}
    Let $G$ be a consistently oriented  cycle graph. Suppose $\delta \geq \max_i \omega(i,i+1)$. Then $\FD_\delta(G)$ has trivial $H_1$ homology if and only if there exist two vertices $y_1,y_2$ in $P(G)$ with $\omega(y_1,y_2),\omega(y_2,y_1)<\delta$.
\end{prop}

\begin{proof}
    ($\Longleftarrow$)     
    Without loss of generality, assume $\max\{ \omega(1,p), \omega(p,1) \} < \delta$ where $p<n$.
    It follows that $1$ sees all vertices $i\leq p-1$ and $p$ sees all vertices $p\leq i\leq n$. Thus $U_1$ and $U_p$ cover the complex. Since their intersection is contractible, it follows that the complex is homotopy equivalent to Nerv$\{U_1, U_p\}$, which is a $1-$simplex and itself contractible.

    ( $\Longrightarrow$ )
    We prove the claim by induction.

    As our base case we take $G$ to be the consistently oriented cycle graph on three vertices. The Dowker complex $\FD_\delta(G)$ 
   then contains a two simplex which fills the cycle when some vertex contains the other two in its $\delta$ out neighborhood. If $\delta$ is also above $\max\{\omega_{P(G)}(1,2),\omega_{P(G)}(2,3),\omega_{P(G)}(3,1)\}$, as is our assumption, then we have a bidirected edge.

    Now assume that the claim is true for cycle graphs on $n-1$ vertices and let $G$ be a cycle graph on $n$ vertices such that no pair of vertices in $P(G)$ see each other with weight under $\delta$. 

    For $\max_i \omega_{P(G)}(i,i+1 \pmod{n})\leq \delta < \min_{i}\{\omega_{P(G)}(i,i+1\pmod{n})+\omega_{P(G)}(i+1\pmod{n},i+2\pmod{n})\}$ the lower bound allows the cycle of the form $\sum_{i=1}^n a_ie_i$ to exist, and the upper bound implies the complex has no $2-$simplices and thus the cycle cannot be a boundary.
    
    Now assume $\delta \geq \min_{i}\{\omega_{P(G)}(i,i+1\pmod{n})+\omega_{P(G)}(i+1\pmod{n},i+2\pmod{n})\}$. By symmetry we may assume without loss of generality this minimum is achieved by $\omega_{P(G)}(n-1\pmod{n},n)+\omega_{P(G)}(n,1)$. Let $G'$ be the contraction of $G$ with respect to vertex $n$. If no pair of vertices in $P(G)$ see each other with this $\delta$, then the same is true for $P(G')$. Since $G'$ is a cycle graph on $n-1$ vertices, our assumption indicates that its Dowker complex has a nontrivial $H^1$ cycle. Proposition \ref{prop:1sthom_equiv} indicates that under these assumptions $\FD_\delta(G')$ has a nontrivial $H^1$ cycle if and only if $\FD_\delta(G)$ has a nontrivial $H^1$ cycle, which proves our claim for $\FD_\delta(G)$.
\end{proof}

This result relates the persistence of the Dowker complex to a shortest cycle in $P(G)$. Nontrivial elements of $H_1$ correspond to shortest paths from a vertex to itself using edges with weight less than $\delta$, and die when $\delta$ allows a path with two edges from a vertex to itself. With this intuition in mind, we can use Proposition \ref{prop:bidirectionH1} to compute the exact interval for which the complex admits a nontrivial first homology.

\begin{prop}\label{prop:concycle}
    Let $G$ be a cycle graph. There exists a nontrivial equivalence class in the $H^1$ homology group of $\FD_\delta(G)$, for $\max_i \omega(i,i+1\pmod{n})\leq \delta< \min_{i\neq j}\max \{\omega(i,j),\omega(j,i)\}$.
\end{prop}

\begin{proof}
    By Proposition \ref{prop:hom_to_bdry} we know that this homology group is trivial for $\delta$ below the lower bound. 

    In Proposition \ref{prop:bidirectionH1} we found that there exists a nontrivial element of the first homology group for $\delta$ above that range if and only if there do not exist $i$ and $j$ such that $\omega(i,j),\omega(j,i)<\delta$. The term $\min_{i\neq j}\max \{\omega(i,j),\omega(j,i)\}$ indicates the first value of $\delta$ at which this property would hold for some $i$ and $j$. If $\delta \geq \min_i \omega(i,i-1\pmod{n})$, then a vertex has low weight edges too all vertices of the graph and the dowker complex is fully connected with trivial first homology. Thus if $\delta$ satisfies the given inequalities, the first homology group has a nontrivial element.
\end{proof}

\begin{figure}[t]
    \begin{minipage}{.5\textwidth}
        \centering
        \includegraphics[width=\linewidth]{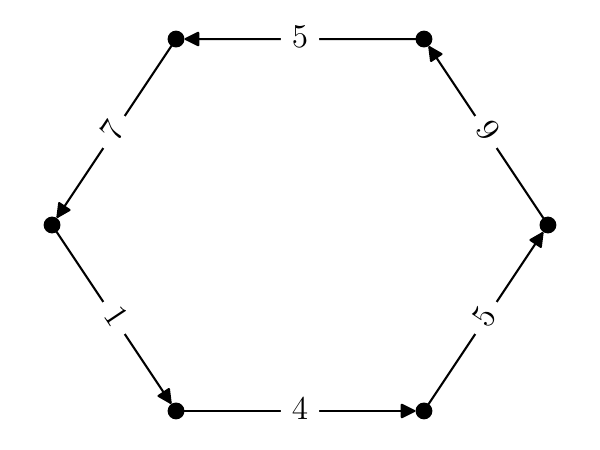}
    \end{minipage}%
    \begin{minipage}{0.5\textwidth}
        \centering
        \includegraphics[width=\linewidth]{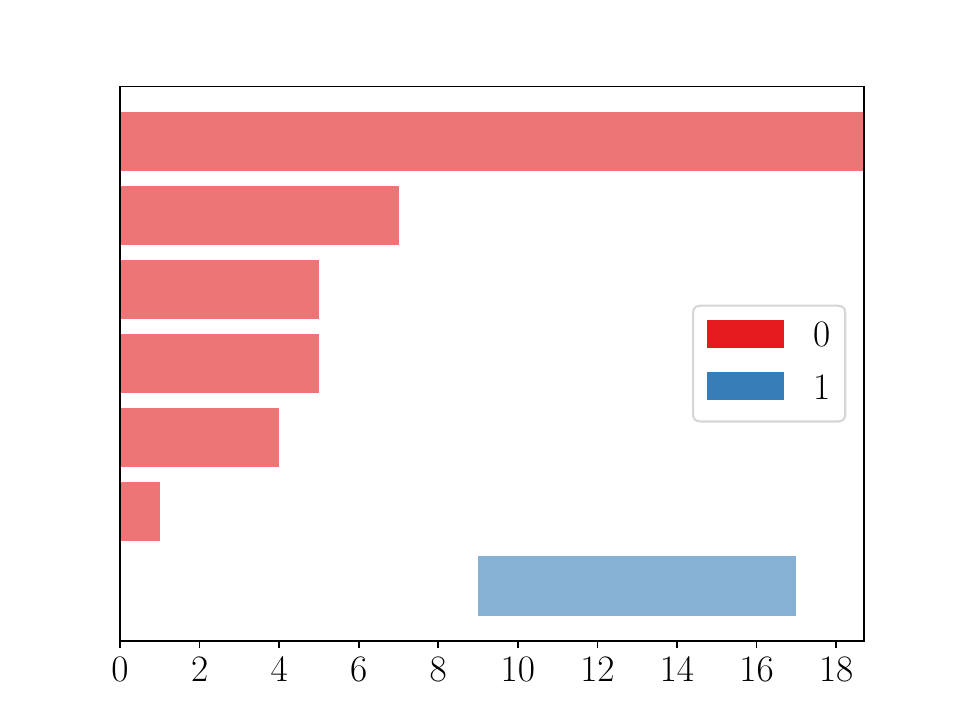}
    \end{minipage}
    \caption{Example of a weighted, consistently-oriented cycle with six vertices (left) and its corresponding barcode (right). The cycle in the first homology is born at the maximimum edge weight, $9$. It dies when two vertices can see each other for the first time, at $17$. There are two such pairs in this graph, the top right and bottom right, or the top left and right-most, each of which are separated by a weight of 5 + 9 = 14, and 5 + 7 + 1 + 4 = 17, and 17 is the minimum number for which this is true. }
    \label{fig:weighted_consistent_hex}
\end{figure}



\begin{figure}[t]
    \begin{minipage}{.5\textwidth}
        \centering
        \includegraphics[width=0.8\linewidth]{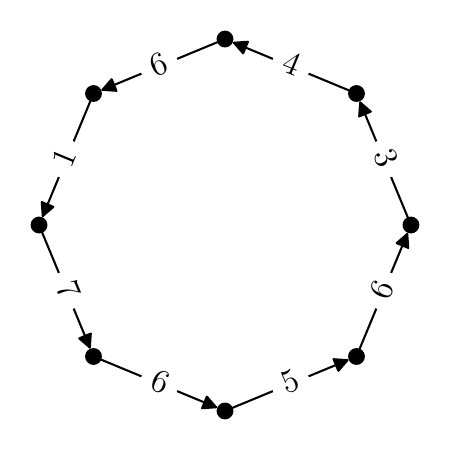}
    \end{minipage}%
    \begin{minipage}{0.5\textwidth}
        \centering
        \includegraphics[width=\linewidth]{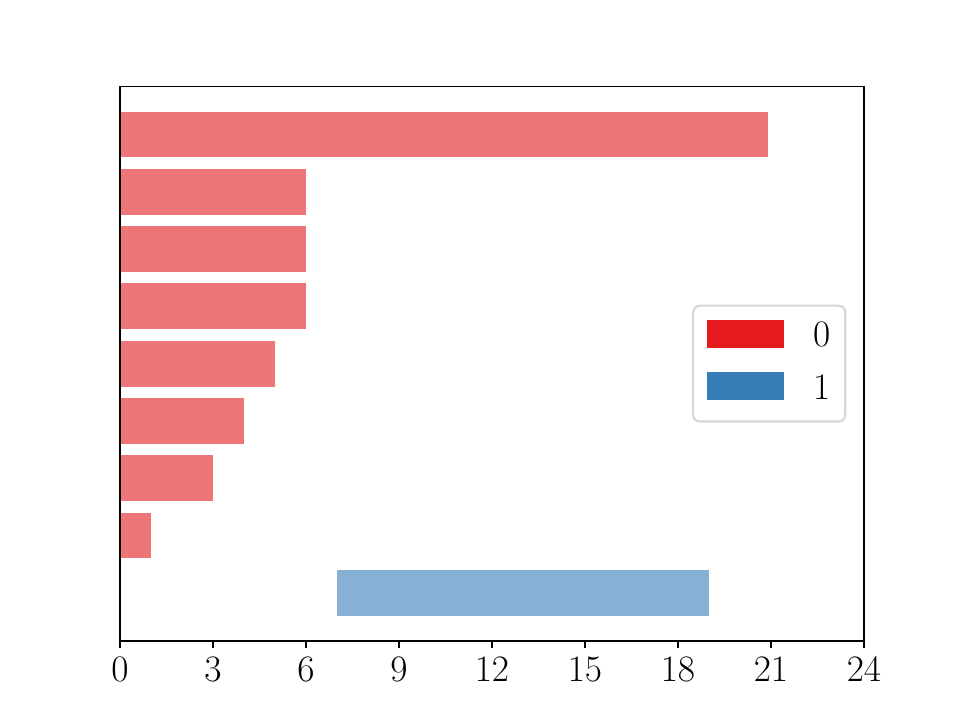}
    \end{minipage}
    \caption{Example of a weighted, consistently-oriented cycle with eight vertices (left) and its corresponding barcode (right). The cycle in the first homology is born at the maximum edge weight, $7$. It dies when two vertices can see each other for the first time, at $19$. In this case, the top left and bottom right vertices are distance $19$ from one another, and close the cycle at that time.}
    \label{fig:weighted_consistent_oct}
\end{figure}


We can see this result exemplified in Figure \ref{fig:weighted_consistent_hex} and Figure \ref{fig:weighted_consistent_oct}.

We note that if no such $\delta$ exists, then there exists a vertex able to see the entirety of the graph with a $\delta$ smaller than the maximum edge weight of $G$. This means that a closure condition is met before a nontrivial cycle is able to form in the filtration and the first homology class remains trivial for the entire filtration. An example of this can be seen in Figure \ref{fig:weighted_consistent_oct_largeEge}.

\begin{figure}[t]
    \begin{minipage}{.5\textwidth}
        \centering
        \includegraphics[width=0.8\linewidth]{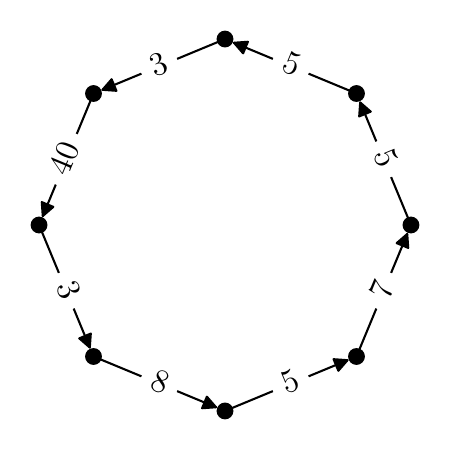}
    \end{minipage}%
    \begin{minipage}{0.5\textwidth}
        \centering
        \includegraphics[width=\linewidth]{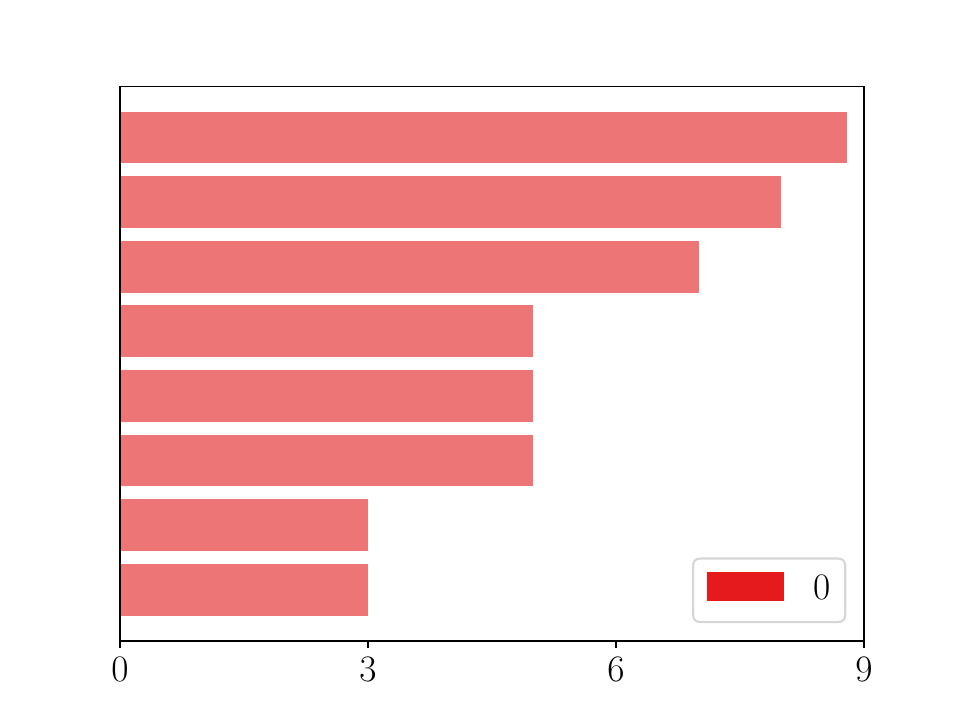}
    \end{minipage}
    \caption{Example of a weighted, consistently-oriented cycle with eight vertices, including one large edge weight, (left) and its corresponding barcode (right). There is no 1-homology because it closes before the persistence clears the maximum edge weight.}
    \label{fig:weighted_consistent_oct_largeEge}
\end{figure}

\subsubsection{Inconsistently Oriented Cycle Graphs}

Now that we've characterized the persistence of consistently oriented cycles, we can use some of the same tools to describe the persistence of inconsistently oriented cycles. In a consistently oriented cycle, the total symmetry is reflected in the fact that every vertex forms a dominating set. This is not the case for inconsistently oriented cycle graphs. In these graphs, minimal dominating sets are uniquely the set of sources in the graph; that is vertices whose finite weight edges point outward. We leverage these unique minimal dominating sets to describe reductions with the nerve. There are three cases to consider. The first is if there are more than two elements of the dominating set, followed by two, then one. Addressing the first case, we show in Proposition \ref{prop:inconcycledom3} when in the filtration the first homology gains a nontrivial element, and that the graph is not path connected enough to close it.

\begin{figure}
    \centering
    \includegraphics[width=0.8\linewidth]{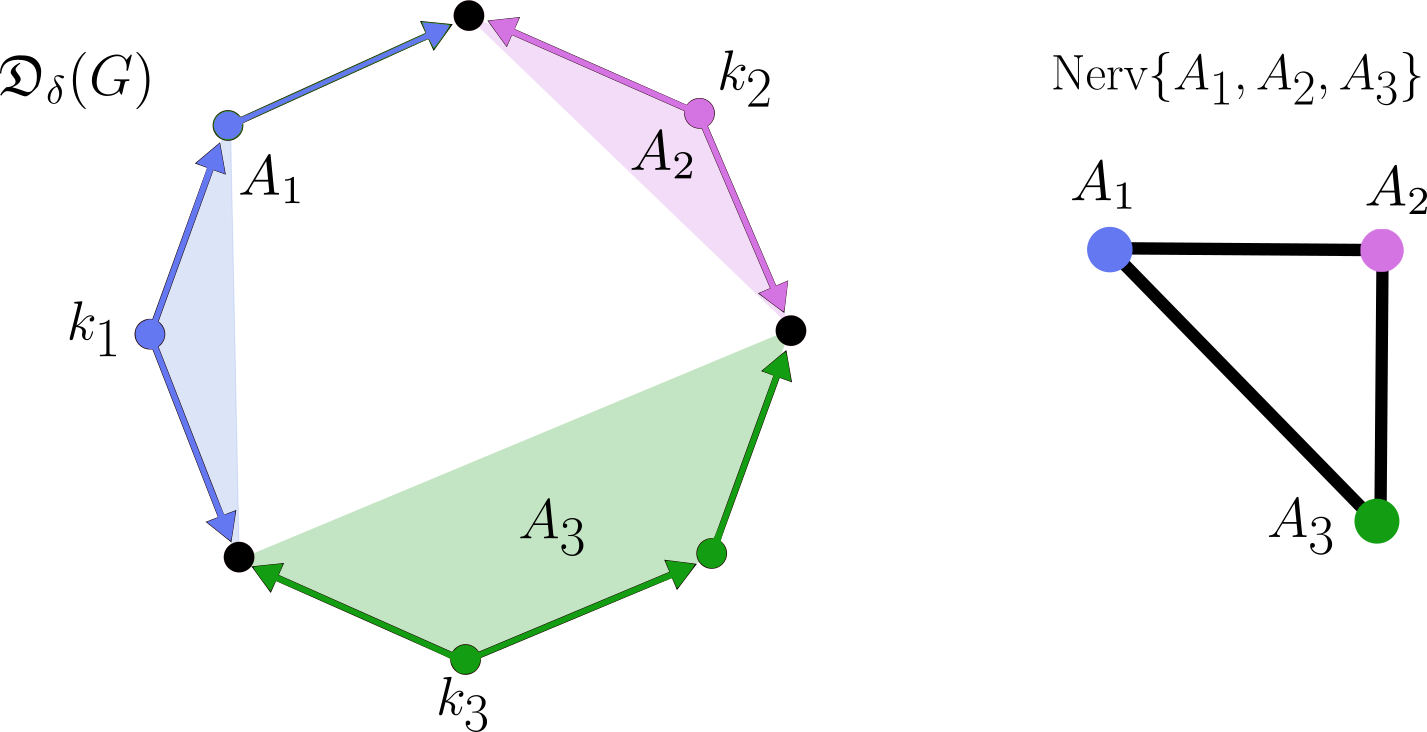}
    \caption{This figure exemplifies the argument of Proposition \ref{prop:inconcycledom3}. On the left we see a cycle graph overlayed with the simplices of its associated Dowker complex at $\delta$. Using the three colors we depict the cover $\{A_1,A_2,A_3\}$. On the right the associated nerve is shown, exemplifying their shared homology.}
\end{figure}

\begin{prop}\label{prop:inconcycledom3}
    Let $G$ be an inconsistently oriented cycle graph. Let $K$ be a minimal dominating set of $P(G)$. If $|K|\geq 3$ then the $1$st homology class of $\FD_\delta(G)$ has one nontrivial element for all $\delta\geq\max_i\omega(i,i+1 \pmod{n})$, and trivial homology otherwise. 
\end{prop}

\begin{proof} 
Below the lower bound, the graph $G$ does not contain a cycle with edge weights less than $\delta$, ignoring orientation. By Proposition \ref{prop:hom_to_bdry} it follows the complex has trivial first homology in this case. This Proposition also guarantees that there is at most one first homology class.

Let $\delta\geq\max_i\omega(i,i+1 \pmod{n})$. For all $k\in K$, let $A_k =\cup_{i\in \Gamma_k} U_i$ where $\Gamma_k=\{i | \omega_{P(G)}(k,i)<\infty\}$ is the set of vertex indices which vertex $k$ sees with finite weight. The set $A_k$ is connected since $\delta$ is greater than the maximum sequential edge weight, and if $k$ sees a vertex, there exists a path with edges of weight less than $\delta$ to it in $G$. 

The set $\{A_k\}_{k\in K}$ forms a cover by the fact that $K$ is a dominating set. Since $|K|\geq 3$ it follows that each $A_k$ is only incident to the cover elements on either side through a sink vertex (the orange vertices in Figure \ref{fig:weighted_everyOther}). Thus Nerv$\{|A_k| | k\in K\}$ is a cycle of $1-$simplices with $|K|$ vertices. By the nerve theorem, this is homotopy equivalent to $\FD_\delta(G)$, concluding the proof.

\end{proof}

\begin{figure}[t]
    \begin{minipage}{0.5\textwidth}
        \centering
        \includegraphics[width=0.8\linewidth]{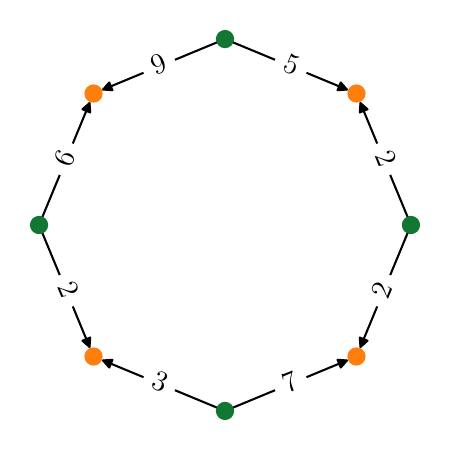}
    \end{minipage}%
    \begin{minipage}{0.5\textwidth}
        \centering
        \includegraphics[width=\linewidth]{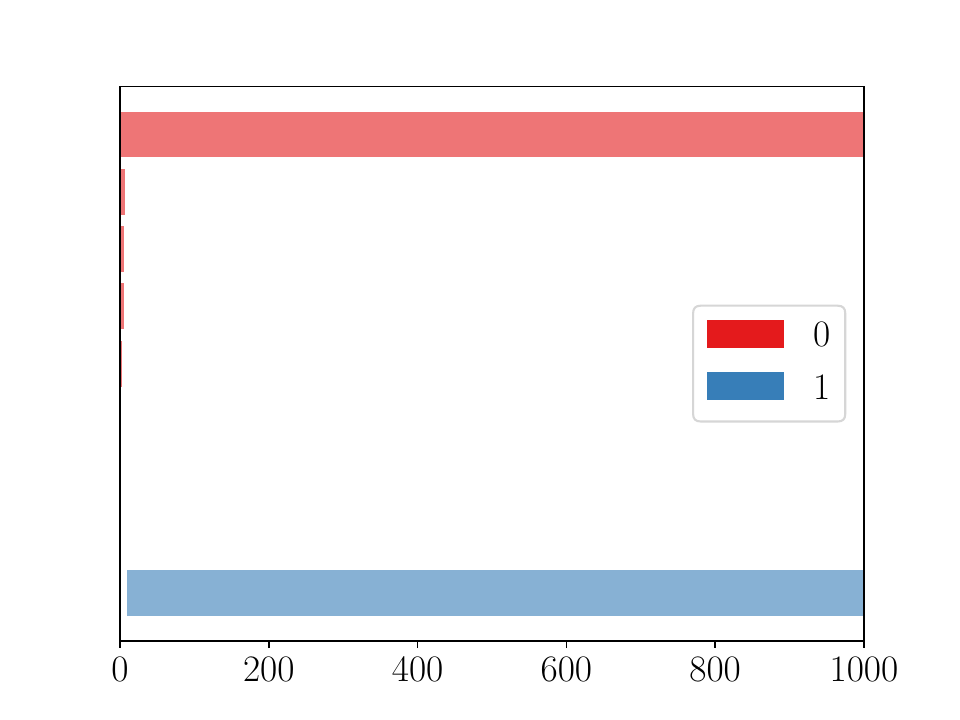}
    \end{minipage}
    \caption{Example of a weighted graph with alternating sources and sinks (left) and its corresponding barcode (right). The source and sink nodes are indicated in green and orange, respectively, in addition to the appropriate arrow orientations. After it is born, the maximal complex persists indefinitely because the graph is not well connected enough to close it with such a large minimal dominating set.}
    \label{fig:weighted_everyOther}
\end{figure}

In Figure \ref{fig:weighted_everyOther}, the 1D-homology never dies because there are at least three elements in the smallest dominating set. Proposition \ref{prop:inconcycledom3} tells us that with more than three dominating set elements, the cycle graph is too disconnected to ever close the first homology class. When the dominating set is less than $3$, the first homology is able to close, but we are still able to leverage the unique minimal dominating set to reduce the complex. We start with the size $2$ dominating set.

\begin{figure}
    \centering
    \includegraphics[width=0.85\linewidth]{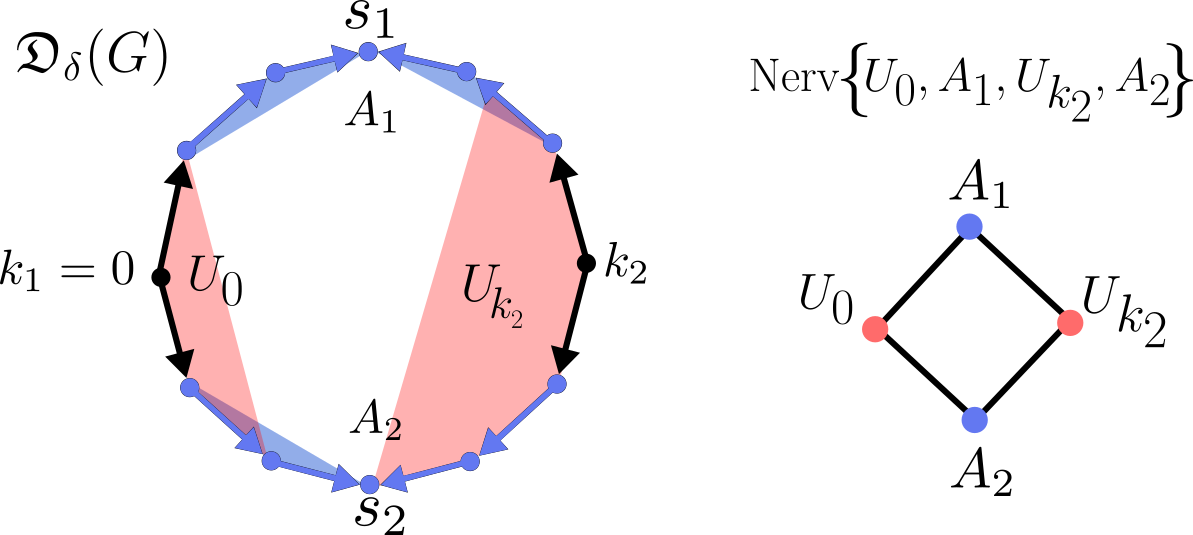}
    \caption{This figure exemplifies the argument in the proof of Proposition \ref{prop:inconcycledom2}. On the left we see a cycle graph dominating set size $2$, overlayed with the simplices of its associated Dowker complex at $\delta$. On the left is the cycle graph $G$ with overlayed shaded regions corresponding to the maximal simplices of $\FD_\delta(G)$. The vertices $k_1,k_2$ form a minimal dominating set of $P(G)$. The red color distinguishes $U_0$ and $U_{k_2}$ from the other elements of the cover $A_1,A_2$, whose components are in blue. On the right Nerv$\{U_0,A_1,U_{k_2},A_2\}$ is depicted, exemplifying the reduction.}
\end{figure}

\begin{prop}\label{prop:inconcycledom2}
    Let $G$ be an inconsistently oriented cycle graph. Let $K$ be a minimal source dominating set of $P(G)$. If $K=\{k_1,k_2\}$ then the $1$st homology class of $\FD_\delta(G)$ has one nontrivial element for all $\max_i\omega(i,i+1\pmod{n})\leq\delta\leq \max\{\omega(k_1,s_1),\omega(k_1,s_2),\omega(k_2,s_1),\omega(k_2,s_2)\}$ where $s_1,s_2$ are the sink vertices.
\end{prop}

\begin{proof}
    Following the assumption, let $K=\{k_1,k_2\}$. Since $G$ is a cycle graph it follows that there must exist two sink vertices, $s_1,s_2$, which are each seen by $k_1$ and $k_2$. Without loss of generality, assume that $0=k_1<s_1<k_2<s_2$.

    Let $A_1=\cup_{0<i<k_2}U_i$ and $A_2=\cup_{k_2<i\leq n}U_i$. We show that $\FD_\delta(G)$ is homotopy equivalent to Nerv$\{U_0,A_1,U_{k_2},A_2\}$, and that this nerve has a nonbounding cycle precisely when $\delta$ lies in the specified region.

    By construction, we know that $U_0\cup A_1\cup U_{k_2}\cup A_2=\cup_{0\leq i\leq n}U_i=\FD_\delta(G)$, implying that this set forms a proper cover.

    by construction $A_1$ and $A_2$ are disjoint because all vertices $0<i<k_2$ have no directed paths to the vertices $k_2<j\leq n$, and vice versa. At the same time $U_0\cap U_{k_2}\subseteq\text{Cl}([s_1,s_2])$ because $s_1$ and $s_2$ are the only vertices which both source vertices can see in $P(G)$ for any positive weight.

    If we consider another intersection, we see
    \begin{align*}
        U_0\cap A_1 
        &=
        U_0\cap \left(\cup_{0<i<k_2}U_i\right)
        \\
        &=
        \cup_{0<i<k_2} 
        \left(
        U_0\cap U_i
        \right)
        \\
        &=
        \cup_{0<i<s_1} 
        \left(
        U_0\cap U_i
        \right)\\
        &=U_0\cap U_1
    \end{align*}
    because the union of out neighborhood 
    intersections is over a directed path, implying its components are nested simplices. This resulting term is a simplex and thus contractible. Similar reasoning finds that $U_0\cap A_2$, $U_{k_2}\cap A_1$, and $U_{k_2}\cap A_2$ are contractible. The lower bound on $\delta$ also guarantees that they are nonempty.
    
    By the nerve theorem, we get the desired homotopy equivalence. From the nonempty intersections described in the previous paragraph, we find that this nerve contains a $4$ vertex cycle. Whether this cycle is filled with two simplices or not depends on the other intersections. Since $A_1\cap A_2=\emptyset$, it follows that $A_1\cap A_2\cap U_0=A_1\cap A_2\cap U_{k_2}=\emptyset$. Thus the $4-$cycle in the nerve is filled if and only if $A_1\cap U_{k_2}\cap U_0$ and $A_2\cap U_{k_2}\cap U_0$ are nonempty. from our previous conclusions from the structure of the underlying graph this means $A_1\cap U_{k_2}\cap U_0=[s_1]$ and $A_2\cap U_{k_2}\cap U_0=[s_2]$. This implies $\omega(0,s_1),\omega(0,s_2),\omega(k_2,s_1),\omega(k_2,s_2)\leq\delta$. Thus we may conclude that the complex has a nontrivial element in the first homology if and only if
    \begin{align*}
        \max_{0\leq i\leq n}\omega(i,i+1\pmod{n})
        &\leq 
        \delta
        <
        \max\{\omega(0,s_1),\omega(0,s_2),\omega(k_2,s_1),\omega(k_2,s_2)\}.
    \end{align*}
    Since the source vertices are the vertices furthest away from the sinks, the above reduces to the desired inequality.    
\end{proof}

\begin{figure}[t]
    \begin{minipage}{0.5\textwidth}
        \centering
        \includegraphics[width=0.8\linewidth]{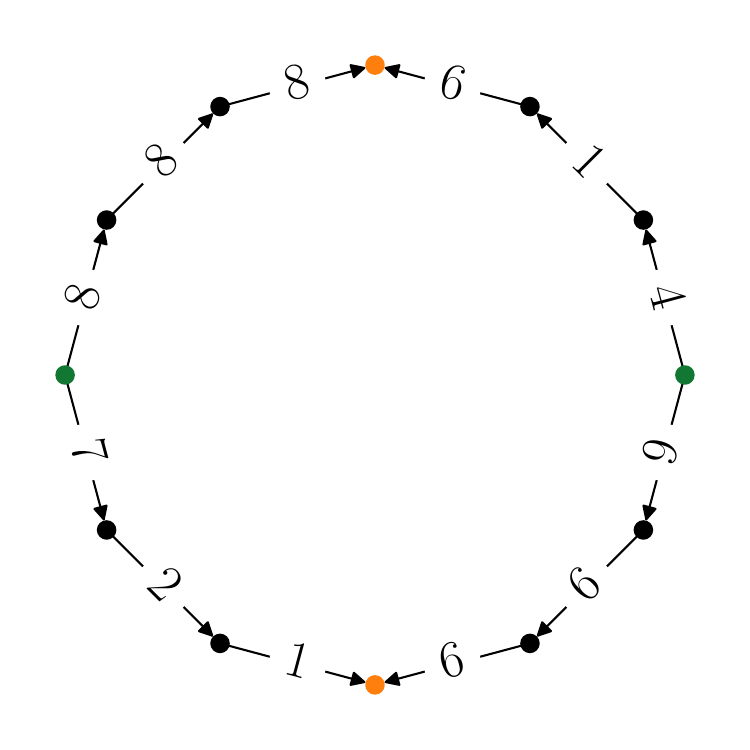}
    \end{minipage}%
    \begin{minipage}{0.5\textwidth}
        \centering
        \includegraphics[width=\linewidth]{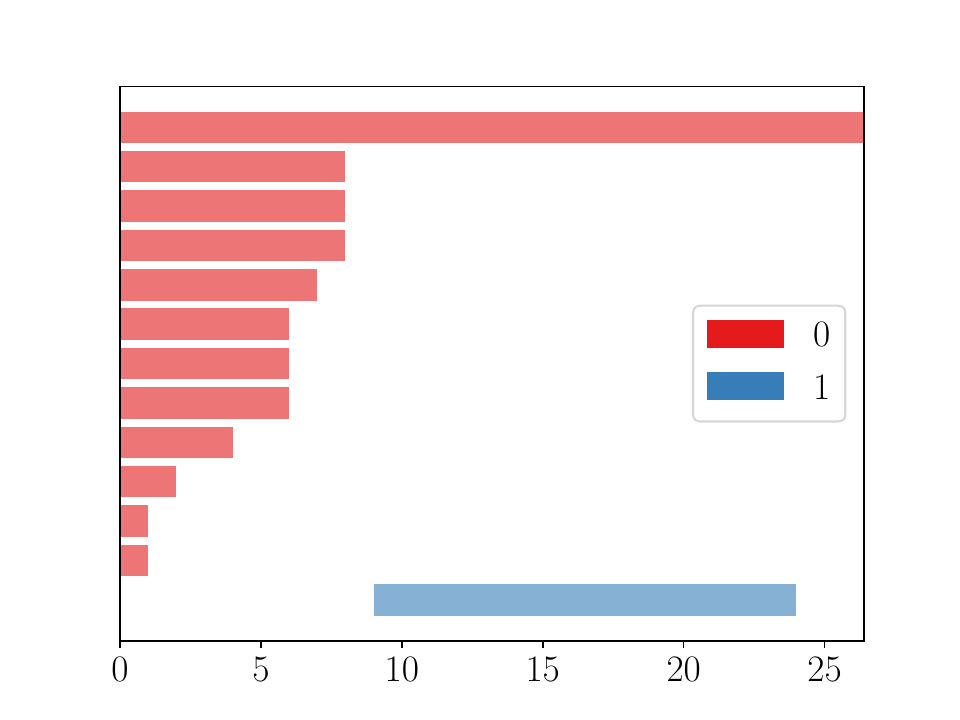}
    \end{minipage}
    \caption{Example of a weighted 12-vertex graph with two sources and two sinks (left) and its corresponding barcode (right). The source and sink vertices are indicated in green and orange, respectively, in addition to the appropriate arrow orientations. The first homology class is born at the maximum edge weight of $9$, and dies at the maximum weight of a path from source to sink, $24$. In this case, the top left quarter of the cycle.}
    \label{fig:weighted_twoSource}
\end{figure}


The separation of the sink and source vertices allows for the above analysis. See the results exemplified in Figure \ref{fig:weighted_twoSource}. In the case of a single vertex dominating set, there are two directed paths from a source vertex to a sink vertex, which requires more case work.

\begin{figure}
    \centering
    \includegraphics[width=0.8\linewidth]{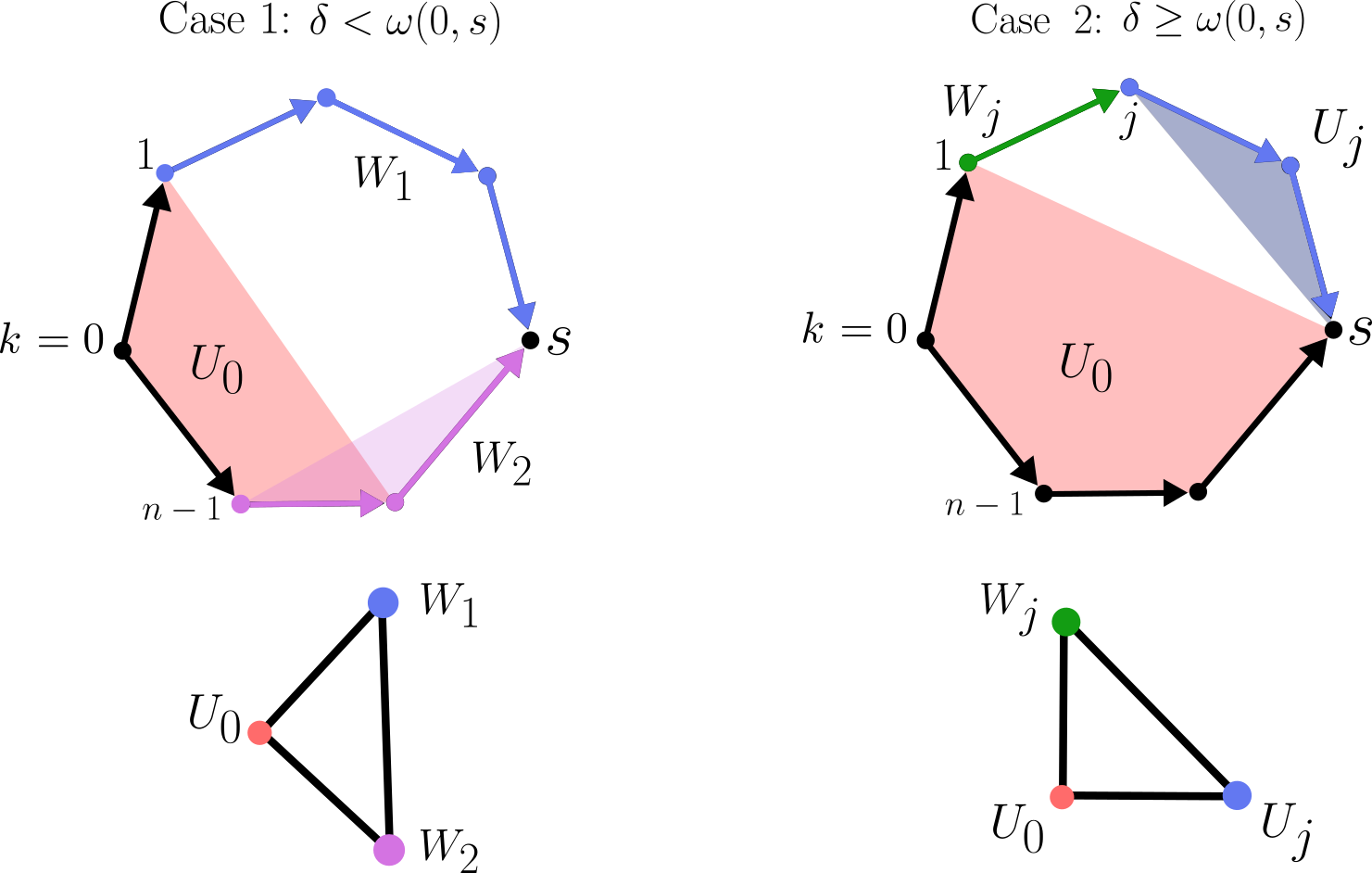}
    \caption{This figure exemplifies the two cases presented in the proof of Proposition \ref{prop:inconcycledom1}. \textbf{Case 1} shows when no weighted path from $k=0$ to $s$ exists with weight less than $\delta$. In this case the nerve is constructed with the cover $\{U_0,W_1,W_2\}$. \textbf{Case 2} shows the case in which one directed path from $k=0$ is shorter than the other, leaving the homology until both halves close. In this case $U_j$ and $W_j$ are defined create the cover $\{U_0,U_j,W_j\}$ whose nerve is depicted below.}
\end{figure}

\begin{prop}\label{prop:inconcycledom1}
    Let $G$ be an inconsistently oriented cycle graph. Let $K$ be a minimal source dominating set of $P(G)$. If $|K|= 1$ then there exists one source $k\in K$ and one sink $s$. In this case the $1$st homology class of $\FD_\delta(G)$ has one nontrivial element for all $\max_i\omega(i,i+1\pmod{n})\leq\delta\leq\max\{\min_{i<s}\max\{\omega(k,s),\omega(k,i),\omega(i,s)\}
    ,
    \min_{s<i<n}\max\{\omega(k,s),\omega(k,i),\omega(i,s)\}\}$.
\end{prop}

\begin{proof}
    Since $G$ is a cycle graph with one source and one sink, it must consist of two directed paths joining vertices $k$ and $s$ which are disjoint outside of their endpoints. By symmetry, assume without loss of generality that $k=0$. We proceed with cases.

    (Case 1)
    Suppose $\delta<\omega(k,s)$. Let $W_1=U_1\cup \cdots\cup U_{s-1}$ and $W_2=U_{n-1}\cup \cdots \cup U_{s+1}$. We prove that $\FD_\delta(G)$ is homotopy equivalent to Nerv$\{W_1,W_2,U_0\}$.

    By construction and the assumption on $\delta$,  $W_1\cap W_2=[s]$. Since $s$ is a source vertex, it has a trivial out neighborhood. Thus $W_1\cup W_2\cup U_0=\cup_{0\leq i\leq n}U_0$ suggesting that these three sub-complexes form a cover of the complex.

    From $\delta<\omega(k,s)$, it follows that $s\not\in U_o$. Thus $W_1\cap W_2\cap U_0=\emptyset$. At the same time, our assumption on $\delta$ gives that $U_0\cap W_1$ and $U_0\cap W_2$ are nonempty. Also
    \begin{align*}
        U_0\cap W_1 &=  U_0\cap (U_1\cup\dotsc\cup U_{s-1})\\
        &=\left(U_0\cap U_1\right) \cup \dotsc \cup \left(U_0\cap U_{s-1}\right)\\
        &=U_0\cap U_1
    \end{align*}
    since these out neighborhood intersections along a directed path are nested. It follows that this intersection is a simplex and therefore contractible. The same reasoning follows for $W_2$. 

    By the Nerve theorem, it follows that $\FD_\delta(G)$ is homotopy equivalent to Nerv$\{W_1,W_2,U_0\}$. Further, we've shown that this nerve contains $3$ vertex cycle and since $W_1\cap W_2\cap U_0=\emptyset$ it follows that it is not filled in and therefore not a boundary.

    (Case 2) 
    Suppose $\delta\geq\omega(k,s)$. If there does not exist an $i$ such that $\max\{\omega(k,s),\omega(k,i),\omega(i,s)\}\neq \omega(k,s)$, then following the analysis of the previous case, we find the resulting nerve is a filled $2-$simplex, and therefore has no nontrivial $1-$homology.
    
    If there exists $i$ such that $\max\{\omega(k,s),\omega(k,i),\omega(i,s)\}> \omega(k,s)$, then it follows that the weight sum of one path from $k$ to $s$ is smaller than that of the other. Thus for any $i$ along the shorter path, $\max\{\omega(k,s),\omega(k,i),\omega(i,s)\}= \omega(k,s)$ by construction of the weight $\omega(k,s)$. Thus $i$ is a vertex along the longer path, and by symmetry let us assume that it is some $i<s$. That is, we assume the longer path from $k$ to $s$ encompasses the vertices $i<s$. 
    
     Let $j<s$ be the vertex on the longer path which maximizes $\omega(j,s)$ less than $\delta$. If $j\geq 1$, let $W_j=U_1\cup \cdots \cup U_{j-1}$, and otherwise let $W_j=[j]$. We show homotopy equivalence to Nerv$\{U_0,W_j,U_j\}$. Notice that by assumption the vertices of the shorter path, $i\geq s$, are in $U_0$. Following the same reasoning as the previous case in regard to the intersection with a union of out neighborhoods along a path, we find that $W_j$ intersects contractibly with $U_0$ and $U_j$.

    If $U_0\cap W_j\cap U_j=\emptyset$ then again by the Nerve theorem the complex is homotopy equivalent to Nerv$\{U_0,W_j,U_j\}$. We have also shown that this nerve is a triangle not filled with a $2-$simplex, admitting a single nontrivial element in the the first homology.
    
    Now we show that $\delta$ is below our upper bound in this case. Because the intersection is empty it follows that $j\not\in U_0$. Thus $\omega(k,j)>\delta$. It follows that $\max\{\omega(k,s),\omega(k,j),\omega(j,s)\}=\omega(k,j)>\delta$. Since  $ \omega(k,i)>\omega(k,j)$ for all $j<i<s$, it follows that
    \begin{align*}
        \max\{\omega(k,s),\omega(k,i),\omega(i,s)\}&>\max\{\omega(k,s),\omega(k,j),\omega(j,s)\}>\delta
    \end{align*}
    for all $j<i\leq s$. We conclude that a minimizing vertex $m$ such that
    \begin{align*}
        \min_{0<i<s}\max\{\omega(k,s),\omega(k,i),\omega(i,s)\}&=\max\{\omega(k,s),\omega(k,m),\omega(m,s)\}
    \end{align*}
    must have index less than $j$, meaning $m<j$. However, in this case $\omega(m,s)>\omega(j,s)$, and $j$ was chosen so that $\omega(j,s)$ is maximal, meaning that $\omega(m,s)>\delta$. It then follows 
    \[
    \min_{0<i<s}\max\{\omega(k,s),\omega(k,i),\omega(i,s)\}>\delta.
    \]
     and so $\delta$ is less than the proposed upper bound.
    
     If $U_0\cap W_j\cap U_j\neq \emptyset$ then $j\in U_0$. It follows $\omega(k,s),\omega(k,j),\omega(j,s)\leq \delta$, which implies that $\delta\geq \max\{\omega(k,s),\omega(k,j),\omega(j,s)\}$. Thus $\delta\geq \min_{0<i<s}\max\{\omega(k,s),\omega(k,i),\omega(i,s)\}$. We concluded before that for vertices on the shorter path, $\min_{s<i<n}\max\{\omega(k,s),\omega(k,i),\omega(i,s)\}=\omega(k,s)$. Thus $\min_{s<i<n}\max\{\omega(k,s),\omega(k,i),\omega(i,s)\}\leq \min_{0<i<s}\max\{\omega(k,s),\omega(k,i),\omega(i,s)\}\leq \delta$ which implies $\delta$ is greater than the proposed upper bound.
\end{proof}

\begin{figure}[t]
    \begin{minipage}{.5\textwidth}
        \centering
        \includegraphics[width=0.8\linewidth]{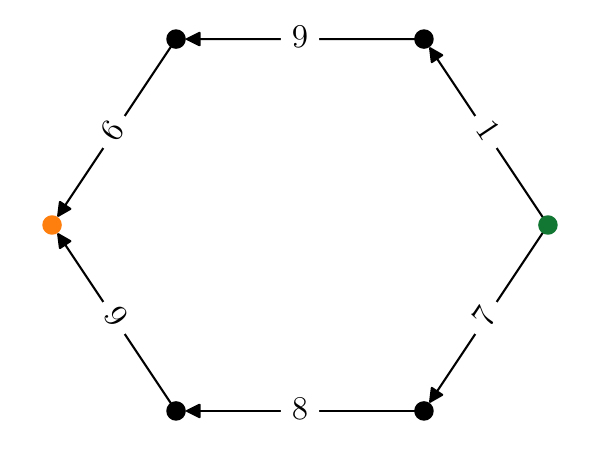}
    \end{minipage}%
    \begin{minipage}{0.5\textwidth}
        \centering
        \includegraphics[width=\linewidth]{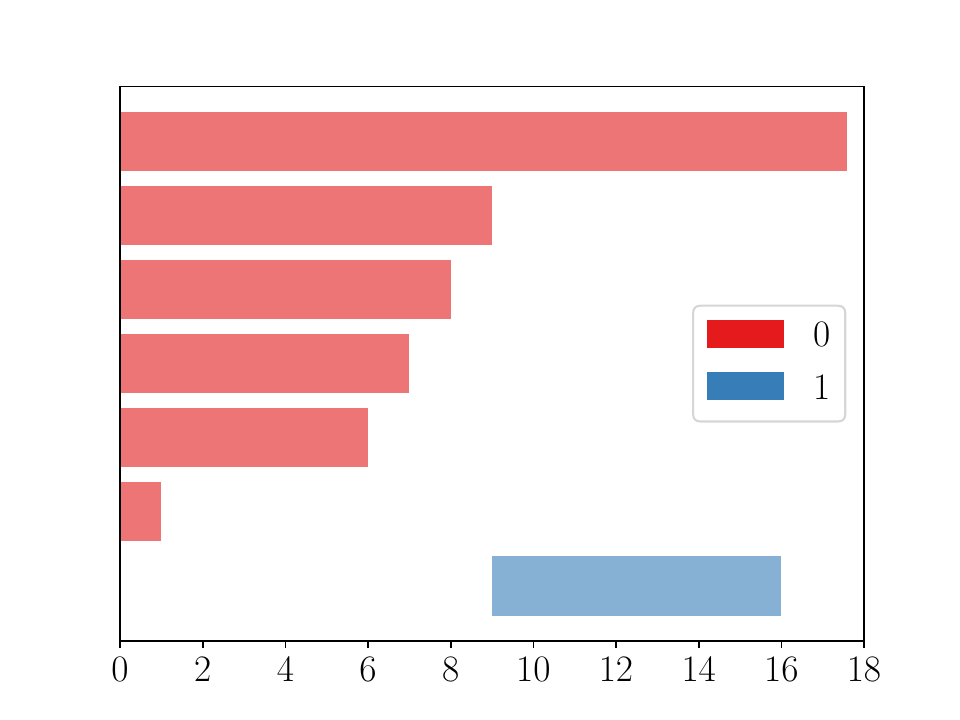}
    \end{minipage}
    \caption{Example of a weighted graph with one source and one sink (left) and its corresponding barcode (right). The source and sink nodes are indicated in green and orange, respectively, in addition to the appropriate arrow orientations. We see the first homology class born at the maximum edge weight, $9$, and die when the source vertex has seen the remaining vertices of the graph, at $16$.}
    \label{fig:weighted_1src1sink}
\end{figure}

Proposition \ref{prop:inconcycledom1} is exemplified in Figure \ref{fig:weighted_1src1sink}. Together, Propositions \ref{prop:inconcycledom1}, \ref{prop:inconcycledom2}, \ref{prop:inconcycledom3}, and \ref{prop:concycle} fully characterize the Dowker persistence of the first homologies of all cycle graphs. From this analysis we find that the Dowker complex reflects a number of key features from the underlying cycle graph. Cycle graphs with large minimal dominating sets are not connected enough for a Dowker complex placed on top to close an induced first homology. In the case of an oriented cycle graph, the nontrivial element of the first homology group corresponds with directed graph cycles in the underlying graph and closes when $\delta$ allows the shortest cycle to be of length $2$. In the remaining cases the dominating sets allow for the first homology class to remain only transiently, but the persistence behaves differently from the  consistently oriented case.

\subsection{Homologies of Wedge Sums}\label{sec:wedge_persistence}

Now that we have fully characterized the persistence of cycles in the path completion, we consider a generalization in the form of wedge sums. In continuous space, the wedge sum of two topological spaces is their union with one point from each identified. Formally, if $A$ and $B$ are simplicial complexes (topological spaces) with distinguished vertices (points) $a\in A$ and $b\in B$, their wedge sum is defined as,  $A\vee B = A\sqcup B/ \sim$ with the equivalence $a\sim b$. Similarly, a wedge sum of two graphs is their union with distinguished vertices from each identified. In an effort to better understand the persistence of multi cycle dynamical manifolds (such as the Lorenz system discussed in Section \ref{sec:experiments}), we turn our attention to the persistence of the wedge sum of graphs.

\begin{theorem}\label{thm:wedge_sum}
    Let $G_1$ and $G_2$ be directed graphs with distinguished vertices $x_0$ and $y_0$ respectively. Let $\delta\in\mathbb{R}$. Then $\tilde{H}_n(\FD_\delta(G_1\vee G_2))\approx \tilde{H}_n(\FD_\delta(G_1))\oplus \tilde{H}_n(\FD_\delta(G_2))$ when the graph wedge sum is formed using vertices $x_0,y_0$ that share a $\delta$ neighborhood with at least one other vertex. That is, the reduced Dowker homology of the path completion of the wedge sum of two graphs, is equivalent to the direct sum of the Dowker homologies separately.
\end{theorem}

In order to prove this result, we require two intermediate Lemmas.
\begin{lem}\label{lem:path_wedge}
    Let $G_1$ and $G_2$ be directed graphs with distinguished vertices $x_0$ and $y_0$ respectively. Let $\delta\in\mathbb{R}$. It follows that  $\tilde{H}_n(\FD^\star_\delta(P(G_1\vee G_2)))
    \approx
    \tilde{H}_n(\FD^\star_\delta(P(G_1)\vee P(G_2)))
        $.
\end{lem}
\begin{proof}
    Let $A= \text{Cl}(\FD^\star_\delta(P(G_1\vee G_2))\setminus \FD^\star_\delta(P(G_1)\vee P(G_2)))$ and $B=\FD^\star_\delta(P(G_1)\vee P(G_2))$.

    Our first claim is that $A$ is contractible. Notice that the maximal simplices of $A$ are not in $\FD^\star_\delta(P(G_1)\vee P(G_2))$. Thus they require edges between $G_1\setminus x_0$ and $G_2\setminus x_0$ in $P(G_1\vee G_2)$. Since any path in $G_1\vee G_2$ inducing such an edge must pass through $x_0$, it follows that the maximal simplices of $A$ all contain $x_0$. This induces a free face of every maximal simplex, which we may collapse with respect to. That is a maximal simplex $\sigma=[\sigma_1,\cdots,x_0,\cdots,\sigma_n]$ implies $ [\sigma_1,\cdots,\hat{x_0},\cdots,\sigma_n]$ is free. Consider the complex $A'$ where we perform a simplicial collapse with respect to this free face of every maximal simplex. It follows that the maximal simplices of $A'$ also contain $x_0$. Thus we may inductively collapse the maximal simplices until only $x_0$ remains. This shows the desired contractibility.

    Our second claim is that $A\cap B$ is contractible. This is precisely simplices returned by the closure in the definition of $A$. Let $\sigma=[\sigma_1,\dotsc,\sigma_n]$ be a maximal simplex of $A$, induced by a path from $\sigma_1\in V(G_1)\setminus x_0$ to $\sigma_n\in v(G_2)\setminus x_0$. This simplex implies $\eta_1=[\sigma_1,\dotsc,x_0]$ and $\eta_2=[x_0,\dotsc,\sigma_n]$ in $A\cap B$ representing paths to and from $x_0$, respectively. Thus all maximal simplices of $A\cap B$ contain $x_0$. Repeating the same argument as for $A$, we find that $A\cap B$  is contractible.

    Given that $\tilde{H}_n(A)=\tilde{H}_n(A\cap B)=0$, it follows by Mayer Vietoris that the following sequence is exact.

    \[
    0
    \longrightarrow
    \tilde{H}_n(B)
    \longrightarrow
    \tilde{H}_n(A\cup B)
    \longrightarrow
    0
    \]

    This implies the desired isomorphism

    \[
    \tilde{H}_n(\FD^\star_\delta(P(G_1)\vee P(G_2)))
    \approx
    \tilde{H}_n(\FD^\star_\delta(P(G_1\vee G_2))).
    \]

\end{proof}

\begin{lem}\label{lem:graph_sum}
    Let $P_1$ and $P_2$ be directed graphs with distinguished vertices $x_0$ and $y_0$ respectively. Let $\delta\in\mathbb{R}$. It follows $\tilde{H}_n(\FD^\star_\delta(P_1\vee P_2))
    \approx
    \tilde{H}_n(\FD^\star_\delta(P_1))\oplus \tilde{H}_n(\FD^\star_\delta(P_2))
    $  when the graph wedge sum is formed using vertices $x_i\in P_i$ that share a $\delta$ neighborhood with at least one other vertex.
\end{lem}
\begin{proof}
    Notice that $\FD^\star_\delta(P_1),\FD^\star_\delta(P_2)\subseteq \FD^\star_\delta(P_1\vee P_2)$. It follows that $\FD^\star_\delta(P_1)\vee \FD^\star_\delta(P_2)\subseteq  \FD^\star_\delta(P_1\vee P_2)$, where $\vee$ denotes the simplicial wedge sum. We assume the simplicial and graph wedge sums identify the same vertex, $x_0$, and that that vertex is not isolated for our choice of $\delta$.

    By construction $x_0$, in the graph $P_1\vee P_2$, is the only vertex with neighbors in $P_1$ and $P_2$, when they are seen as subsets of their wedge sum. Otherwise connectivity is unchanged. It follows that $\FD^\star_\delta(P_1\vee P_2)
    =
    \FD^\star_\delta(P_1)\vee \FD^\star_\delta(P_2) \cup N_\delta(x_0)$ where $N_\delta(x_0)$ is the $\delta$ neighborhood simplex of $x_0$.

    Notice now that the intersection $(\FD^\star_\delta(P_1)\vee \FD^\star_\delta(P_2)) \cap N_\delta(x_0)=(\FD^\star_\delta(P_1)\cap N_\delta(x_0)\vee (\FD^\star_\delta(P_2)\cap N_\delta(x_0)
    )$ since the base point $x_0$ is in the intersection. This is the wedge of the $\delta$ neighborhood simplices of $x_0$ in the component complexes separately, and is clearly contractible. 

    Since the intersection is contractible, through a standard Mayer-Vietoris argument we can show that 
    $\tilde{H}_n(\FD^\star_\delta(P_1)\vee \FD^\star_\delta(P_2) \cup N_\delta(x_0)))
    \approx
    \tilde{H}_n(\FD^\star_\delta(P_1)\vee \FD^\star_\delta(P_2))$.

    Notice that the assumption on the base points guarantees that $x_0$ is not an isolated point in the complex (in particular, that the basepoints form good pairs with their parent complex). It follows from a standard result that $\tilde{H}_n(\FD^\star_\delta(P_1)\vee \FD^\star_\delta(P_2))\approx \tilde{H}_n(\FD^\star_\delta(P_1))\oplus \tilde{H}_n(\FD^\star_\delta(P_2))$(e.g. Corollary $2.25$ in \cite{Hatcher_AT}).
\end{proof}

\begin{proof}[Proof of Theorem \ref{thm:wedge_sum}]
    The result follows immediately after applying Lemma \ref{lem:path_wedge} then Lemma \ref{lem:graph_sum}.
\end{proof}

\begin{figure}[t]
    \begin{minipage}{.5\textwidth}
        \centering
        \includegraphics[width=0.8\linewidth]{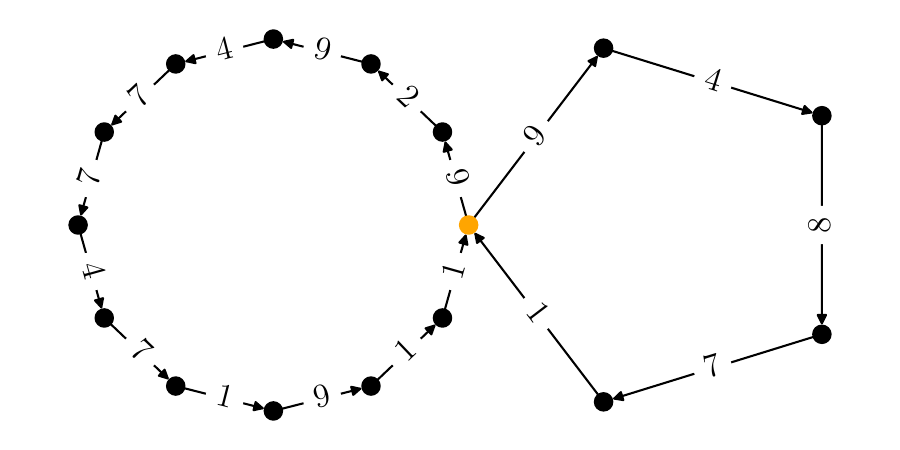}
    \end{minipage}%
    \begin{minipage}{0.5\textwidth}
        \centering
        \includegraphics[width=\linewidth]{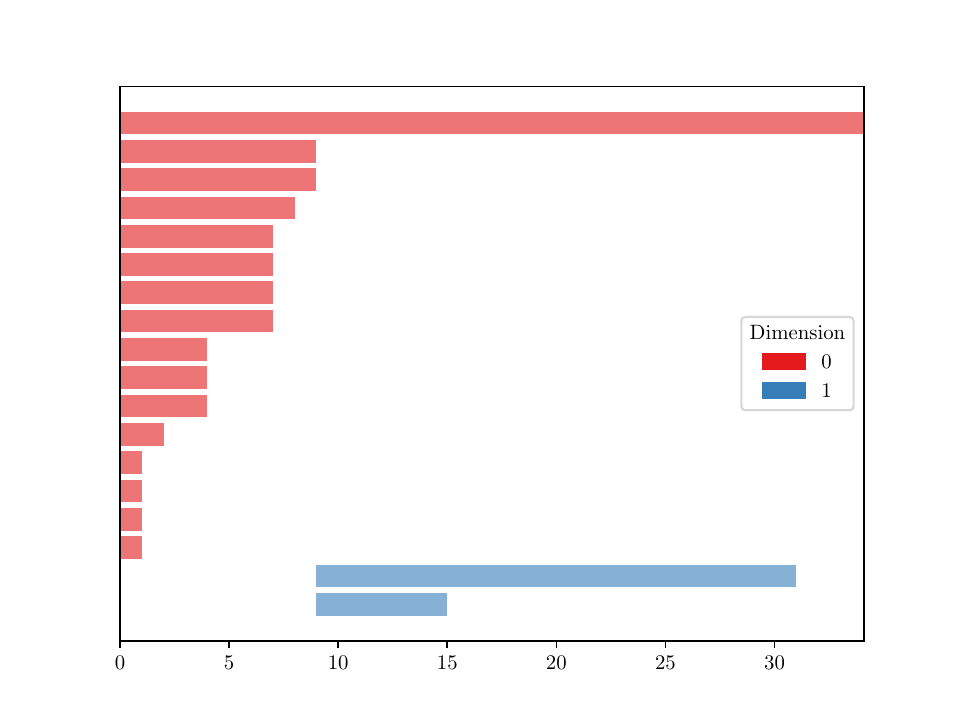}
    \end{minipage}
    \caption{Example of the Dowker persistence of the wedge sum of two cycle graphs. The two 1-D bars correspond exactly to the expected values of the left cycle (longer bar) and the right cycle (shorter bar).}
    \label{fig:wedge}
\end{figure}

This result provides a bridge between wedge sums in the underlying graph topology, and the resulting Dowker homology. 

Figure \ref{fig:wedge} shows an example of a wedge sum of two consistently oriented cycle graphs and the corresponding Dowker persistence. Both 1-D homology classes are born at $9$ because this is highest edge weight in both cycles. The homology class corresponding to the right cycle dies at $15$, while the homology class corresponding to the left cycle dies at $31$. This follows the expectation from Proposition \ref{prop:concycle}. This test case provides a rough approximation of the persistence of a system with two attracting limit cycles like the Lorenz system (Section \ref{sec:experiments}).

Notice that Theorem \ref{thm:wedge_sum} suggests a wedge sum with something of trivial homology, such as the path graph on two vertices, does not change the homology. Thus a natural and expected corollary is that any directed graph without a cycle subgraph has trivial reduced homology. This class of graphs is called the directed acyclic graphs.

We can combine this with the results of Section \ref{sec:cycle_persistence}, characterizing cycle graph persistence. It follows from Theorem \ref{thm:wedge_sum} that the Dowker persistence of any graph which may be described as a sequence of wedge sums of a directed edge and cycle graphs may be described in terms of the persistence of those cycle graphs (Corollary \ref{cor:cactus_hom}). This class of graphs is the set of graphs on which cycles share at most one vertex. In the undirected case, these graphs are called  cactus graphs.  We refer to cactus graphs with oriented edges as directed cactus graphs.

\begin{cor}\label{cor:cactus_hom}
    Let $G$ be a directed cactus graph. Let $\{C_i\}_{i=1}^k$ be the unique set of cycle subgraphs of $G$. Then $\tilde{H}_n(\FD_\delta(G))\approx \oplus_{i=1}^k \tilde{H}_n(\FD_\delta(C_i))$.
\end{cor}

Thus the persistence of a directed cactus graph is the direct sum of the persistence of its cycles, the first homology of which we characterized in the previous sections. Figure \ref{fig:cactus} exemplifies this result, showing the persistence of a generic weighted directed cactus graph.

\begin{figure}[t]
    \begin{minipage}{0.5\textwidth}
        \centering
        \includegraphics[width=0.8\linewidth]{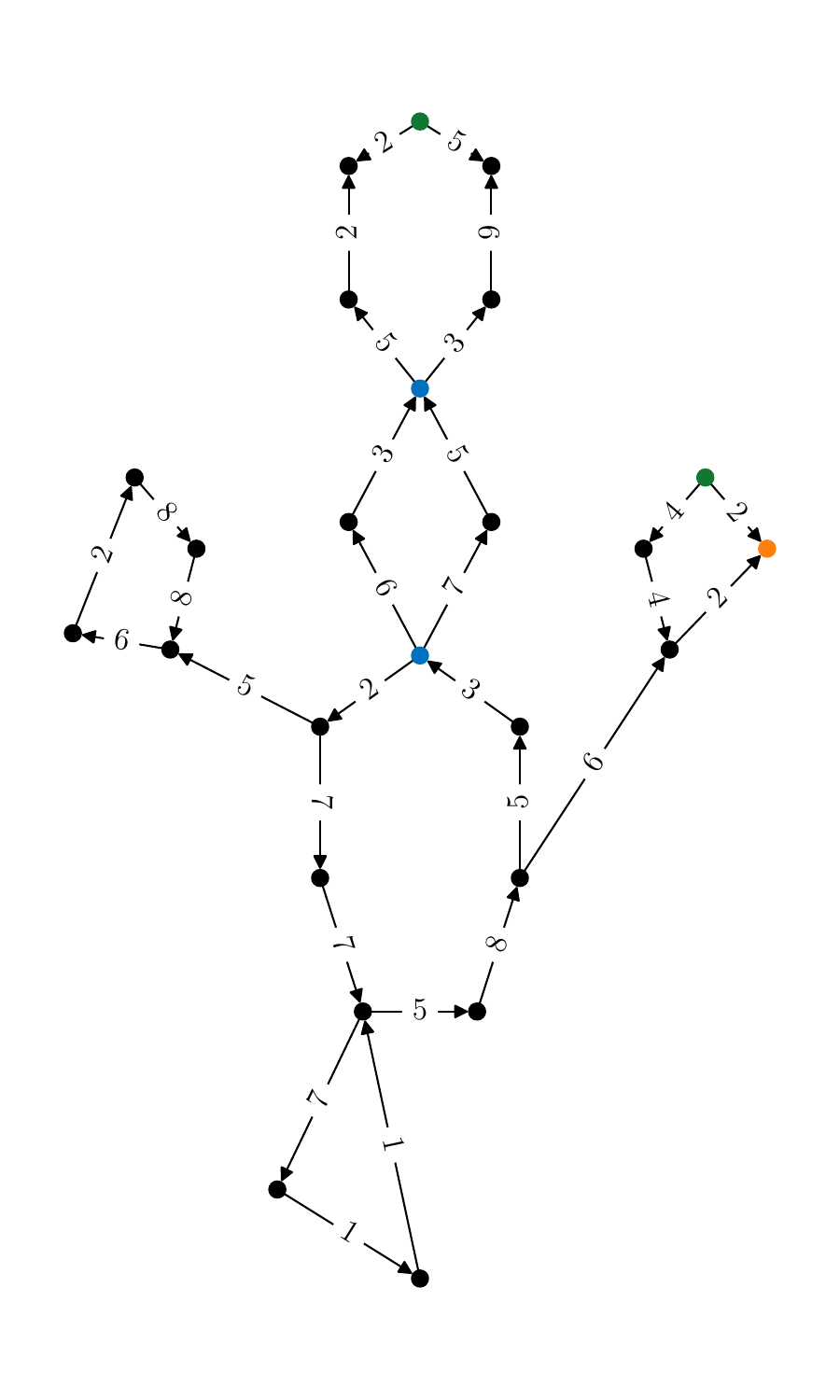}
    \end{minipage}%
    \begin{minipage}{0.5\textwidth}
        \centering
        \includegraphics[width=\linewidth]{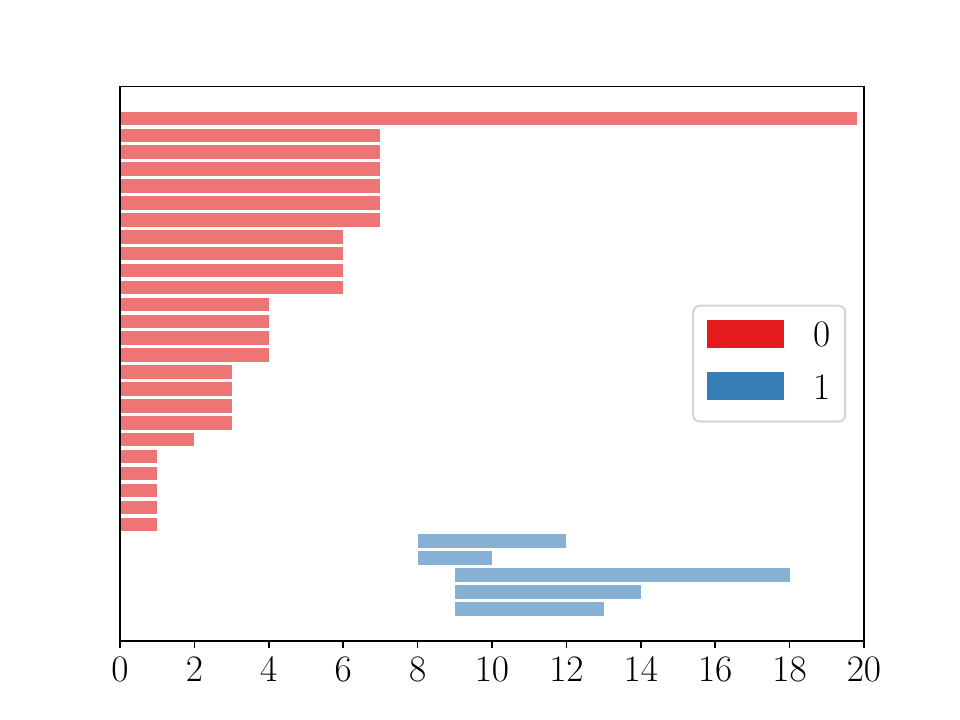}
    \end{minipage}
    \caption{An example of the Dowker persistence of a generic cactus graph. this graph has $6$ cycle subgraphs of various orientations. Notice the bottom-most cycle with three edges has no first homology at any $\delta$ because at its highest edge weight it is already a boundary. For this reason we see five first homologies with varying persistence lengths on the right, corresponding to the persistences of the five other cycles.}
    \label{fig:cactus}
\end{figure}

\subsection{Dynamical Trajectories}\label{sec:dyn_traj}

Our analysis this far has focused on the relationship between graph structure and a Dowker persistence framework built thereupon. Section \ref{sec:dom_set_persistence} investigated the relationship of Dowker persistence with dominating sets, and how weakly connected graphs may have indefinitely persisting features. 
In this section we formalize the binning procedure of our framework, creating a coarse grained structure from a dynamical trajectory. We then consider the properties of this graph and its persistence, given work from Section \ref{sec:dom_set_persistence}.

Assume a given point cloud is associated with a single dynamical trajectory, yielding a total order of its points. This allows us to describe them as the vertices of a directed path graph, i.e. a single directed path spanning all the vertices. Thus the spatial binning of a trajectory in phase space corresponds with a vertex partition of the corresponding directed path graph. 

Given a (di)graph, $G$, and vertex partition, $\pi$, we may define a quotient graph, $G^\pi$. The vertex set of this graph is the set of partition cells of $\pi$, and we add an edge between vertices if an edge between those partition cells exists in $G$. This graph represents the connectivity between partition cells of $\pi$, and can be weighted or unweighted. Weighted variants often add the number of connections between partition cells, but may also normalize edge weights by the number of vertices contained in each cell or total number of edges (making the adjacency stochastic). For the experiments in section \ref{sec:experiments} the unweighted quotient graph is used, and the results of this section apply to all quotient graphs.




We have found in our analysis that the Dowker filtration on the path completion of a digraph allows for the persistence to reflect the connectivity through shortest paths with total weight less than $\delta$. This centers directed paths in the connectivity of the graphs considered, and motivates understanding the relationship between paths in $G$ and $G^\pi$. 

\begin{prop}\label{prop:quotientwalk}
    Let $G$ be a network with a (vertex spanning) walk of finite length $W$. Let $\pi$ be a partition of $V(G)$. There exists a (vertex spanning) walk of finite length in $G^\pi$ in correspondence with $W$.
\end{prop}

\begin{proof}    
    Since $W$ is of finite length, it follows that all edges along it are finite as well. Consider the partition cell the starting vertex of $W$ is in. We denote this partition cell with $v_1$ and refer to it interchangeably with the corresponding vertex of $G^\pi$. Because $W$ spans $V(G)$, it follows that there exists a first vertex outside of partition cell $v_1$. Let the partition cell containing this vertex be $v_2$. This process can be continued inductively to span all the vertices of $G^\pi$. Note that this walk may revisit the same vertices or edges multiple times before seeing new ones.
\end{proof}

Thus directed walks in $G$ correspond with directed walks in $G^\pi$. The weight sum of the edges may differ however, as not all edges used in a walk of $G$ may be needed to traverse the corresponding walk in $G^\pi$.

Naturally, if the underlying walk from $G$ is a vertex spanning walk, then the corresponding  walk in $G^\pi$ spans its vertices. This in turn creates a finite edge from the starting vertex to each vertex in the graph's path completion. This implies that the starting vertex to a vertex spanning walk in $G$ source dominates $P(G^\pi)$. From this observation we are able to conclude that any binned graph capturing a single trajectory has only finitely persisting homological features.

\begin{prop}
    Any quotient graph resulting from a binned dynamical trajectory, will have contractible maximal Dowker source and sink complexes under a path completion.
\end{prop}
\begin{proof}
    As discussed previously, represented as a graph $G$, a dynamical trajectory is simply a directed path. Thus $P(G)$ has both a one element source dominating set (universal source) and a one element sink dominating set (universal sink). The binning procedure can be described as a reduction of the graph to an unnormalized quotient graph under some vertex partition $\pi$. By Proposition \ref{prop:quotientwalk} it follows that there exists a vertex of $G^\pi$ with a directed path to any other vertex, and a vertex which all other vertices have a directed path to. Thus $P(G^\pi)$ must also have a universal source and sink. By Proposition \ref{prop:Dominating_set_reduction} it follows that $P(G^\pi)$ has contractible maximal source and sink complexes.
\end{proof}

Now suppose the underlying point cloud was not associated to a single dynamical trajectory, but instead mulitple dynamical trajectories. For example it may consist of repeated measurements of the same dynamical system, or perhaps multiple simultaneous measurements of the system. 

In this case instead of a total order on the measurements, we have a partial order where each of the underlying trajectories induce a directed path graph. Thus the underlying graph, $G$, is a disjoint union of directed path graphs. This structure reveals a natural set of directed walks in $G$ which span the quotient graph. In combination with our results on dominating sets, this lets us bound the dimension of the maximal Dowker complex.






\begin{cor}\label{cor:disjnt_paths}
    Let $G$ be a graph and $\pi$ a partition of its vertices. Let $h$ be a set of directed paths in $G$ such that their vertices span the partition cells (that is, each partition cell contains at least one vertex of some directed path in $h$). It follows that for $|h|>1$, the maximal Dowker complex $\FD_{\delta_\text{max}}(G^\pi)$ has trivial $k$th homology for all $k\geq |h|-1$.
\end{cor}
\begin{proof}
    By Proposition \ref{prop:quotientwalk}, every directed path of $h$ induces a directed walk in $G^\pi$. Let the $h^\pi$ be the set of these walks.
    
    Consider the first vertex, $v_0$, on some walk $w\in h^\pi$. Naturally $G^\pi$ contains a directed path from $v_0$ to each vertex of $w$. Thus the out neighborhood of $v_0$ in $P(G^\pi)$ contains all vertices of $w$. 
    
    By assumption the walks in $h^\pi$ span the vertex set of $G^\pi$. Thus $h^\pi$ induces a source dominating set of $P(G^\pi)$  of size $|h^\pi|=|h|$.

    The conclusion follows from Proposition \ref{prop:Dominating_set_reduction}.
\end{proof}

Corollary \ref{cor:disjnt_paths} tells us that when binning multiple dynamical trajectories, we can guarantee finitely persisting features when a small set of them span the bins. In particular, a spanning set of size $k>1$ procludes indefinitely persisting features of dimension $k-1$ or larger.

\section{Experimental Results}\label{sec:experiments}

In this section, we present illustrative examples of our methodology. We first build directed graphs and compute the shortest path distance using NetworkX (\cite{networkX}). This distance matrix is then used as a sub-level set filtration of binary relations for computing Dowker persistence, performed by Hellmer and Spali{\'n}ski's pyDowker (\cite{hellmer2024density}), which produces a simplex tree in the popular computational geometry and topology library Gudhi (\cite{boissonnat2014simplex}). This pipeline is available as an installable Python library (details in Section \ref{sec:code}).

\begin{figure}[!htbp]
        \centering
        \includegraphics[width=0.8\linewidth]{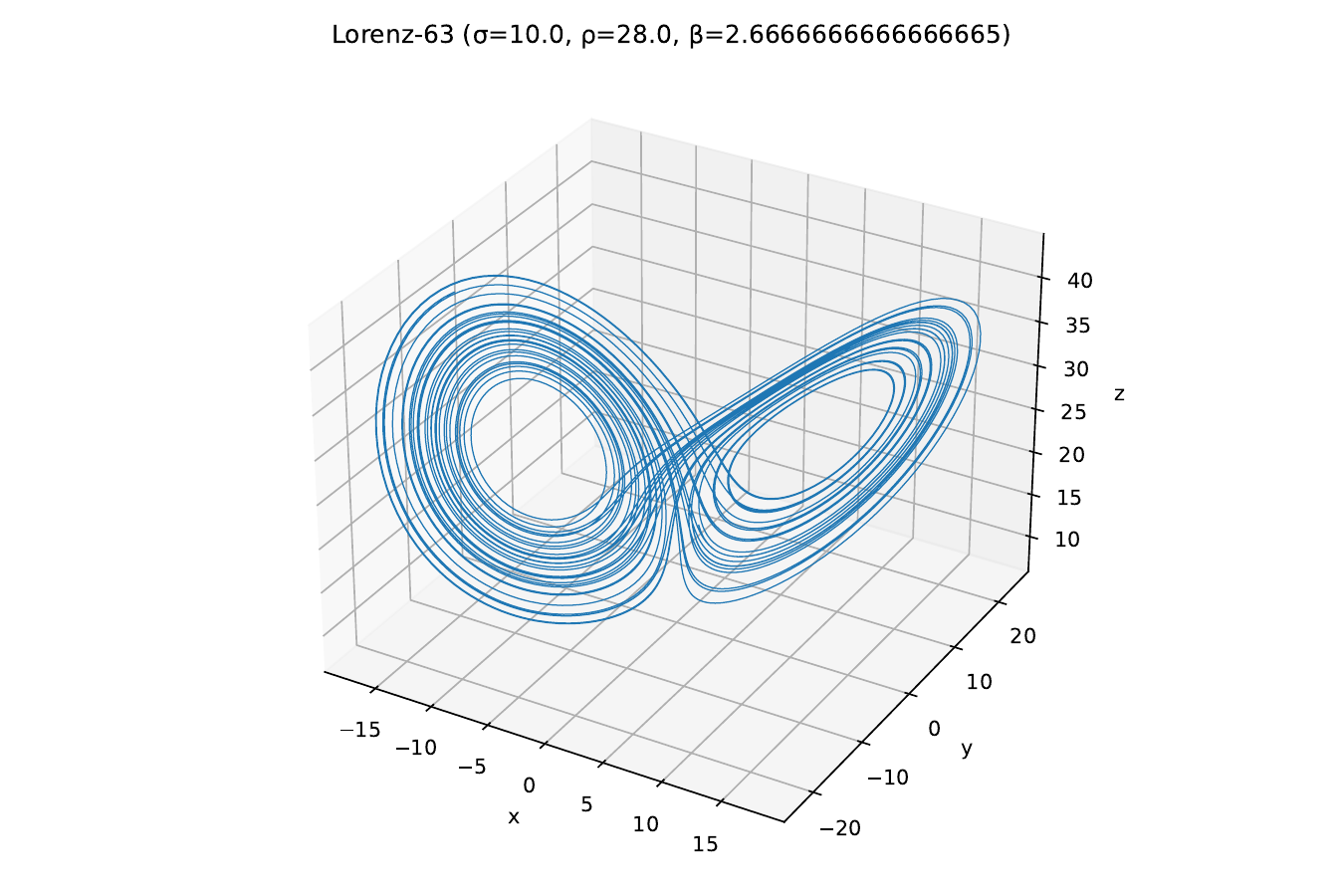}
    \caption{A sample from the Lorenz '63 attractor with standard parameter choices.}
    \label{fig:lorenzts}
\end{figure}

\subsection{Lorenz '63}\label{sec:l63}

Our first test case is the well-known Lorenz '63 system. This system of differential equations is an early cellular convection model (\cite{Lorenz63}) that is at the heart of many didactic examples in chaos theory and dynamical systems. See \href{https://joshdorrington.github.io/L63_simulator/}{this website} for a delightful interactive simulation courtesy of Josh Dorrington. A sampling from the Lorenz '63 system is shown in Figure \ref{fig:lorenzts}, where the familiar butterfly shape is apparent. Unsurprisingly, it is commonly accepted that there are two regimes in Lorenz '63, the two `wings' of the butterfly. Many of the systems' trajectories transition back and forth between these regimes at chaotic intervals.

\cite{Strommen_2022} studies the Lorenz '63 system using a bifiltration (see \cite{Janes_Mpimbo_Mwanzalima_2025}), where one dimension is the standard Rips filtration radius, and the other dimension is a density threshold for sampled points in the attractor. They are able to detect the two regimes at density thresholds admitting most of the sampled points. (Figure 14 in \cite{Strommen_2022}.) They recover the actual cycles representing their persistence using Persloop (\cite{persloop}), software designed for computing persistent cycles in $H_1$.

In this section we present experimental results using our method of Dowker homology on asymmetric graphs. As is highlighted in this paper, Dowker homology differs from Rips homology in its respect for the direction of the trajectory. In our setting, with a binned time series, this means that in the corresponding complex, the filtration parameter captures flow between regions, relating time and position.

\begin{figure}[t]
    \begin{minipage}{0.6\textwidth}
        \centering
        \includegraphics[width=\linewidth]{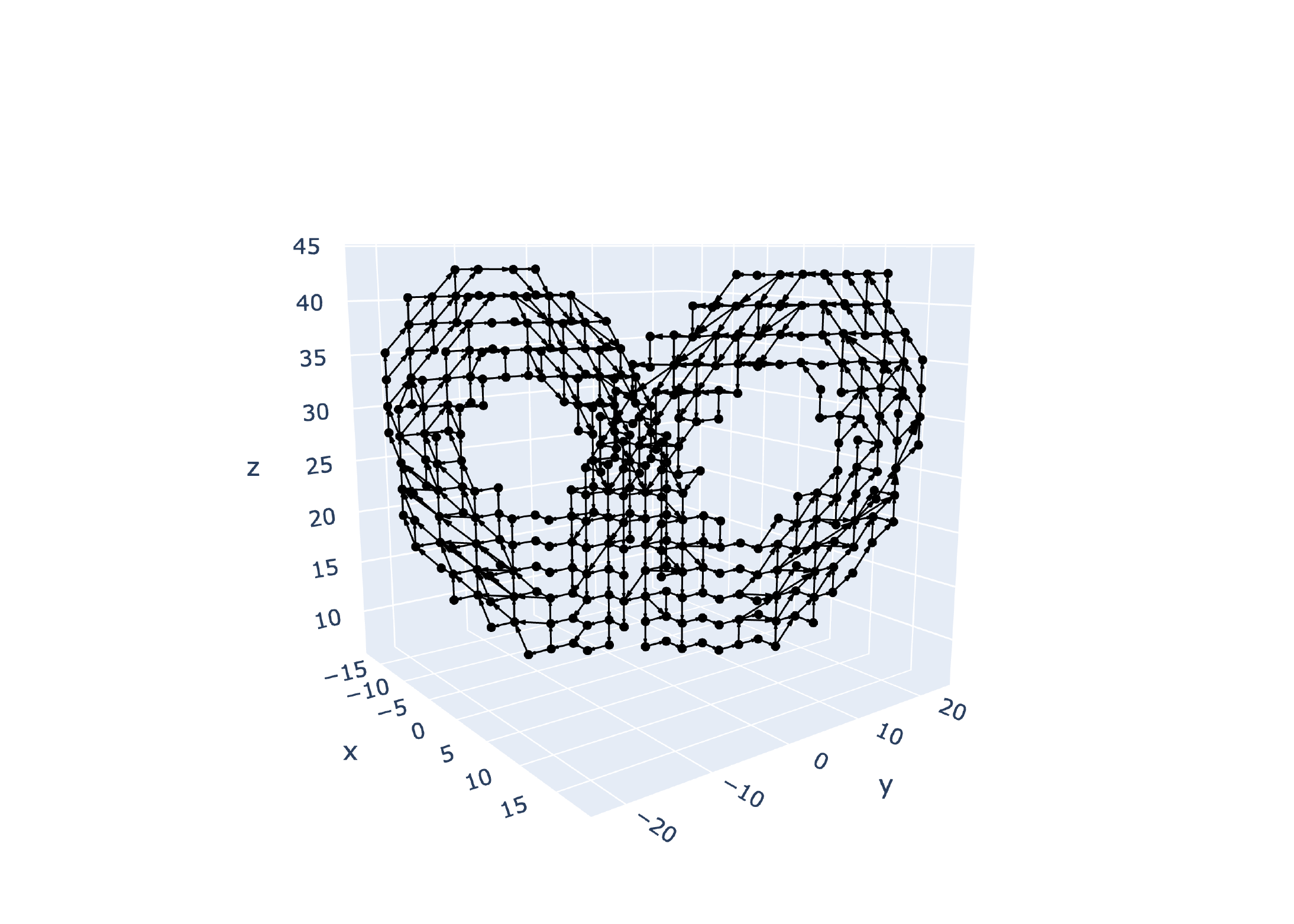}
    \end{minipage}%
    \begin{minipage}{0.4\textwidth}
        \centering
        \includegraphics[width=\linewidth]{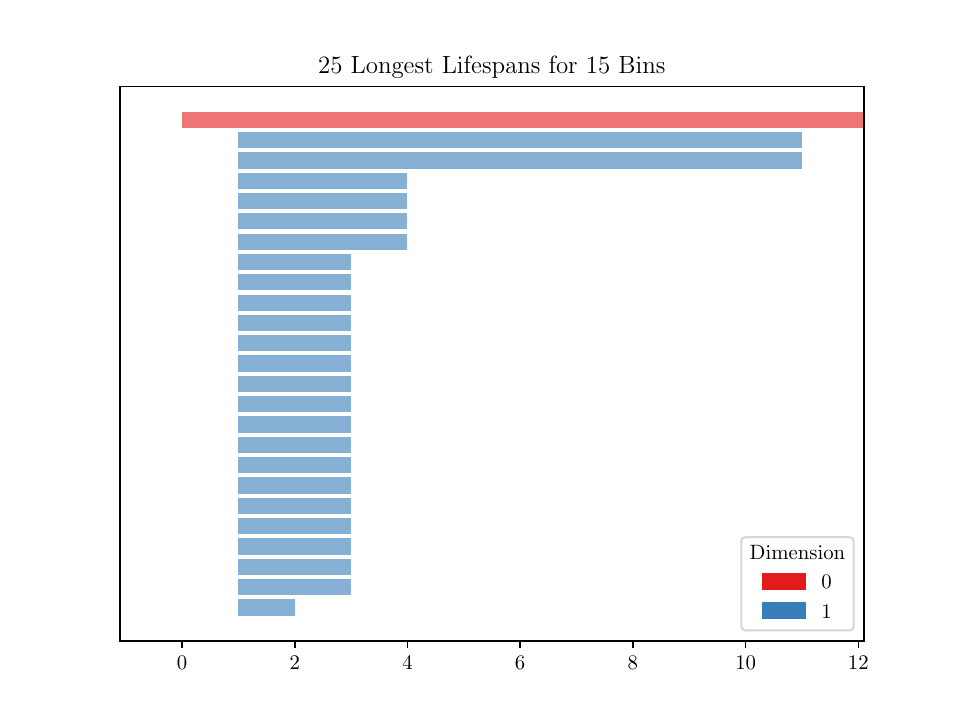}
    \end{minipage}
    \caption{The binned trajectory graph (left) with $b = 15$ and corresponding Dowker barcode (right).}
    \label{fig:l63graph_bin15}
\end{figure}

\begin{figure}[t]
    \begin{minipage}{0.6\textwidth}
        \centering
        \includegraphics[width=\linewidth]{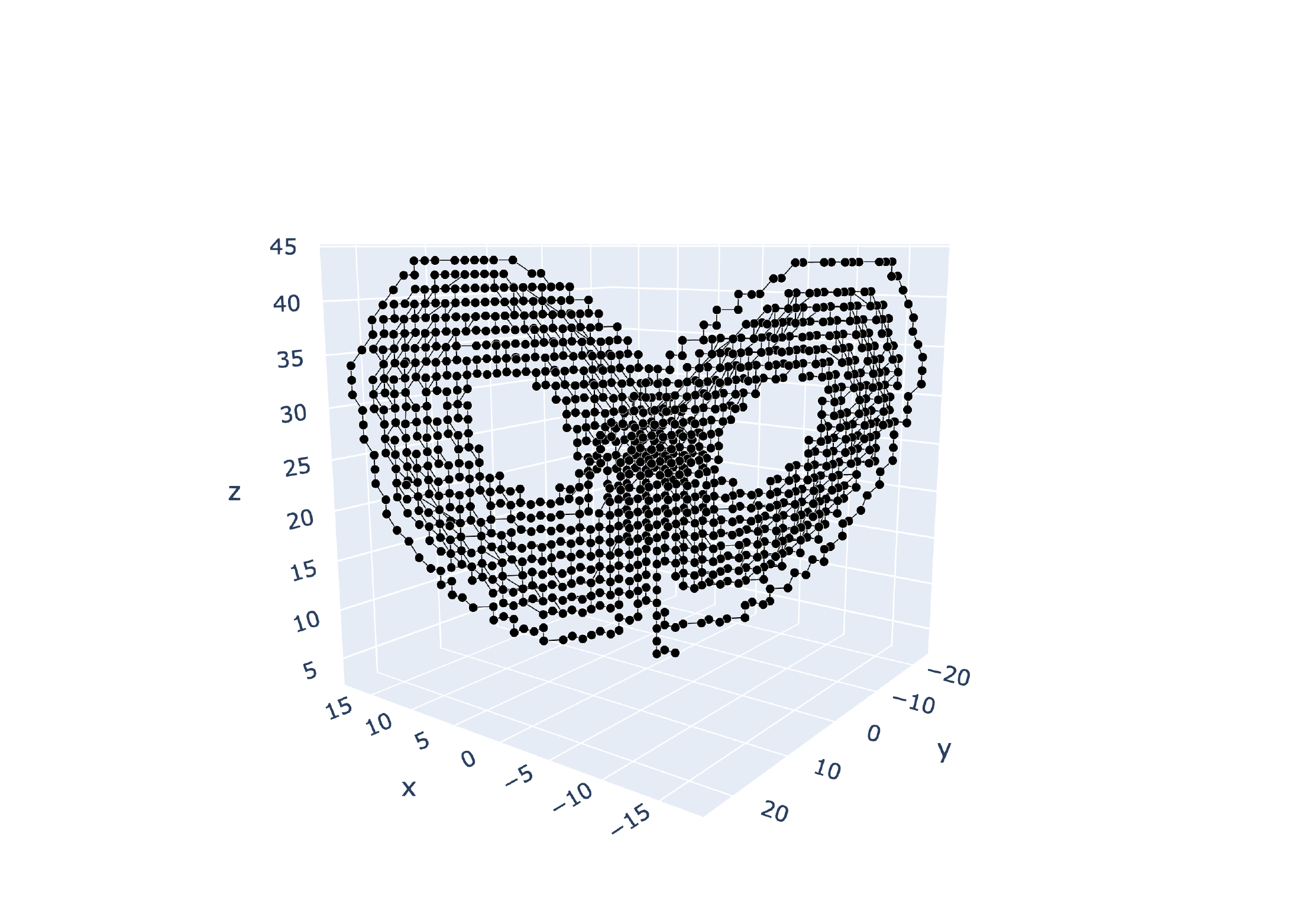}
    \end{minipage}%
    \begin{minipage}{0.4\textwidth}
        \centering
        \includegraphics[width=\linewidth]{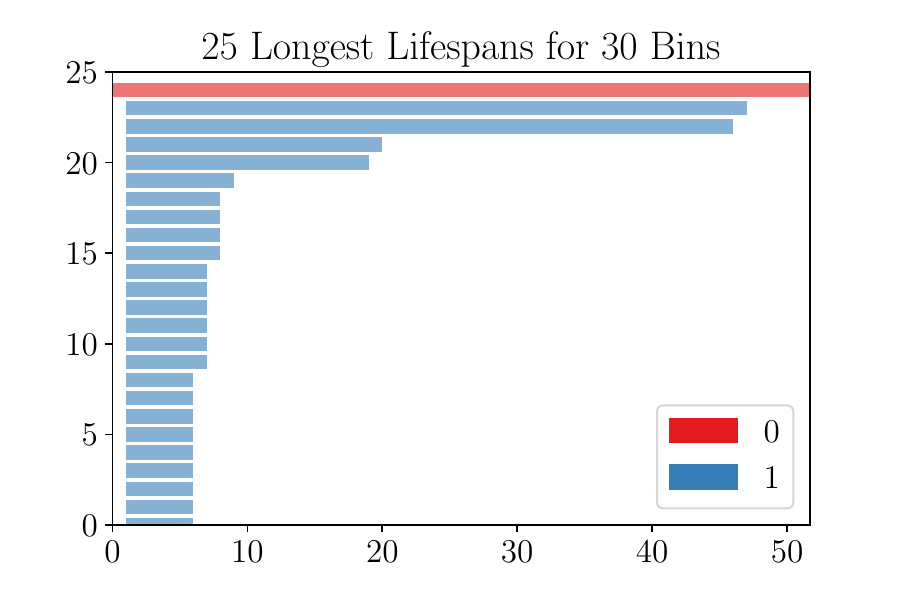}
    \end{minipage}
    \caption{The binned trajectory graph (left) with $b = 30$ and corresponding Dowker barcode (right).}
    \label{fig:l63graph_bin30}
\end{figure}

A coarse grained network of Lorenz '63 is shown in Figure \ref{fig:l63graph_bin15}. This example uses $15$ bins in each dimension. Notice the two $H_1$ classes which persist much longer than the rest. These correspond to the loop flows encompassing each `wing' of the attractor. This illustrates the two regimes of the attractor.

The symmetry of the flows, both spatially and temporally, allow for the similarity in their persistence in Figure \ref{fig:l63graph_bin15}. Let us consider the role of edge directions and highlight the additional temporal information captured. Figure \ref{fig:l63edge_rand} shows two alternative Dowker persistence barcodes. Each was obtained from the network depicted in \ref{fig:l63graph_bin15} after a randomization of its edge directions. In both cases we see a radically altered distribution. In \ref{fig:l63edge_rand_a} two prominent cycles remain, but no longer have comparable persistence, while in \ref{fig:l63edge_rand_b} one cycle persists far longer than the others.

We can interpret these examples using the cycle graph characterization presented in Section \ref{sec:cycle_persistence}. Seen roughly, the directed graph we create from a trajectory in the Lorenz system is sampled from two intersecting punctured disks (the two regimes of Lorenz 63), with many possible consistently oriented graph cycles circling each disk. Recall that in proposition \ref{prop:hom_to_bdry} we found that any nontrivial element of the first homology is generated by a graph cycle and is thus closed by one as well. Consider one encircling consistently oriented cycle subgraph. For larger cycle size, a randomization of edge directions is more likely to yield a cycle graph whose path completion has a large minimal dominating set and thus cannot close a first homology class. This means that the cycles in one disk which previously closed the first homology class are less likely to do so after edge randomization. This results in the decoupling of the two prominent first homology classes in Figure \ref{fig:l63graph_bin15}.

This example illustrates the additional flow information captured in Dowker persistence. Alternative edge orientations which would shift our interpretation of the system can radically alter persistence in this framework. This altered flow information would not correspond to alteration in positional data, and thus would not be picked up by a point cloud based persistence, like Rips. When a Rips filtration finds two persistent homology classes, it points to a similarity of point cloud structure, whereas such a result in an asymmetric Dowker framework points more to a similarity of the flows.

\begin{figure}[t]
    \begin{subfigure}{0.48\textwidth}
        \centering
        \includegraphics[width=\linewidth,trim={0 0 0 1.35cm},clip]{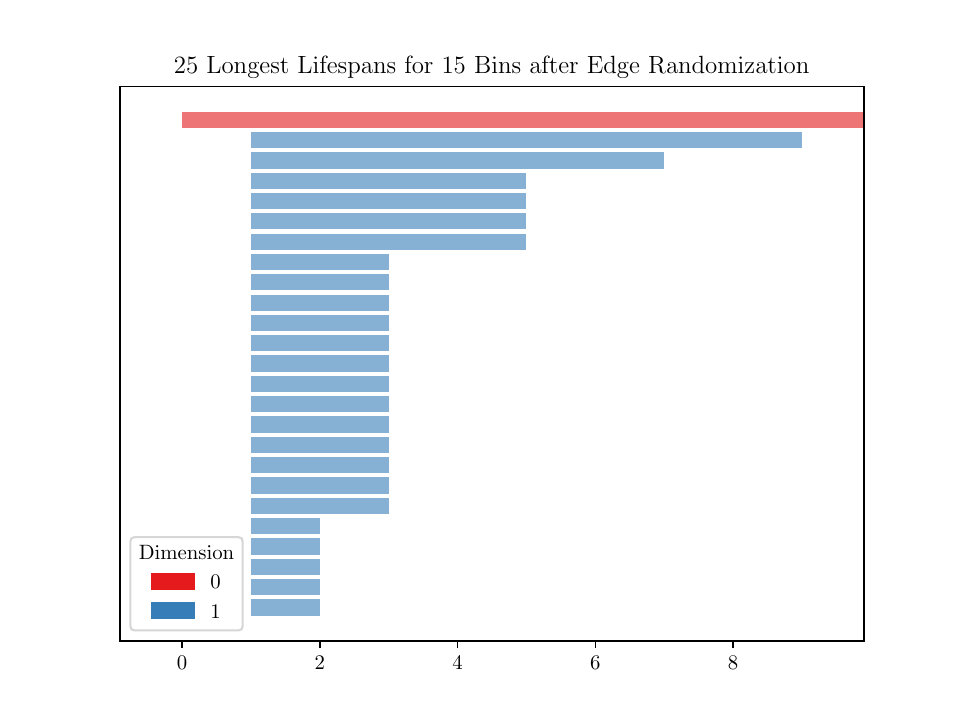}
        \caption{ }
        \label{fig:l63edge_rand_a}
    \end{subfigure}%
    \begin{subfigure}{0.48\textwidth}
        \centering
        \includegraphics[width=\linewidth,trim={0 0 0 1.35cm},clip]{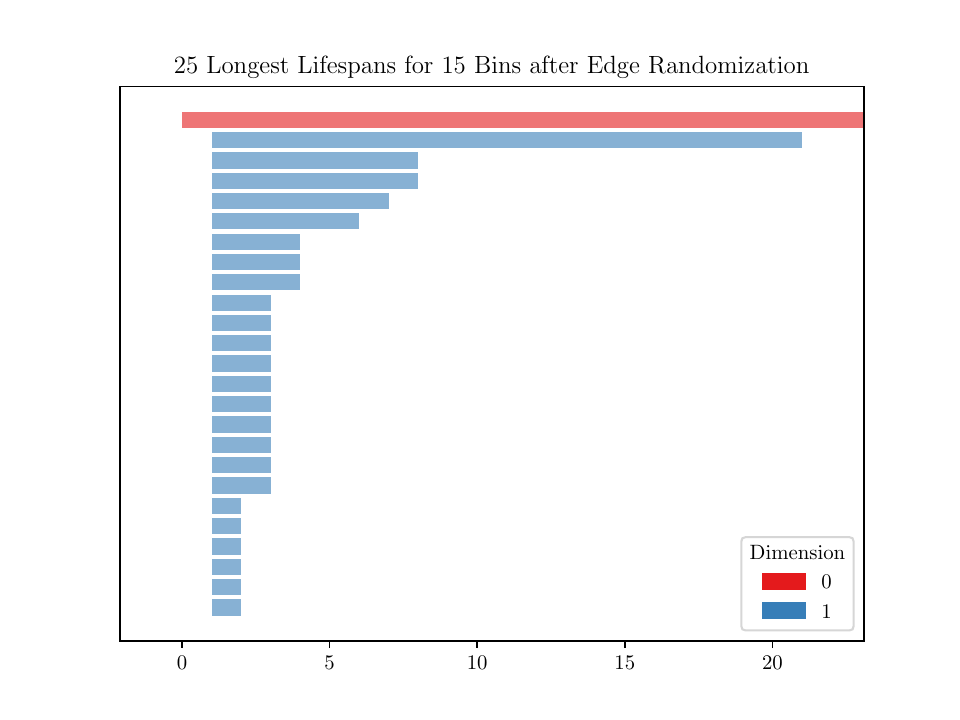}
        \caption{}
        \label{fig:l63edge_rand_b}
    \end{subfigure}
    \caption{ Given the Network depicted in Figure \ref{fig:l63graph_bin15}, we consider the persistence of the same network with randomized edge directions. In this figure, the left and right barcodes depict the 25 longest persisting Dowker homology classes in two instances of this randomization procedure.}\label{fig:l63edge_rand}
\end{figure}

In Figure \ref{fig:l63graph_bin30} we show this framework on the same trajectory data as Figure \ref{fig:l63graph_bin15} but this time with $30$ bins in each dimension. For this very granular binning we see some separation between groups of trajectories in the Lorenz loop flows as represented in the persistent $H_1$ homology classes. First, notice the two longest lasting cycle classes, each persisting close to $50$ steps. These correspond to long cycles about each wing of the Lorenz butterfly, which were far from the main mass of trajectories. As we saw in Figure \ref{fig:l63graph_bin15}, these are more connected to the rest of the graph at coarser binnings, getting incorporated into their homology classes. The second pair of long-lived cycle classes also persist markedly longer than the rest. These correspond to the two loop flows through the Lorenz wings, captured by the rest of the trajectory.

Optimizing the number of bins or how the space is quotiented is an area for further study. We note that a division of the traversed phase space into bins as opposed to choosing a fixed bin size can create some instability, as removing points may shift the bins enough to change the resultant persistent homology.

Finally, to quantitatively showcase the robustness of our pipeline to missing or fragmented observations, we performed a controlled stability experiment using trajectories of the Lorenz system. These numerical results follow up on the description of this use case in Section \ref{sec:dyn_traj}.  For a given trajectory duration $T$, we first compute the Dowker persistence diagram of the full trajectory after converting it into a directed graph. We then simulate signals with missing data by sampling $n$ segments of approximate length $T/n$ on the attractor. 

We then used the bottleneck distance (see, for example, \cite{edelsbrunner2010computational}) to compare the persistence of the full trajectory to the union of the fragmented trajectories. A persistence diagram takes every bar in the barcode and plots it as an ordered pair $(birth, death)$ in a coordinate plane (persistence diagram). Features that have short lifespans are plotted close to the diagonal $y = x$ line and points that have long lifespans are plotted further away from the diagonal. The bottleneck distance takes two persistence diagrams and matches the features to each other (points on the diagonal can be used if there are not the same number of features). The bottleneck distance is the maximum $L_\infty$ distance between two points in the matching that has smallest such maximum distance. A small bottleneck distance indicates more similar persistence diagrams, and the opposite for a larger bottleneck distance. 

Dowker persistence was recomputed on the combined graph, and the bottleneck distance between the full-trajectory and fragmented-trajectory diagrams was measured separately for $H_0$ and $H_1$ homologies. Repeating this procedure across a grid of durations $T \in [5, 100]$ and segment counts $n \in [2, 22]$ produced the heatmaps in Figure \ref{fig:bottleneck_heatmaps}. The results show that the bottleneck distances remain uniformly small throughout most of the parameter range, with localized increases only when the length of the trajectory is long enough for the full trajectory to be nontrivial, but too short for the fragments to contain sufficient dynamical information. These findings suggest remarkable stability under missing-data perturbations, with the example reliably recovering the underlying topological signatures of the dynamics even with severely fragmented trajectories.
\begin{figure}[!htbp]
\includegraphics[width=\linewidth]{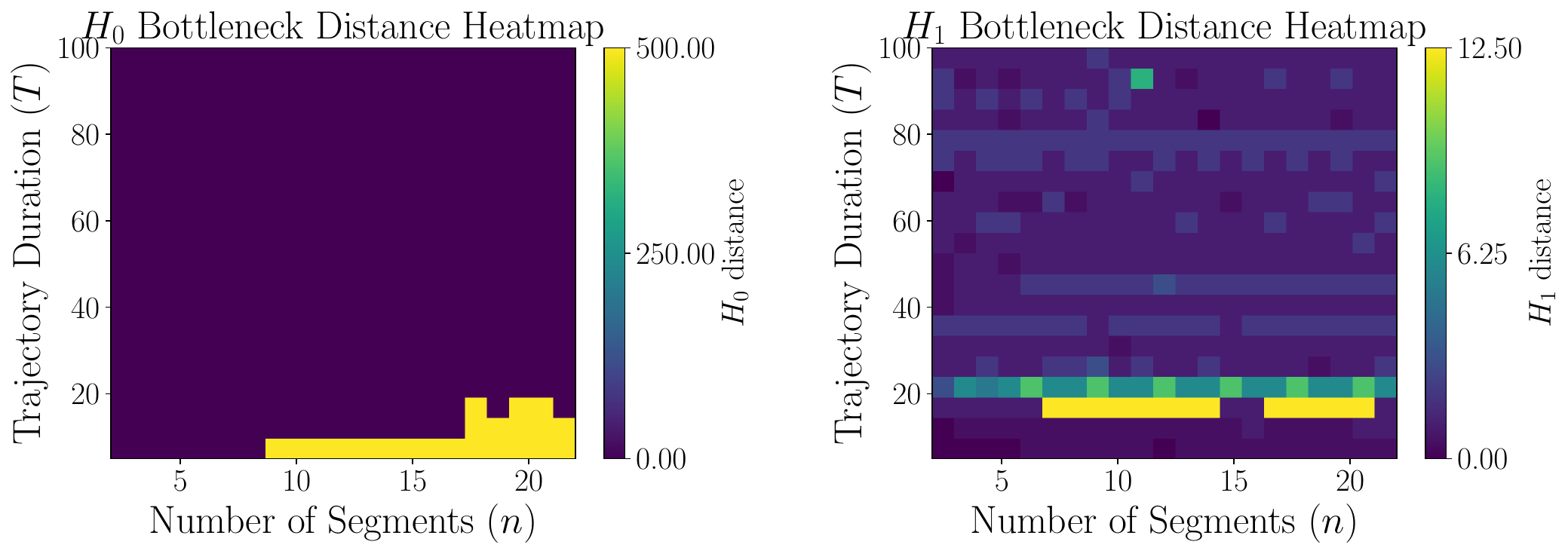}
\caption{Heatmaps for bottleneck distances between full lorenz and fragmented lorenz trajectories---for both $H_0$ and $H_1$.}
\label{fig:bottleneck_heatmaps}
\end{figure}

\subsection{Charney-DeVore}

We also used our method to compute the Dowker persistent homology of the Charney-DeVore dynamical system, a model simulating atmospheric flow with a blocking pattern (\cite{charney1979multiple}). A plot of the time series can be seen in Figure \ref{fig:cdvts}. The system is made up of several exterior rings that all loop back in to a main and much slower moving cylinder. 

\begin{figure}[t]
        \centering
        \includegraphics[width=0.8\linewidth]{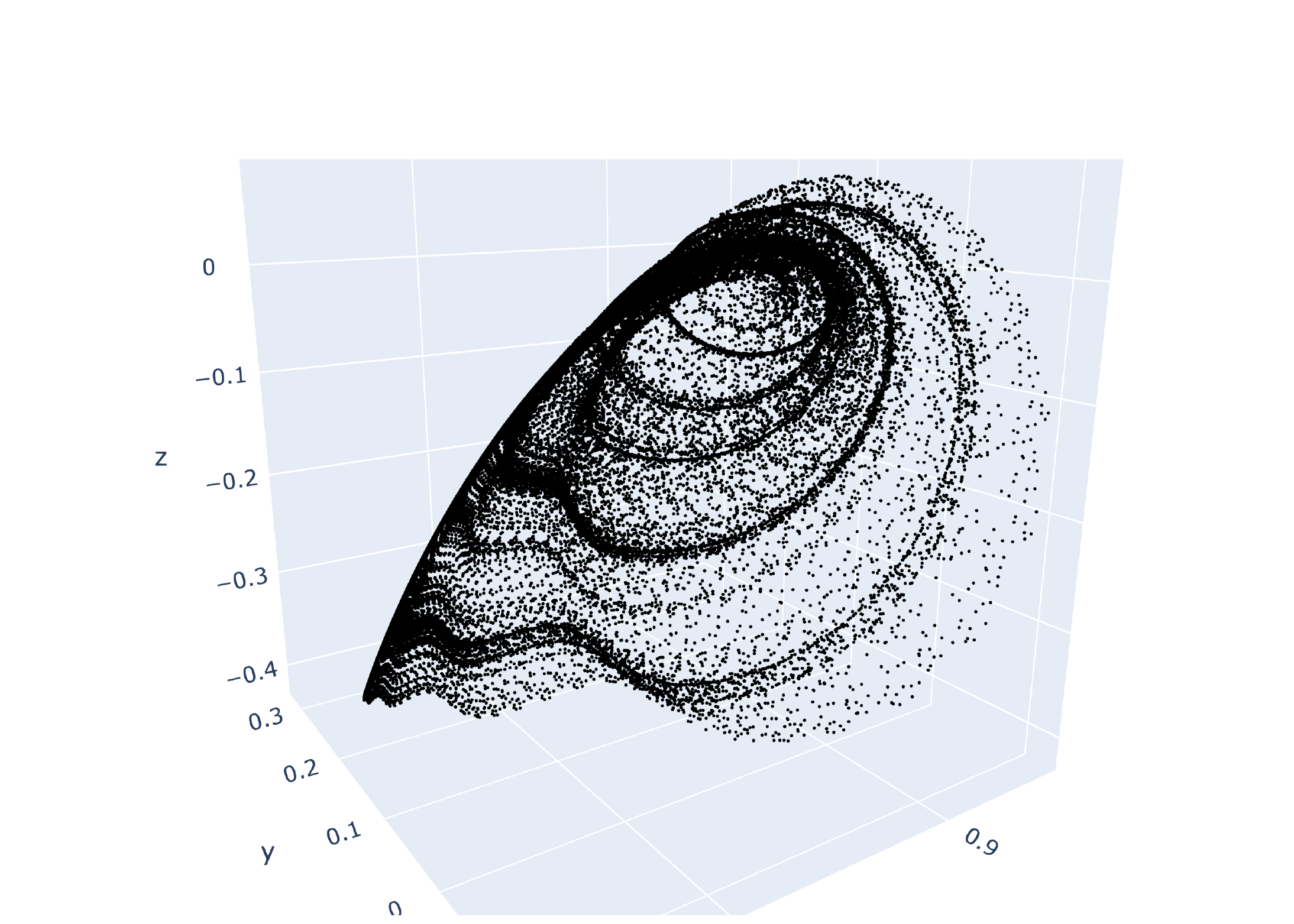}
    \caption{A sample from the Charney-DeVore time series.}
    \label{fig:cdvts}
\end{figure}

In Figure \ref{fig:cdvgraph}, our binned graph and corresponding barcodes are shown for a binning number of $b=30$, while in Figure \ref{fig:cdvbarcodes} shows the barcodes for several additional choices of $b$. One thing to note here is that the one dimensional barcodes get consistently longer as the bin number increases. This indicates that the exterior rings consistently stay in their lane and do not cross into the other rings, even at higher resolution.

This result sets our method apart from a more traditional filtration, which can have difficulty with the distributed fine scale structure presented in the CDV system. Because the loops lie in a region that is sparsely inhabited, a Rips filtration misses them, forming a large connected component.
To combat this, the authors of \cite{Strommen_2022} use a direct binning system akin to ours to filter out the low density regions in a bifiltration for the CDV system. Our framework picks these low density loops out natively because the filtration respects the flows of the trajectories.

In this experiment, we use edge weight $1$ for the directed graphs. Future research could entail weighting the edges such that the 0 dimensional barcode provides information towards differentiating the slower moving blocking central region from the faster moving exterior rings. Another direction for future research is to combine the directed binning of \cite{Strommen_2022} with our approach in a Dowker bifiltration to capture both the fine and high-density loops.

\begin{figure}[t]
    \begin{minipage}{0.6\textwidth}
        \centering
        \includegraphics[width=\linewidth]{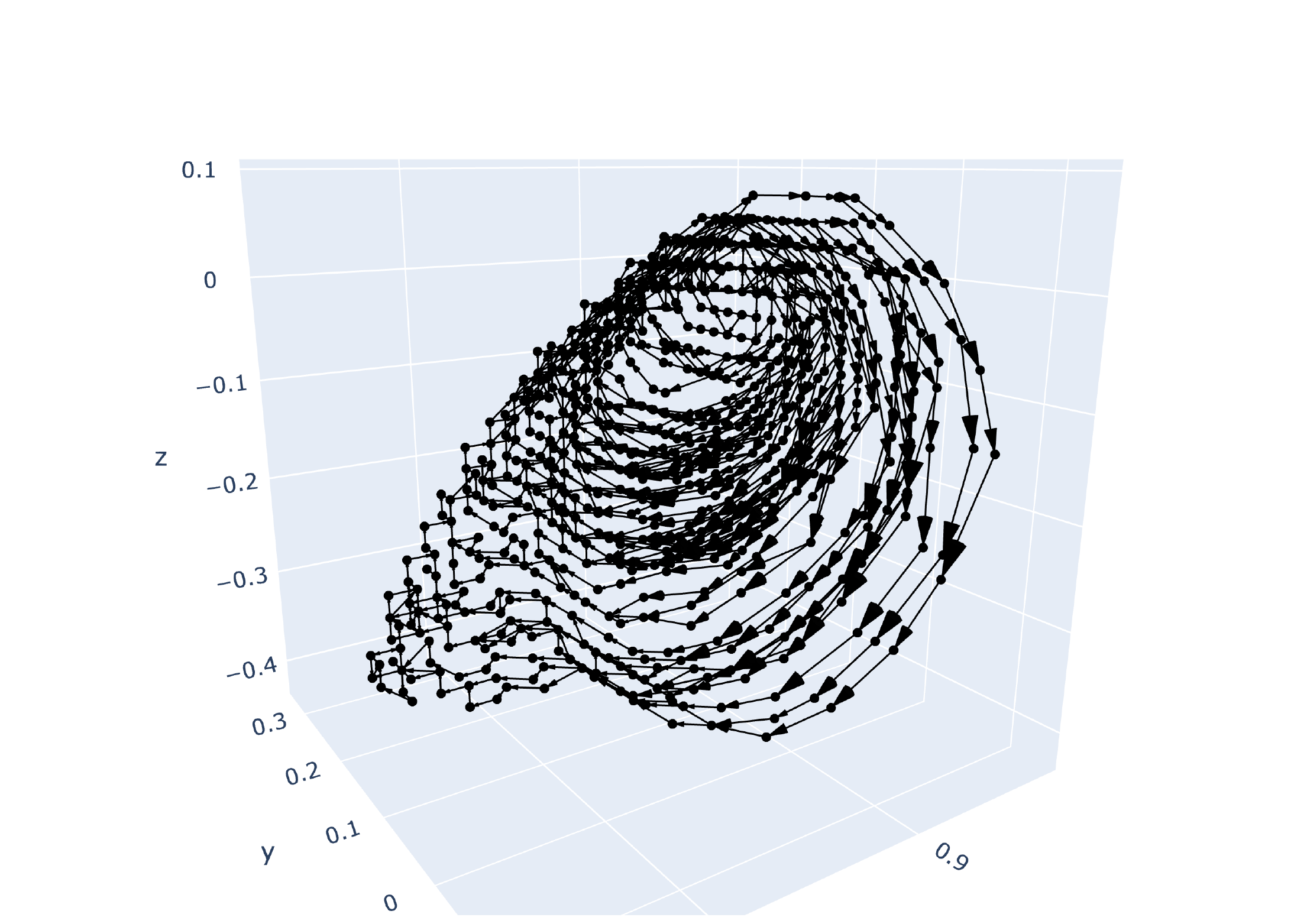}
    \end{minipage}%
    \begin{minipage}{0.4\textwidth}
        \centering
        \includegraphics[width=\linewidth,trim={0 0 0 1.35cm},clip]{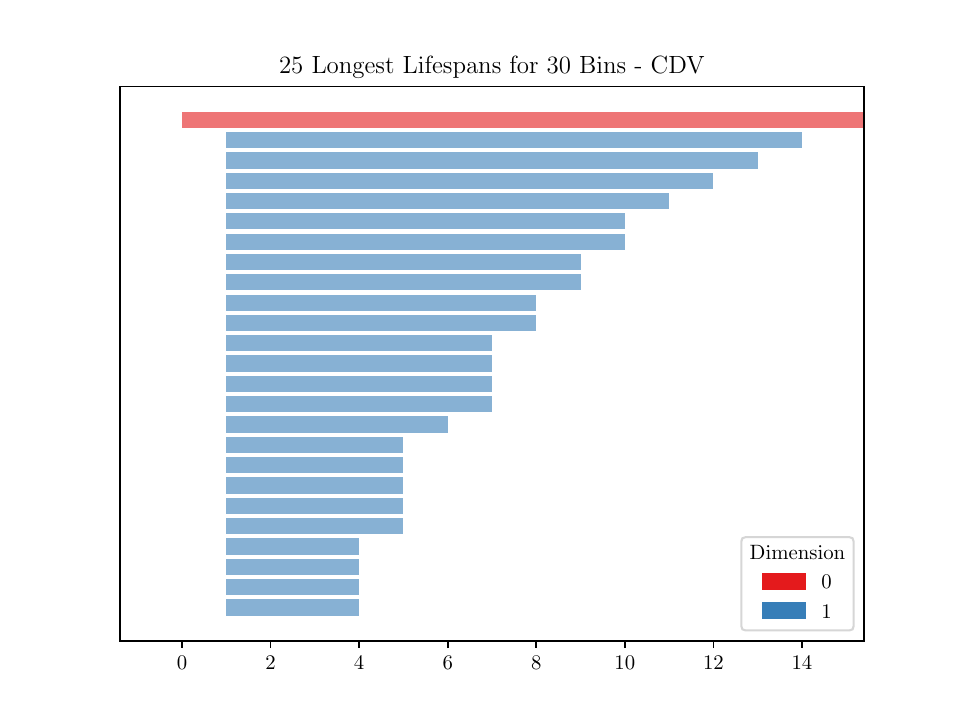}
    \end{minipage}
    \caption{The binned CDV trajectory graph (left) with $30$ bins per axis and the corresponding Dowker barcode (right), showing the 25 longest persisting homology classes.}
    \label{fig:cdvgraph}
\end{figure}

\begin{figure}[t]
    \centering
    
    \begin{subfigure}{0.33\textwidth}
        \centering
        \includegraphics[width=\linewidth,trim={0 0 0 1.35cm},clip]{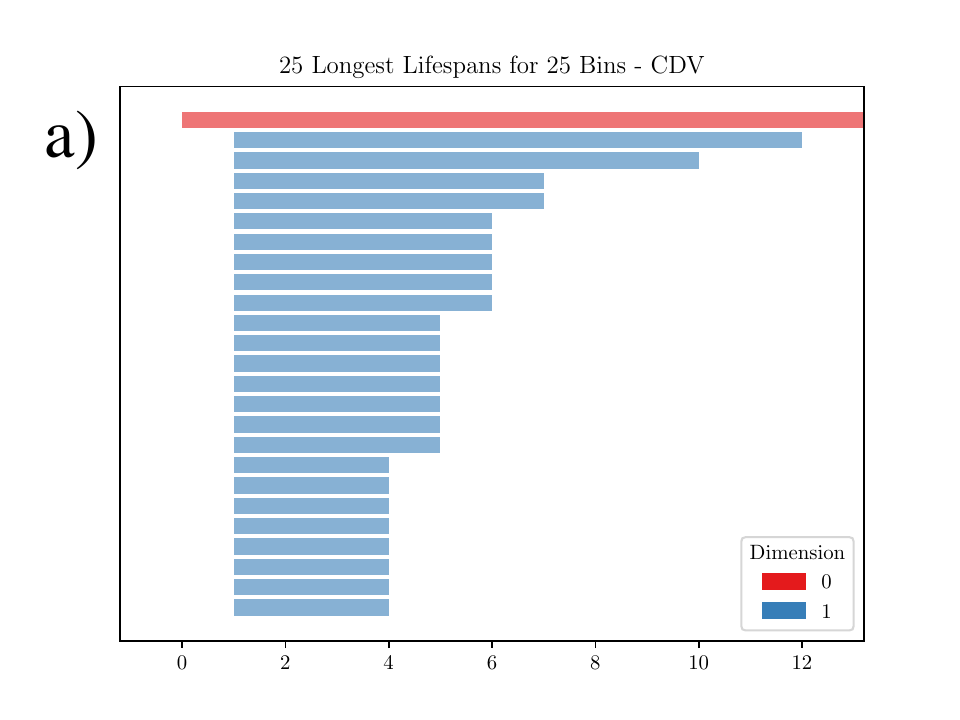}
        \caption{25 bins}
        \label{fig:cdvbarcode25}
    \end{subfigure}%
    \begin{subfigure}{0.33\textwidth}
        \centering
        \includegraphics[width=\linewidth,trim={0 0 0 1.35cm},clip]{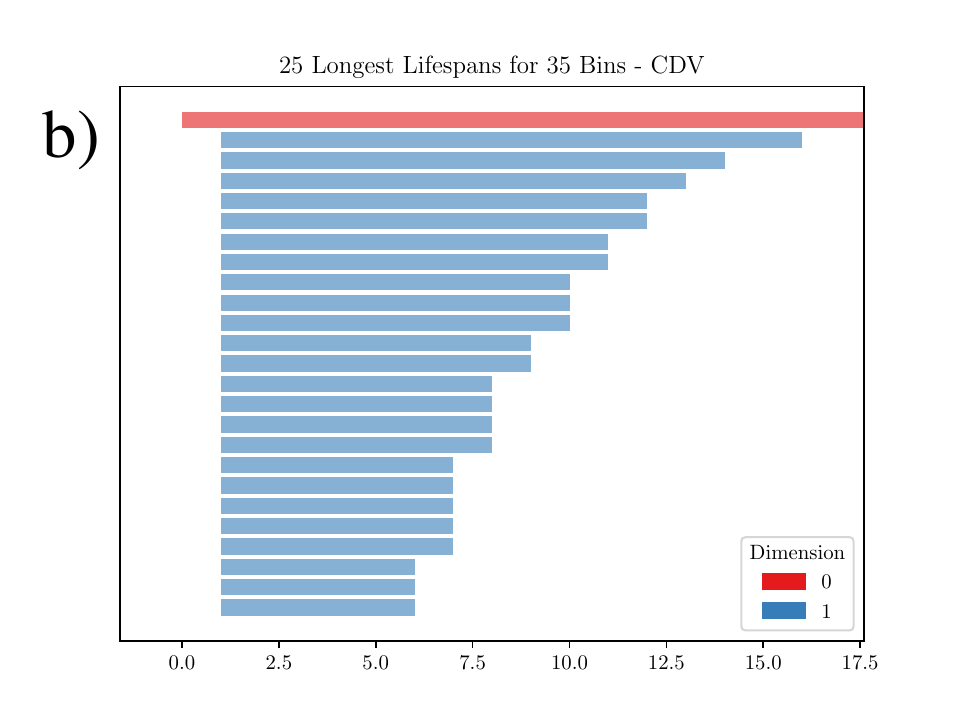}
        \caption{35 bins}
        \label{fig:cdvbarcode35}
    \end{subfigure}
    \begin{subfigure}{0.33\textwidth}
        \centering
        \includegraphics[width=\linewidth,trim={0 0 0 1.35cm},clip]{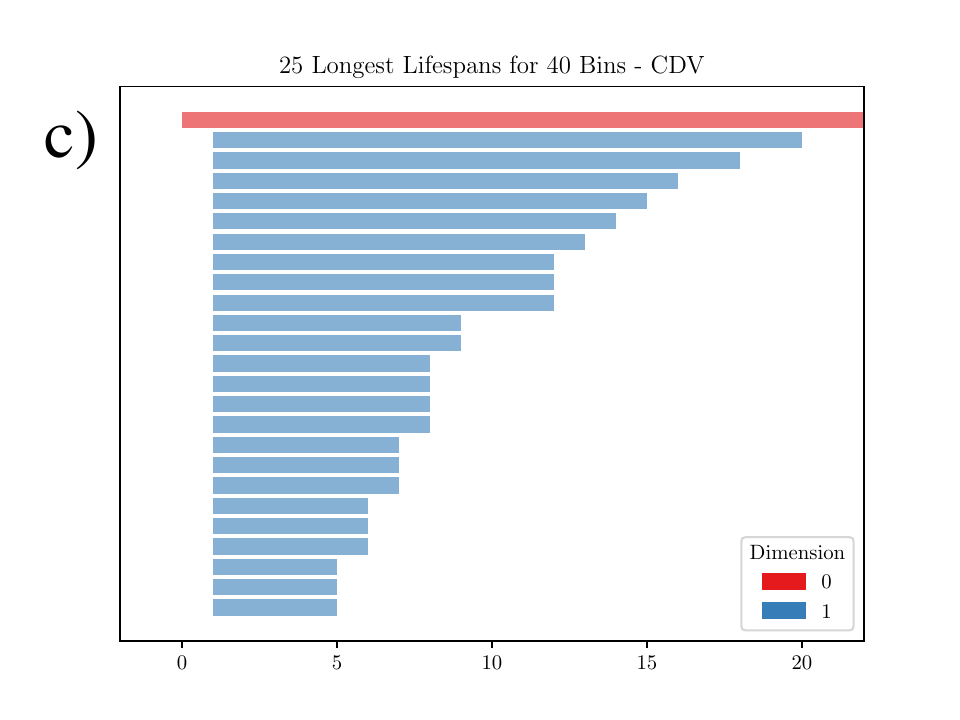}
        \caption{40 bins}
        \label{fig:cdvbarcode40}
    \end{subfigure}

    \caption{Dowker barcodes showing the 25 longest persisting homology classes for several choices of bin number in the CDV system.}
    \label{fig:cdvbarcodes}
\end{figure}

\section{Conclusion}

This paper aims to offer an avenue of pursuit towards topological characterization of large dynamical systems. Our approach is notable for two reasons: the symbolization of the dynamical system into bins, which allows us to encode spatial and temporal information into the base space, and the use of an asymmetric homology theory based on the Dowker complex. With this approach, we connect persistence to the graph theoretic notions of the dominating sets and directed paths, toward better connection with underlying dynamics.

The analytical section of this paper develops the theory of the Dowker persistent homology for weighted, directed graphs. First, we classify the 1D persistent homology of the consistently oriented weighted cycle graph in Proposition \ref{prop:concycle}. Then, we move on to cycles that are not consistently oriented. In Propositions \ref{prop:inconcycledom3}, \ref{prop:inconcycledom2} and \ref{prop:inconcycledom1}, we show that, in the inconsistently oriented case, the 1D barcode depends upon the size of the dominating set of the cycle. In doing so, we fully classify the 1D persistent homology of any directed cycle graph, regardless of weights or orientation. We add a result on the 1D persistent homology of the wedge sums of two graphs, showing that the 1D barcode of the wedge sum is the union of the barcodes. Throughout, the results are accompanied by computational examples. Our theoretical discussion closes by framing the sampling process of trajectory data in graph-theoretical terms, shedding light on how our theoretical results fit into the framework of analyzing time series samples of dynamical systems. 

We close with numerical experiments on the Lorenz '63 and Charney-DeVore dynamical systems, highlighting how our method differs from traditional point cloud persistence for dynamical trajectory data. Our tests also suggest that this method is robust to subsampling, where data may be missing or corrupted.

We offer a couple additional directions in which we believe that this research could be expanded. First could be to modify the encoding of sequential information into the quotient graph, either by the consideration of different quotient graph edge weight functions or encoding this information in hypergraph edges. This could allow for the adjustment of the kinds of dynamical features picked up in the persistence or add additional robustness. Second, we believe there is more to be done in exploring the graph-theoretic properties of our binned graphs. For example, using network clustering methods could reveal other kinds of state space organization.

Finally, we believe future work could combine this work with dynamic choices of phase space partitions, or bin number.

\backmatter

\section*{Code Availability}\label{sec:code}
The source code used for this work is available as an installable python library at \url{https://github.com/cactismath/CACTIS}.

\section*{Acknowledgements}

This project began as a subgroup at a 2024 American Mathematical Society Mathematics Research Community (AMS MRC), of which \cite{Faranda_2024} was the starting point. We are grateful to the AMS for running the MRC program, and to Davide Faranda, Theo Lacombe, Nina Otter and Kristian Strommen for organizing ours specifically. Other researchers who were at the MRC who have helped shape this paper include Fangfei Lan, Enrique Alvarado, and Josh Dorrington. 

CP would also like to thank fellow MRC participant Soheil Anbouhi, and Warren Wilson College students Willow Solomon, Kian De Rensis-Williams, and Ryan Kiser for helpful conversations in the time since the MRC. 

TT would like to thank Mike Miller Eismeier and Alice Patania for helpful conversations related to the content of this paper.

\section*{Declarations}

This material is based upon work supported by the National Science Foundation Graduate 
Research Fellowship Program under Grants No. 2235204 and No. 1916439. Any opinions, findings, 
and conclusions or recommendations expressed in this material are those of the authors 
and do not necessarily reflect the views of the National Science Foundation.








\bibliography{bibliography}

@article{Borsuk1948,
author = {Borsuk, Karol},
journal = {Fundamenta Mathematicae},
language = {eng},
number = {1},
pages = {217-234},
title = {On the imbedding of systems of compacta in simplicial complexes},
url = {http://eudml.org/doc/213158},
volume = {35},
year = {1948},
}

@book{edelsbrunner2010computational,
  title={Computational topology: an introduction},
  author={Edelsbrunner, Herbert and Harer, John},
  year={2010},
  publisher={American Mathematical Soc.}
}

@article{hellmer2024density,
  title={Density Sensitive Bifiltered Dowker Complexes via Total Weight},
  author={Hellmer, Niklas and Spali{\'n}ski, Jan},
  journal={arXiv preprint arXiv:2405.15592},
  year={2024}
}

@techreport{networkX,
  title={Exploring network structure, dynamics, and function using NetworkX},
  author={Hagberg, Aric and Swart, Pieter J and Schult, Daniel A},
  year={2008},
  institution={Los Alamos National Laboratory (LANL), Los Alamos, NM (United States)}
}

@article{boissonnat2014simplex,
  title={The simplex tree: An efficient data structure for general simplicial complexes},
  author={Boissonnat, Jean-Daniel and Maria, Cl{\'e}ment},
  journal={Algorithmica},
  volume={70},
  number={3},
  pages={406--427},
  year={2014},
  publisher={Springer}
}

@book{Hatcher_AT,
  address = {Cambridge},
  author = {Hatcher, Allen},
  isbn = {0-521-79160-X; 0-521-79540-0},
  mrclass = {55-01 (55-00)},
  mrnumber = {1867354 (2002k:55001)},
  mrreviewer = {Donald W. Kahn},
  pages = {xii+544},
  publisher = {Cambridge University Press},
  title = {Algebraic topology},
  username = {mwpb479},
  year = 2002
}

@Article{luet2019Flagser,
AUTHOR = {Lütgehetmann, Daniel and Govc, Dejan and Smith, Jason P. and Levi, Ran},
TITLE = {Computing Persistent Homology of Directed Flag Complexes},
JOURNAL = {Algorithms},
VOLUME = {13},
YEAR = {2020},
NUMBER = {1},
ARTICLE-NUMBER = {19},
URL = {https://www.mdpi.com/1999-4893/13/1/19},
ISSN = {1999-4893},
DOI = {10.3390/a13010019}
}

@article{Strommen_2022,
   title={A topological perspective on weather regimes},
   volume={60},
   ISSN={1432-0894},
   url={http://dx.doi.org/10.1007/s00382-022-06395-x},
   DOI={10.1007/s00382-022-06395-x},
   number={5–6},
   journal={Climate Dynamics},
   publisher={Springer Science and Business Media LLC},
   author={Strommen, Kristian and Chantry, Matthew and Dorrington, Joshua and Otter, Nina},
   year={2022},
   month=jul, pages={1415–1445} }

@article{Dowker_1952,
 ISSN = {0003486X, 19398980},
 URL = {http://www.jstor.org/stable/1969768},
 author = {C. H. Dowker},
 journal = {Annals of Mathematics},
 number = {1},
 pages = {84--95},
 publisher = {[Annals of Mathematics, Trustees of Princeton University on Behalf of the Annals of Mathematics, Mathematics Department, Princeton University]},
 title = {Homology Groups of Relations},
 urldate = {2024-07-15},
 volume = {56},
 year = {1952}
}

@article{Myers_2023,
   title={Persistent homology of coarse-grained state-space networks},
   volume={107},
   ISSN={2470-0053},
   url={http://dx.doi.org/10.1103/PhysRevE.107.034303},
   DOI={10.1103/physreve.107.034303},
   number={3},
   journal={Physical Review E},
   publisher={American Physical Society (APS)},
   author={Myers, Audun D. and Chumley, Max M. and Khasawneh, Firas A. and Munch, Elizabeth},
   year={2023},
   month=mar }

@article{Faranda_2024, 
    title = {Climate Science at the Interface Between Topological Data Analysis and Dynamical Systems Theory},
    volume = {71}, ISSN = {1088-9477}, url = {https://www.ams.org/notices/202402/rnoti-p267.pdf}, DOI = {10.1090/noti2864}, number = {2}, journal = {Notices of the American Mathematical Society}, publisher = {American Mathematical Society}, author = {Faranda, Davide and Lacombe, Th\'eo and Otter, Nina and Strommen, Kristian}, year = {2024}, month = feb, pages = {267-271}}

@article{chowdhury2018AsymNets, 
ISSN = {2367-1734}, URL = {https://link.springer.com/article/10.1007/s41468-018-0020-6}, author = {Chowdhury, Samir and M\'emoli, Facundo}, journal = {Journal of Applied and Computational Topology}, volume = {2}, pages = {115-175}, publisher = {Springer Nature}, title = {A functorial Dowker theorem and persistent homology of asymmetric networks}, year = {2018}}

@article{Edelsbrunner2000,
  title={Topological Persistence and Simplification},
  author={Herbert Edelsbrunner and David Letscher and Afra Zomorodian},
  journal={Discrete \& Computational Geometry},
  year={2000},
  volume={28},
  pages={511-533},
  url={https://api.semanticscholar.org/CorpusID:9191014}
}

@article{Janes_Mpimbo_Mwanzalima_2025, title={On theoretical construction of bifiltrations}, volume={12}, ISSN={null}, DOI={10.1080/27684830.2025.2474773}, number={1}, journal={Research in Mathematics}, publisher={Taylor & Francis}, author={Janes, Faustine and Mpimbo, Marco and Mwanzalima, Makungu}, year={2025}, month=dec, pages={2474773} }

@article{Franch2020, 
title= {Markov modeling of dynamical systems via clustering and graph minimization}, author = {Daniel K. Franch and Daniel P.B. Chaves and Cecilio Pimentel and Diego M. Hamilton}, journal = {Digital Signal Processing}, year = {2020}, volume = {104}, publisher = {Elsevier}, url = {https://www.sciencedirect.com/science/article/pii/S1051200420301147}}

@article{Galatolo2010, 
title = {Effective symbolic dynamics, random points, statistical behavior, complexity and entropy}, author = {Stefano Galatolo and Mathieu Hoyrup and Crist\'obal Rojas}, journal = {Information and Computation}, year = {2010}, volume = {208}, publisher = {Elsevier}, pages = {23-41}}

@article{Palmer99, author = {T. N. Palmer}, title = {A Nonlinear Dynamical Perspective on Climate Prediction}, journal = {Journal of Climate}, year = {1999}, publisher = {American Meteorological Society}, volume = {12}, number = {2}, pages = {575-591}}

@article{Yalniz20,
    author = {Yalnız, Gökhan and Budanur, Nazmi Burak},
    title = {Inferring symbolic dynamics of chaotic flows from persistence},
    journal = {Chaos: An Interdisciplinary Journal of Nonlinear Science},
    volume = {30},
    number = {3},
    pages = {033109},
    year = {2020},
    month = {03},
    issn = {1054-1500},
    doi = {10.1063/1.5122969},
    url = {https://doi.org/10.1063/1.5122969},
    eprint = {https://pubs.aip.org/aip/cha/article-pdf/doi/10.1063/1.5122969/14624130/033109_1_online.pdf},
}

@article { Lorenz63,
      author = "Edward N.  Lorenz",
      title = "Deterministic Nonperiodic Flow",
      journal = "Journal of Atmospheric Sciences",
      year = "1963",
      publisher = "American Meteorological Society",
      address = "Boston MA, USA",
      volume = "20",
      number = "2",
      doi = "10.1175/1520-0469(1963)020<0130:DNF>2.0.CO;2",
      pages=      "130 - 141",
      url = "https://journals.ametsoc.org/view/journals/atsc/20/2/1520-0469_1963_020_0130_dnf_2_0_co_2.xml"
}

@InProceedings{persloop,
author="Dey, Tamal K.
and Hou, Tao
and Mandal, Sayan",
editor="Marfil, Rebeca
and Calder{\'o}n, Mariletty
and D{\'i}az del R{\'i}o, Fernando
and Real, Pedro
and Bandera, Antonio",
title="Persistent 1-Cycles: Definition, Computation, and Its Application",
booktitle="Computational Topology in Image Context",
year="2019",
publisher="Springer International Publishing",
address="Cham",
pages="123--136",
isbn="978-3-030-10828-1"
}

@article{CARO_Dominating_Number,
title = {Directed domination in oriented graphs},
journal = {Discrete Applied Mathematics},
volume = {160},
number = {7},
pages = {1053-1063},
year = {2012},
issn = {0166-218X},
doi = {https://doi.org/10.1016/j.dam.2011.12.027},
url = {https://www.sciencedirect.com/science/article/pii/S0166218X12000054},
author = {Yair Caro and Michael A. Henning}
}

@article{Nicolau2011,
  title={Topology based data analysis identifies a subgroup of breast cancers with a unique mutational profile and excellent survival},
  author={Monica Nicolau and Arnold J. Levine and Gunnar E. Carlsson},
  journal={Proceedings of the National Academy of Sciences},
  year={2011},
  volume={108},
  pages={7265 - 7270},
  url={https://api.semanticscholar.org/CorpusID:11745807}
}

@article{Giusti2016,
  title={Two’s company, three (or more) is a simplex},
  author={Chad Giusti and Robert Ghrist and Danielle S. Bassett},
  journal={Journal of Computational Neuroscience},
  year={2016},
  volume={41},
  pages={1 - 14},
  url={https://api.semanticscholar.org/CorpusID:6883428}
}

@misc{DONUT,
  author = {Giunti, Barbara and Lazovskis, J{\=a}nis and Rieck, Bastian},
  title  = {
    {DONUT}: {D}atabase of {O}riginal \& {N}on-{T}heoretical {U}ses of {T}opology
  },
  note   = {\url{https://donut.topology.rocks}},
  year   = {2022},
  key    = {DONUT},
}

@article{Carlsson2007,
  title={On the Local Behavior of Spaces of Natural Images},
  author={Gunnar E. Carlsson and Tigran Ishkhanov and Vin de Silva and Afra Zomorodian},
  journal={International Journal of Computer Vision},
  year={2007},
  volume={76},
  pages={1-12},
  url={https://api.semanticscholar.org/CorpusID:207252002}
}

@article{giotto,
  author  = {Guillaume Tauzin and Umberto Lupo and Lewis Tunstall and Julian Burella PÃ©rez and Matteo Caorsi and Anibal M. Medina-Mardones and Alberto Dassatti and Kathryn Hess},
  title   = {giotto-tda: : A Topological Data Analysis Toolkit for Machine Learning and Data Exploration},
  journal = {Journal of Machine Learning Research},
  year    = {2021},
  volume  = {22},
  number  = {39},
  pages   = {1--6},
  url     = {http://jmlr.org/papers/v22/20-325.html}
}

@inproceedings{javaplex,
author = {Tausz, Andrew and Vejdemo-Johansson, Mikael},
year = {2011},
month = {01},
pages = {},
title = {javaPlex: A Research Software Package for Persistent (Co)Homology},
isbn = {978-3-662-44198-5},
doi = {10.1007/978-3-662-44199-2_23}
}

@article{Fasy2014,
  title={Introduction to the R package TDA},
  author={Brittany Terese Fasy and Jisu Kim and Fabrizio Lecci and Cl{\'e}ment Maria},
  journal={ArXiv},
  year={2014},
  volume={abs/1411.1830},
  url={https://api.semanticscholar.org/CorpusID:11279296}
}

@article{Hannachi2017,
author = {Hannachi, Abdel. and Straus, David M. and Franzke, Christian L. E. and Corti, Susanna and Woollings, Tim},
title = {Low-frequency nonlinearity and regime behavior in the Northern Hemisphere extratropical atmosphere},
journal = {Reviews of Geophysics},
volume = {55},
number = {1},
pages = {199-234},
doi = {https://doi.org/10.1002/2015RG000509},
url = {https://agupubs.onlinelibrary.wiley.com/doi/abs/10.1002/2015RG000509},
eprint = {https://agupubs.onlinelibrary.wiley.com/doi/pdf/10.1002/2015RG000509},
year={2017}}

@article{charney1979multiple,
  title={Multiple flow equilibria in the atmosphere and blocking},
  author={Charney, Jule G and DeVore, John G},
  journal={Journal of the atmospheric sciences},
  volume={36},
  number={7},
  pages={1205--1216},
  year={1979}
}

@article{Bauer_Hien_Junge_Mischaikow_2025, title={Cycling signatures: Identifying cycling motions in time series using algebraic topology}, volume={12}, rights={http://creativecommons.org/licenses/by/3.0/}, ISSN={2158-2491}, DOI={10.3934/jcd.2025007}, abstractNote={Recurrence is a fundamental characteristic of dynamical systems with complicated behavior. Understanding the inner structure of recurrence is challenging, especially if the system has many degrees of freedom and is subject to noise. We develop algebraic topological notions for identifying and classifying elementary recurrent motions – called cycling – and the transitions between those. Statistics on these cycling motions can be computed from sampled trajectories (time series data), providing coarse global information on the structure of the recurrent behavior. We demonstrate this through three examples; in particular, we identify and analyze six cycling motions in a four-dimensional system with a hyperchaotic attractor. We see this as a promising approach to reveal coarse-grained dynamical information on high-dimensional systems.}, number={4}, journal={Journal of Computational Dynamics}, publisher={Journal of Computational Dynamics}, author={Bauer, Ulrich and Hien, David and Junge, Oliver and Mischaikow, Konstantin}, year={2025}, month=oct, pages={554–595}, language={en} }

@misc{chaplin_notion_2024, title={A notion of homotopy for directed graphs and their flag complexes}, journal={arxiv}, url={http://arxiv.org/abs/2411.04572}, DOI={10.48550/arXiv.2411.04572}, abstractNote={Directed graphs can be studied by their associated directed flag complex. The homology of this complex has been successful in applications as a topological invariant for digraphs. Through comparison with path homology theory, we derive a homotopy-like equivalence relation on digraph maps such that equivalent maps induce identical maps on the homology of the directed flag complex. Thus, we obtain an equivalence relation on digraphs such that equivalent digraphs have directed flag complexes with isomorphic homology. With the help of these relations, we can prove a generic stability theorem for the persistent homology of the directed flag complex of filtered digraphs. In particular, we show that the persistent homology of the directed flag complex of the shortest-path filtration of a weighted directed acyclic graph is stable to edge subdivision. In contrast, we also discuss some important instabilities that are not present in persistent path homology. We also derive similar equivalence relations for ordered simplicial complexes at large. Since such complexes can alternatively be viewed as simplicial sets, we verify that these two perspectives yield identical relations.}, note={arXiv:2411.04572 [math]}, number={arXiv:2411.04572}, publisher={arXiv}, author={Chaplin, Thomas and Harrington, Heather A. and Tillmann, Ulrike}, year={2024}, month=nov, language={en} }

@inproceedings{Chowdhury_Memoli_2016, title={Persistent homology of directed networks}, url={https://ieeexplore.ieee.org/document/7868997/}, DOI={10.1109/ACSSC.2016.7868997}, abstractNote={While persistent homology has been successfully used to provide topological summaries of point cloud data, the question of computing persistent homology of graphs or networks remains unclear. In particular, the existing literature does not provide a treatment of persistent homology for directed networks that is sensitive to asymmetry. We study a method for constructing simplicial complexes from weighted, directed networks that captures directionality information, and we are able to prove that the persistent homology of such complexes is stable with respect to a certain notion of network distance. We illustrate our construction on a database of simulated hippocampal networks.}, booktitle={2016 50th Asilomar Conference on Signals, Systems and Computers}, author={Chowdhury, Samir and Memoli, Facundo}, year={2016}, month=nov, pages={77–81} }

@article{Gilmore_1998, title={Topological analysis of chaotic dynamical systems}, volume={70}, DOI={10.1103/RevModPhys.70.1455}, abstractNote={Topological methods have recently been developed for the analysis of dissipative dynamical systems that operate in the chaotic regime. They were originally developed for three-dimensional dissipative dynamical systems, but they are applicable to all “low-dimensional” dynamical systems. These are systems for which the flow rapidly relaxes to a three-dimensional subspace of phase space. Equivalently, the associated attractor has Lyapunov dimension. Topological methods supplement methods previously developed to determine the values of metric and dynamical invariants. However, topological methods possess three additional features: they describe how to model the dynamics; they allow validation of the models so developed; and the topological invariants are robust under changes in control-parameter values. The topological-analysis procedure depends on identifying the stretching and squeezing mechanisms that act to create a strange attractor and organize all the unstable periodic orbits in this attractor in a unique way. The stretching and squeezing mechanisms are represented by a caricature, a branched manifold, which is also called a template or a knot holder. This turns out to be a version of the dynamical system in the limit of infinite dissipation. This topological structure is identified by a set of integer invariants. One of the truly remarkable results of the topological-analysis procedure is that these integer invariants can be extracted from a chaotic time series. Furthermore, self-consistency checks can be used to confirm the integer values. These integers can be used to determine whether or not two dynamical systems are equivalent; in particular, they can determine whether a model developed from time-series data is an accurate representation of a physical system. Conversely, these integers can be used to provide a model for the dynamical mechanisms that generate chaotic data. In fact, the author has constructed a doubly discrete classification of strange attractors. The underlying branched manifold provides one discrete classification. Each branched manifold has an “unfolding” or perturbation in which some subset of orbits is removed. The remaining orbits are determined by a basis set of orbits that forces the presence of all remaining orbits. Branched manifolds and basis sets of orbits provide this doubly discrete classification of strange attractors. In this review the author describes the steps that have been developed to implement the topological-analysis procedure. In addition, the author illustrates how to apply this procedure by carrying out the analysis of several experimental data sets. The results obtained for several other experimental time series that exhibit chaotic behavior are also described.}, number={4}, journal={Reviews of Modern Physics}, publisher={American Physical Society}, author={Gilmore, Robert}, year={1998}, month=oct, pages={1455–1529} }

@article{Hussain_Shah_Rafia_Fatima_Huerta-Cuellar_Garcia-Lopez_Mata_amirez_Jaimes-Reategui_2025, address={Woodbury, N.Y.}, title={Topological data analysis approach to time series and shape analysis of dynamical system}, volume={35}, ISSN={1089-7682}, DOI={10.1063/5.0268340}, abstractNote={In a dynamical system, the time series and phase space play vital roles, and we applied topological data analysis to these characteristics. More precisely, we consider the well-known Rössler-like attractor to analyze time-series and phase-space images. We studied persistent homology representations directly from the time series of the system to obtain point cloud data. In our approach, we converted the time series to a point cloud and computed homology using the Rips complex. This enabled us to measure the topological features of the system behavior. We also applied cubical homology to phase-space images for the first time, a novel contribution that represents an image-based approach to analyze phase portraits. This article provides a review of the topological data analysis of time series using examples with the Python function. Finally, we computed topological machine learning features, such as persistent landscapes, persistence images, and Betti curves. These features enable the automated analysis and classification of dynamical behaviors and, hence, connect topological data analysis with machine learning. This study is new in that it presents a comprehensive topological data analysis pipeline tailored to dynamical systems. The goal is to make these approaches accessible and usable for nonlinear dynamics to analyze their temporal series and phase portraits.}, number={6}, journal={Chaos}, author={Hussain Shah, W. and Rafia Fatima, S. and Huerta-Cuellar, G. and García-López, J. H. and Mata Ramirez, C. G. and Jaimes-Reátegui, R.}, year={2025}, month=jun, pages={063129}, language={eng} }

@article{Kappe_Bottinger_Leitte_2022, title={Topology-based feature analysis of scalar field ensembles: An application to climate (change) analysis}, volume={104}, ISSN={0097-8493}, DOI={10.1016/j.cag.2022.03.004}, abstractNote={This paper presents a framework for topological feature analysis in time-dependent climate ensembles. Important climate indices such as the El Nino Southern Oscillation Index (ENSO) or the North Atlantic Oscillation Index (NAO) are usually derived by evaluating scalar fields at fixed locations or regions where these extremal values occur most frequently today. However, under climate change, dynamic circulation changes are likely to cause shifts in the intensity, frequency and location of underlying physical phenomena. In case of the NAO for instance, climatologists are interested in the position and intensity of the Icelandic Low and the Azores High as their interplay strongly influences the European climate, especially during the winter season. To robustly extract and track such highly variable features without a-priori region information, we present a topology-based method for dynamic extraction of such uncertain critical points on a global scale. The system additionally integrates techniques to visualize the variability within the ensemble and correlations between features. To demonstrate the utility of our VTK+TTK-based software framework, we explore a 150-year climate projection consisting of 100 ensemble members and particularly concentrate on sea level pressure fields.}, journal={Computers and Graphics}, author={Kappe, Christopher and Böttinger, Michael and Leitte, Heike}, year={2022}, month=may, pages={59–71} }

@article{Maletic_Zhao_Rajkovic_2016, title={Persistent topological features of dynamical systems}, volume={26}, ISSN={1054-1500}, url={https://pubs.aip.org/aip/cha/article/26/5/053105/280320/Persistent-topological-features-of-dynamical}, DOI={10.1063/1.4949472}, abstractNote={Inspired by an early work of Muldoon et al., Physica D 65, 1–16 (1993), we present a general method for constructing simplicial complex from observed time serie}, number={5}, journal={Chaos: An Interdisciplinary Journal of Nonlinear Science}, publisher={AIP Publishing}, author={Maletić, Slobodan and Zhao, Yi and Rajković, Milan}, year={2016}, month=may, language={en} }

@article{molkenthin_networks_2014, title={Networks from flows--from dynamics to topology}, volume={4}, ISSN={2045-2322}, DOI={10.1038/srep04119}, abstractNote={Complex network approaches have recently been applied to continuous spatial dynamical systems, like climate, successfully uncovering the system’s interaction structure. However the relationship between the underlying atmospheric or oceanic flow’s dynamics and the estimated network measures have remained largely unclear. We bridge this crucial gap in a bottom-up approach and define a continuous analytical analogue of Pearson correlation networks for advection-diffusion dynamics on a background flow. Analysing complex networks of prototypical flows and from time series data of the equatorial Pacific, we find that our analytical model reproduces the most salient features of these networks and thus provides a general foundation of climate networks. The relationships we obtain between velocity field and network measures show that line-like structures of high betweenness mark transition zones in the flow rather than, as previously thought, the propagation of dynamical information.}, journal={Scientific Reports}, author={Molkenthin, Nora and Rehfeld, Kira and Marwan, Norbert and Kurths, Jürgen}, year={2014}, month=feb, pages={4119}, language={eng} }

@misc{Munoz_Munch_Khasawneh_2025, journal={arxiv}, title={The Walk-Length Filtration for Persistent Homology on Weighted Directed Graphs}, url={http://arxiv.org/abs/2506.22263}, DOI={10.48550/arXiv.2506.22263}, abstractNote={Directed graphs arise in many applications where computing persistent homology helps to encode the shape and structure of the input information. However, there are only a few ways to turn the directed graph information into an undirected simplicial complex filtration required by the standard persistent homology framework. In this paper, we present a new filtration constructed from a directed graph, called the walk-length filtration. This filtration mirrors the behavior of small walks visiting certain collections of vertices in the directed graph. We show that, while the persistence is not stable under the usual $L_infty$-style network distance, a generalized $L_1$-style distance is, indeed, stable. We further provide an algorithm for its computation, and investigate the behavior of this filtration in examples, including cycle networks and synthetic hippocampal networks with a focus on comparison to the often used Dowker filtration.}, note={arXiv:2506.22263 [cs]}, number={arXiv:2506.22263}, publisher={arXiv}, author={Muñoz, David E. and Munch, Elizabeth and Khasawneh, Firas A.}, year={2025}, month=jun }

@article{Peek_Pritam_Skerritt_Chalup_2025, title={Time series analysis of spiking neural systems via transfer entropy and directed persistent homology}, volume={684}, ISSN={0925-2312}, DOI={https://doi.org/10.1016/j.neucom.2026.133541}, abstractNote={We present a topological framework for analyzing neural time series that integrates Transfer Entropy (TE) with directed Persistent Homology (PH) to characterize information flow in spiking neural systems. TE quantifies directional influence between neurons, producing weighted, directed graphs that reflect dynamic interactions. These graphs are then analyzed using PH, enabling assessment of topological complexity across multiple structural scales and dimensions. We apply this TE+PH pipeline to synthetic spiking networks trained on logic gate tasks, image-classification networks exposed to structured and perturbed inputs, and mouse cortical recordings annotated with behavioral events. Across all settings, the resulting topological signatures reveal distinctions in task complexity, stimulus structure, and behavioral regime. Higher-dimensional features become more prominent in complex or noisy conditions, reflecting interaction patterns that extend beyond pairwise connectivity. Our findings offer a principled approach to mapping directed information flow onto global organizational patterns in both artificial and biological neural systems. The framework is generalizable and interpretable, making it well suited for neural systems with time-resolved and binary spiking data.}, journal={Neurocomputing}, author={Peek, Dylan and Pritam, Siddharth and Skerritt, Matthew P. and Chalup, Stephan}, year={2026}, pages={133541} }

@article{Tanweer_Khasawneh_Munch_Tempelman_2024, title={A Topological Framework for Identifying Phenomenological Bifurcations in Stochastic Dynamical Systems}, volume={112}, ISSN={0924-090X, 1573-269X}, DOI={10.1007/s11071-024-09289-1}, abstractNote={Changes in the parameters of dynamical systems can cause the state of the system to shift between different qualitative regimes. These shifts, known as bifurcations, are critical to study as they can indicate when the system is about to undergo harmful changes in its behavior. In stochastic dynamical systems, there is particular interest in P-type (phenomenological) bifurcations, which can include transitions from a mono-stable state to multi-stable states, the appearance of stochastic limit cycles, and other features in the probability density function (PDF) of the system’s state. Current practices are limited to systems with small state spaces, cannot detect all possible behaviours of the PDFs, and mandate human intervention for visually identifying the change in the PDF. In contrast, this study presents a new approach based on Topological Data Analysis (TDA) that uses superlevel persistence to mathematically quantify P-type bifurcations in stochastic systems through a "homological bifurcation plot’’ -- which shows the changing ranks of 0th and 1st homology groups. Using these plots, we demonstrate the successful detection of P-bifurcations on the stochastic Duffing, Raleigh-Vander Pol and Quintic Oscillators given their analytical PDFs, and elaborate on how to generate an estimated homological bifurcation plot given a kernel density estimate (KDE) of these systems by employing a tool for finding topological consistency between PDFs and KDEs.}, note={arXiv:2305.03118 [math]}, number={6}, journal={Nonlinear Dynamics}, author={Tanweer, Sunia and Khasawneh, Firas A. and Munch, Elizabeth and Tempelman, Joshua R.}, year={2024}, month=mar, pages={4687–4703} }

@article{Ye_Jiang_Jiang_Li_2023, title={Stable distance of persistent homology for dynamic graph comparison}, volume={278}, ISSN={0950-7051}, DOI={10.1016/j.knosys.2023.110855}, abstractNote={Persistent homology theory provides approaches for analyzing topological features, which is now widely applied in graph comparison on social networks, biological networks, and co-location networks. These approaches utilize filtration techniques to extract the topological properties of a graph and construct vectorizations that represent these properties for further computation. However, most existing methods are designed for static scenarios and are unsuitable for the time-varying structure in realistic dynamic graphs. In this paper, we propose the Stable Distance of Persistent Homology (SDPH) to compare and quantify the differences in the topological properties of dynamic graphs. In detail, we design Dynamic Dowker Filtration (DDF) to map dynamic graph to a persistent complex based on the ɛ-interleaved theory, which enables us to trace the structure holes formed by the accumulation of temporal edge via computing the persistent homology. DDF exhibits stability and duality, inducing a time-structure triangle inequality. Based on this inequality, we finally construct Time-interlevel Kernel (TIK) for vectorizing the extracted topological features with an inner product. We conduct the graph clustering and classification experiments on synthetic and real-world datasets. Experimental results show that the proposed SDPH outperforms the baseline methods in these tasks and validate that the proposed SDPH can effectively measure the topological difference of dynamic graphs. Through SDPH, we would provide insight and inspiration on how to apply persistent homology theory to dynamic graph analysis.}, journal={Knowledge-Based Systems}, author={Ye, Dongsheng and Jiang, Hao and Jiang, Ying and Li, Hao}, year={2023}, month=oct, pages={110855} }

@article{Pun_Lee_Xia_2022, title={Persistent-homology-based machine learning: a survey and a comparative study}, volume={55}, ISSN={1573-7462}, DOI={10.1007/s10462-022-10146-z}, abstractNote={A suitable feature representation that can both preserve the data intrinsic information and reduce data complexity and dimensionality is key to the performance of machine learning models. Deeply rooted in algebraic topology, persistent homology (PH) provides a delicate balance between data simplification and intrinsic structure characterization, and has been applied to various areas successfully. However, the combination of PH and machine learning has been hindered greatly by three challenges, namely topological representation of data, PH-based distance measurements or metrics, and PH-based feature representation. With the development of topological data analysis, progresses have been made on all these three problems, but widely scattered in different literatures. In this paper, we provide a systematical review of PH and PH-based supervised and unsupervised models from a computational perspective. Our emphasizes are the recent development of mathematical models and tools, including PH software and PH-based functions, feature representations, kernels, and similarity models. Essentially, this paper can work as a roadmap for the practical application of PH-based machine learning tools. Further, we compare between two types of simplicial complexes (alpha and Vietrois-Rips complexes), two types of feature extractions (barcode statistics and binned features), and three types of machine learning models (support vector machines, tree-based models, and neural networks), and investigate their impacts on the protein secondary structure classification.}, number={7}, journal={Artificial Intelligence Review}, author={Pun, Chi Seng and Lee, Si Xian and Xia, Kelin}, year={2022}, month=oct, pages={5169–5213}, language={en} }

@article{Adler_1996, title={Symbolic Dynamics and Markov Partitions}, volume={35}, DOI={10.1090/S0273-0979-98-00737-X}, abstractNote={The decimal expansion real numbers, familiar to us all, has a dramatic generalization to representation of dynamical system orbits by symbolic sequences. The natural way to associate a symbolic sequence with an orbit is to track its history through a partition. But in order to get a useful symbolism, one needs to construct a partition with special properties. In this work we develop a general theory of representing dynamical systems by symbolic systems by means of so-called Markov partitions. We apply the results to one of the more tractable examples: namely hyperbolic automorphisms of the two dimensional torus. While there are some results in higher dimensions, this area remains a fertile one for research.}, journal={Bulletin of the American Mathematical Society}, author={Adler, Roy}, year={1996}, month=jul }

\end{document}